\documentclass[10pt,a4paper,reqno]{amsart}
\usepackage{amssymb,amsmath,amsthm}
\usepackage{xcolor}
\usepackage[all,cmtip]{xy}
\theoremstyle{plain}
\newtheorem{thm}{Theorem}[section]
\newtheorem{prop}[thm]{Proposition}
\newtheorem{lem}[thm]{Lemma}
\newtheorem{cor}[thm]{Corollary}
\theoremstyle{definition}
\newtheorem{defn}[thm]{Definition}
\newtheorem{exmpl}[thm]{Example}
\newtheorem{rem}[thm]{Remark}
\newtheorem{question}[thm]{Question}

\makeatletter
\def\longrightleftarrows{\mathrel{
	\mathop{\vcenter{
	\offinterlineskip
	\hbox to 0.6truecm{\rightarrowfill}%
	\hbox to 0.6truecm{\leftarrowfill}}}%
	}}
\newcommand{\reflective}[1][\empty]{
	\mathrel{\mathop{\vcenter{%
	\offinterlineskip
	\hbox to 0.6truecm{\rightarrowfill}%
	\hbox to 0.6truecm{$\vphantom\gets\smash{\longleftarrow\joinrel\rhook}$}}}%
	\limits^{#1}%
	}}
\newcommand{\coreflective}[1][\empty]{%
	\mathrel{\mathop{\vcenter{%
	\offinterlineskip
	\hbox to 0.6truecm{$\vphantom\to\smash{\lhook\joinrel\longrightarrow}$}%
	\hbox to 0.6truecm{\leftarrowfill}}}\limits_{#1}%
	}}
\newcommand{\doublewidetilde}[1]{{%
  \mathpalette\double@widetilde{#1}%
}}
\newcommand{\double@widetilde}[2]{%
  \sbox\z@{$\m@th#1\widetilde{#2}$}%
  \ht\z@=.9\ht\z@
  \widetilde{\box\z@}%
}

\newdimen\ex@
\def\nointerlineskip{\prevdepth-\@m\p@}
\def\@projlim{%
		\mathop{\vtop{\ialign{##\crcr
		\hfil\rm lim\hfil\crcr\noalign{\nointerlineskip}\leftarrowfill\crcr
		\noalign{\nointerlineskip\kern-\ex@}\crcr}}}
}
\def\@injlim{%
		\mathop{\vtop{\ialign{##\crcr
		\hfil\rm lim\hfil\crcr\noalign{\nointerlineskip}\rightarrowfill\crcr
		\noalign{\nointerlineskip\kern-\ex@}\crcr}}}
}
\def\@holim{%
		\mathop{\vtop{\ialign{##\crcr
		\hfil\rm holim\hfil\crcr\noalign{\nointerlineskip}\leftarrowfill\crcr
		\noalign{\nointerlineskip\kern-\ex@}\crcr}}}
}
\def\@hocolim{%
		\mathop{\vtop{\ialign{##\crcr
		\hfil\rm holim\hfil\crcr\noalign{\nointerlineskip}\rightarrowfill\crcr
		\noalign{\nointerlineskip\kern-\ex@}\crcr}}}
}
\def\varprojlim{\mathop{\@projlim}}
\def\varinjlim{\mathop{\@injlim}}
\def\varholim{\mathop{\@holim}}
\def\varhocolim{\mathop{\@hocolim}}

\def\subsection{\@startsection{subsection}{2}%
  \z@{.5\linespacing\@plus.7\linespacing}{-.5em}%
  {\bfseries\mathversion{bold}}}
\makeatother

\let\cal=\mathcal
\def\N{\mathbb{N}}
\def\Z{\mathbb{Z}}

\newcommand{\pDer}[1][D]{\mathbb{#1}}
\def\id{\mathop{\rm id}\nolimits}
\def\ker{\mathop{\rm Ker}\nolimits}
\def\coker{\mathop{\rm Coker}\nolimits}

\let\im=\Im\relax
\def\Spec{\mathop{\rm Spec}\nolimits}
\def\Hom{\mathop{\rm Hom}\nolimits}

\def\Ext{\mathop{\rm Ext}\nolimits}
\def\Ab{\mathord{\rm Ab}}

\def\lMod{{\rm\mathchar`-Mod}}

\def\lmod{{\rm\mathchar`-mod}}
\def\Ch{\mathord{{\cal C}\kern-1pt{\it h}}}
\def\Ho{{\cal K}}
\def\De{{\cal D}}
\def\fg{\mathop{\rm fg}\nolimits}
\def\fp{\mathop{\rm fp}\nolimits}
\def\TF{\mathord{\mathcal{T\kern-1pt F}}}
\def\TFT{\mathord{\mathcal{T\kern-1pt FT}}}
\def\HTF{\mathord{\mathcal{HT\kern-1pt F}}}
\newcommand{\adjust}[1][H]{\mathop{{\vphantom\varinjlim}^{\!({\cal#1})}}}
\def\leqdef{\mathrel{\mathrel{\mathop:}=}}
\def\reqdef{\mathrel{=\mathrel{\mathop:}}}

\let\Phi=\varPhi
\let\Sigma=\varSigma

\title[Local Coherence of Thomason Hearts]%
{Local Coherence of Hearts associated with Thomason Filtrations}%

\author[L.~Martini]{Lorenzo Martini}
\address[Lorenzo Martini]{%
Dipartimento di Matematica\\
Universit\`a di Trento\\
Via Sommarive 14, 38123 Povo (Trento), Italy}
\email{lorenzo.martini-2@unitn.it}
\address{%
Dipartimento di Informatica -- Settore di Matematica\\
Universit\`a di Verona\\
Strada le Grazie 15 -- Ca' Vignal, I-37134 Verona, Italy}
\email{lorenzo.martini@univr.it}

\author[C.\,E.~Parra]{Carlos E.~Parra}
\address[Carlos E.~Parra]{%
Universidad Austral de Chile\\
Instituto de Ciencias F\'isicas y Ma\-te\-m\'a\-ti\-cas\\
Edificio Emilio Pugin, Campus Isla Teja, 5090000 Valdivia, Chile}
\email{carlos.parra@uach.cl}%

\subjclass[2010]{13D30, 18E15, 18E30, 18A35, 18C35}
\keywords{%
TTF triple,
t-structure,
compactly generated,
Thomason filtration,
locally coherent,
Grothendieck category}
\thanks{The first named author is supported by a grant from University of Trento}
\thanks{The second named author is supported by grant ANID+Fondecyt/Regular+1200090}

\begin{document}

\begin{abstract}
Any Thomason filtration of a commutative ring yields (at least) two t-structures in the derived category of the ring, one of which is compactly generated \cite{Hrb20,HHZ21}. We study the hearts of these two t-structures and prove that they coincide in case of a weakly bounded below filtration. Prompted by \cite{SS20}, in which it is proved that the heart of a compactly generated t-structure in a triangulated category with coproduct is a locally finitely presented Grothendieck category, we study the local coherence of the hearts associated with a weakly bounded below Thomason filtration, achieving a useful recursive characterisation in case of a finite length filtration. Low length cases involve hereditary torsion classes of finite type of the ring, and even their Happel--Reiten--Smal\o\ hearts; in these cases, the relevant characterisations are given by few module-theoretic conditions.
\end{abstract}
\maketitle

\section*{Introduction}%
The main way to study an arbitrary abelian category is to provide good enough categorical correspondences to a category of modules over an arbitrary ring or, if this is not manageable, to define directly on the abelian category some homological properties that generalise the corresponding module-theoretic ones. Possibly, categories of modules over an associative ring are the ``nicest'' abelian categories one can work with, since they are Grothendieck categories with a finitely generated projective generator and carrying an additional finiteness condition, namely they are locally finitely presented. However, there are other fundamental homological conditions that are not shared by all the module categories. In this sense, we are interested in providing necessary and sufficient conditions for certain locally finitely presented Grothendieck categories to be locally coherent. Such finiteness condition if formulated just by miming the behaviour of the category of modules over a coherent ring, namely by asking for the finitely presented objects to form an abelian category. Locally coherent Grothendieck categories constitute the abelian setting in which it is possible to perform a fruitful purity theory strictly related to the general purity theory for triangulated categories (see \cite{Kra00}; in turn, these theories are a generalisation of the classical one for modules). Purity is a central topic in representation theory and it is in fact interwoven with other powerful homological and categorical tools, such as localisation, tilting theory, cotorsion theory, and derivators (see \cite{AHMV17,Lak18,SSV17,SS20}).

The Grothendieck categories we want to examine come from the world of triangulated categories, more precisely they are the hearts of certain t-structures (the heart of any t-structure is an abelian category, see~\cite{BBD82}). The problem of detecting and characterising the Grothendieck and other homological conditions on the hearts of t-structures in arbitrary triangulated categories is widely open. Nonetheless, two families of t-structures have been intensively studied in the literature, namely the Happel--Reiten--Smal\o\ and the compactly generated ones.

HRS~t-structures were introduced in \cite{HRS96} and they are defined in the derived category of an abelian category by means of a torsion pair of the latter. Many authors (see e.g.\ \cite{CGM07,CMT11,PS15,PS16,SSV17}) have investigated on the module theoretic properties of their hearts, in fact establishing necessary and sufficient conditions for these latter to be equivalent to module categories. When the underlying abelian category is Grothendieck, then the Grothendieck condition for the heart has been completely characterised in \cite{PS16a}: it occurs if and only if the torsion pair is of finite type. On the other hand, crucial results concerning the finiteness conditions have been achieved e.g.~in \cite{Sao17,PSV19}. An exhaustive survey devoted to the study of HRS~hearts and related topics is \cite{PS20}.

Compactly generated t-structures are defined in any triangulated category with coproducts. Very recent works \cite{SSV17,Bon19,SS20} show that their hearts are locally finitely presented Grothendieck categories; thus, it is natural to ask when they are locally coherent.

In the present paper we concentrate on t-structures and hearts in the derived category of a commutative ring. Here, the t-structures we are interested in have a useful homological description by means of classes associated with the ring, in particular with the hereditary torsion classes of finite type of its module category; indeed, these latter are parametrised by the so-called Thomason subsets of the spectrum of the ring (\cite[Theorem~2.2]{GP08}). Among these t-structures, the compactly generated ones are in bijection with the so-called Thomason filtrations of the spectrum, i.e.~decreasing maps of $(\Z,\subseteq)$ into the parts of the spectrum, such that each image is a Thomason subset. This fundamental result goes back to \cite{AJS10} in the case of a noetherian ring (see Theorem~3.10 therein): the aisles of such t-structures consist of the complexes whose cohomologies are supported along the corresponding Thomason filtration. In \cite{Hrb20} the author extends such result to the commutative case, i.e.~dropping the chain condition on the ring, losing however the previous cohomological description of the aisle. Such description still yields an aisle of a t-structure (see \cite[Lemma~3.6]{HHZ21}), though it is not known whether this latter needs to be compactly generated. Consequently, a Thomason filtration of a commutative ring yields at least two t-structures, one of which ---that one in the sense of Hrbek--- is compactly generated, so that has a locally finitely presented Grothendieck heart; we will call the other one the {\it Alonso--Jerem\'\i as--Saor\'\i n~t-structure\/} induced by the Thomason filtration.

In details, we study the local coherence of the hearts of the two mentioned t-structures associated with a Thomason filtration, achieving a characterisation for filtrations of finite length (Theorem~\ref{t:recursive}). Such result is in fact a very special case of a more general characterisation, i.e.~that of the local coherence of a Grothendieck category endowed with a TTF~triple of finite type (Theorem~\ref{t:lc-TTF_finite_type}); furthermore, it is formulated recursively, meaning that at any fixed length the local coherence of the heart associated with a Thomason filtration is encoded by that of the hearts of filtrations with lower length naturally associated with the given one.

The paper is organised as follows.

Section~\ref{s:preliminaries} contains the notations and all the preliminary definitions and results we need, concerning abelian and Grothendieck categories (in particular their torsion theories), triangulated categories (in particular their t-structures), and a brief survey on prederivators. We introduce the notion of quasi locally coherent additive category, which is a generalisation of the notion of local coherence for abelian categories.

In Section~\ref{s:a_criterion_for_the_local_coherence} we state and prove Theorem~\ref{t:lc-TTF_finite_type}, namely the result we want to specialise in the body of the paper, concerning a characterisation of the local coherence of a Grothendieck category equipped with a TTF~triple of finite type.

Section~\ref{s:basics_on_thomason_filtrations_and_hearts} contains the definition of Thomason filtration of a commutative ring, some relations between the two t-structures and hearts it induces, and a first main result concerning the AJS~heart, namely that in case of a weakly bounded below Thomason filtration, it is always a locally finitely presented Grothendieck category, as well as the other heart is (see Proposition~\ref{p:heart_is_bounded}).

In Section~\ref{s:some_useful_results} we continue giving properties relating the hearts and certain subcategories involved in the study of the finitely presented complexes of the hearts, which will be useful in the sequel. Above all, we can realise an arbitrary hereditary torsion class of finite type, together with the HRS~heart it gives rise, as the heart of a suitable Thomason filtration of finite length (respectively~$0$ and~$1$; see Subsection~\ref{ss:a_crucial_example} and Example~\ref{e:HRS-heart&filtration}). Our main result is Theorem~\ref{t:comm_non-coherent}, in which we completely characterise the local coherence of the former category; that is, the local coherence of an arbitrary hereditary torsion class of finite type of the ring.

Sections~\ref{s:arbitrary_thomason_filtrations} and~\ref{s:thomason_filtrations_of_finite_length} are the central part of the paper: they provide the machinery to apply Theorem~\ref{t:lc-TTF_finite_type} to the case of the hearts of filtrations of finite length. 

Section~\ref{s:arbitrary_thomason_filtrations} is devoted to detecting within the hearts of a (weakly bounded below) Thomason filtration a TTF~triple (of finite type), in order to let Theorem~\ref{t:lc-TTF_finite_type} apply. For this task, weakly bounded below Thomason filtrations have a prominent role. Indeed, on the one hand their two induced t-structures have the same hearts, as proved in Theorem~\ref{t:when_hearts_coincide}. On the other hand, they allow to identity for an arbitrary Thomason filtration the TTF~classes we are looking for: in Theorem~\ref{t:H_*-TTF_finite_type} we show that such classes are indeed the hearts of weakly bounded below Thomason filtrations naturally associated with the given one. In particular, these hearts are all locally coherent Grothendieck categories in case the given heart is so (\cite[Theorem~2.16]{Her97} and \cite[Theorem~2.6]{Kra97}).

Section~\ref{s:thomason_filtrations_of_finite_length} contains the main results concerning the Thomason filtrations of finite length, hence those whose two corresponding hearts coincide. In view of Theorem~\ref{t:H_*-TTF_finite_type}, such length is taken in the vein of providing a recursive argument for the characterisation of the local coherence of the heart, namely by taking into account the local coherence of the TTF classes of finite type detected previously. Theorem~\ref{t:recursive}, the main result of the section, identifies the local coherence in such a recursive way, by means of five conditions. These conditions are the most eligible ones, in the sense that for the crucial cases of length~$0,1,2$, almost all translate into module-theoretic properties, as proved in Corollaries~\ref{c:recursive-length1} and~\ref{c:recursive-length2}.

Eventually, Section~\ref{s:applications} contains some applications of Corollary~\ref{c:recursive-length1}: the first part is focused on the case of the HRS~hearts (as discussed in Section~4) while the very last subsection is devoted to specialise the previous main results over a commutative noetherian ring. Corollary~\ref{c:HRS_heart_LC_characterisation} is the direct application of Corollary~\ref{c:recursive-length1} when the involved Thomason filtration gives the HRS~heart of a hereditary torsion pair of finite type; besides, it also allows us to exhibit an example of a quasi locally coherent additive category that is not locally coherent in the usual sense (see Example~\ref{e:qlc-non-lc}). Other further applications of such Corollary are obtained by adding conditions either on the ring, i.e.~when it is coherent, or on the underlying torsion pair, i.e.~when it is stable. In any case, it furtherly enlights into more handleable module-theoretic conditions. Finally, in case the ring is noetherian, we recover for Thomason filtrations of finite length the useful result \cite[Theorem~6.3]{Sao17} by Saor\'\i n  concerning the local coherence of the hearts of the so-called restrictable t-structures of the ring.

\section{Preliminaries}%
\label{s:preliminaries}%
We will refer to \cite{Pop73,Ste75} for the basics on abelian categories, and to \cite{Nee01,Mil} for what does concern triangulated categories. Throughout the present paper, $R$ will denote a commutative ring, while $R\lMod$, $\Ch(R)$, $\Ho(R)$ and $\De(R)$ will denote, respectively, its module category, its category of cochain complexes, its homotopy category, and its unbounded derived category. Given a preadditive category $\cal A$ and a set $\cal S$ of objects of $\cal A$, we will denote, for short,
\begin{align*}
	{\cal S}^{\bot_0} &\leqdef \{M\in{\cal A}\mid \Hom_{\cal A}(S,M)=0\ \forall S\in{\cal S}\}
		\reqdef\ker\Hom_{\cal A}({\cal S},-), \cr
	{}^{\bot_0}{\cal S} &\leqdef \{M\in{\cal A}\mid \Hom_{\cal A}(M,S)=0\ \forall S\in{\cal S}\}
		\reqdef\ker\Hom_{\cal A}(-,{\cal S}) \;.
\end{align*}
Moreover, in case $\cal A$ is an abelian category with coproducts, $\mathop{\rm Gen}{\cal S}$ ($\mathop{\rm gen}{\cal S}$) and $\mathop{\rm Add}{\cal S}$ ($\mathop{\rm add}{\cal S}$) will denote, respectively, the full subcategories formed by the objects admitting an epimorphism, resp.~a split epimorphism, originating in a coproduct of (finitely many) objects of $\cal S$. If $\mathop{\rm Gen}{\cal S}={\cal A}$, then $\cal S$ is said to be a {\it set of generators\/} for $\cal A$.

\subsection{Abelian categories}%
Let $\cal A$ be an additive category and let $I$ be a small category (i.e.~its objects form a set). A functor $F\colon I\to{\cal A}$ will be also called a {\it diagram of shape $I$ on\/ $\cal A$\/} and denoted by $(F_i)_{i\in I}$, where $F_i\leqdef F(i)$ for all $i\in I$. The category having such diagrams as objects and the natural transformations between them as morphisms will be denoted by ${\cal A}^I$. When any $I$-shaped diagram has colimit (resp.~limit) in $\cal A$, then $\cal A$ is said to admit $I$-colimits (resp.~$I$-limits). In this case, we have two adjoint pairs
$$
	\mathop{\rm colim}_{i\in I}:{\cal A}^I\longrightleftarrows {\cal A}:\Delta_I\qquad\hbox{and}
	\qquad \Delta_I:{\cal A}\longrightleftarrows {\cal A}^I:\lim_{i\in I}
$$
where $\Delta_I$ is the constant functor. If $\cal A$ admits $I$-colimits (resp.~$I$-limits) for every small category $I$, then $\cal A$ is said to be {\it co\-complete\/} (resp.~{\it complete\/}). A very important case occurs when $I$ is a directed poset, hence regarded as a small category in the usual way: the $I$-shaped diagrams are called {\it direct systems\/}, while the $I^{\rm op}$-shaped diagrams are called {\it inverse systems\/}, so that the corresponding $I$-colimit and $I$-limit functors are then called respectively the {\it direct limit\/} and the {\it inverse limit\/}, and denoted by
$$
	\varinjlim_{i\in I} F_i\quad\hbox{and}\quad
	\varprojlim_{\hidewidth i\in I^{\rm op}\hidewidth}F_i
$$
for every $F=(F_i)_{i\in I}\in{\cal A}^I$.

Following the celebrated T\^ohoku paper \cite{Gro57} by Grothendieck, an abelian category $\cal A$ is said to be

\begin{enumerate}
\item[AB-3] if it admits coproducts or, equivalently, if it is cocomplete;

\item[AB-4] if it is AB-3 and the coproduct functor $\smash[b]{\coprod\limits_{i\in I}}\colon{\cal A}^I\to{\cal A}$ is exact for every small category $I$;

\item[AB-5] if it is AB-3 and the direct colimit functor $\smash[b]{\varinjlim\limits_{i\in I}}\colon{\cal A}^I\to{\cal A}$ is exact for every directed poset $I$.
\end{enumerate}

The most important and studied AB-5 abelian categories are the {\it Grothendieck categories\/}, namely those having a set of generators (equivalently, a generator). For instance, it is well-known that Grothendieck categories are also complete, provide injective envelopes for their objects and have an injective cogenerator (\cite[Corollary~X.4.4]{Ste75}, \cite{Gro57}).

We want to study certain Grothendieck categories having some additional finiteness conditions which we now recall explicitly in a more general context (cf.~\cite{CB94,Kra97}). Let $\cal A$ be an additive category with direct limits (that is, $\cal A$ admits $I$-colimits for every directed poset $I$; in particular, $\cal A$ has cokernels), then

\begin{itemize}
\item an object $A\in{\cal A}$ is called {\it finitely generated\/} if for every direct system of monomorphisms $(M_i)_{i\in I}$ of $\cal A$ (i.e.~each connection map $M_i\to M_j$ is a monomorphism), the natural group homomorphism
$$
	\varinjlim_{i\in I}\Hom_{\cal A}(A,M_i)\longrightarrow\Hom_{\cal A}(A,\varinjlim_{i\in I}M_i)
$$
is bijective.

\item An object $B\in{\cal A}$ is called {\it finitely presented\/} if for every direct system $(M_i)_{i\in I}$ of $\cal A$, the natural group homomorphism
$$
	\varinjlim_{i\in I}\Hom_{\cal A}(B,M_i)\longrightarrow\Hom_{\cal A}(B,\varinjlim_{i\in I}M_i)
$$
is bijective; that is, the functor $\Hom_{\cal A}(B,-)\colon{\cal A}\to\Ab$ {\it commutes with direct limits\/}.
\end{itemize}

\noindent These definitions in turn provide the aforementioned finiteness conditions for an additive category with kernels and direct limits: such $\cal A$ will be called

\begin{itemize}

\item {\it locally finitely presented\/} if the full subcategory $\fp({\cal A})$ of the finitely presented objects is skeletally small and ${\cal A}=\varinjlim\fp({\cal A})$, meaning that each object of $\cal A$ is isomorphic to a direct limit of a direct system of $\fp({\cal A})$;

\item {\it quasi locally coherent\/} if it is locally finitely presented and $\fp({\cal A})$ is closed under taking kernels of $\cal A$.
\end{itemize}

\noindent When $\cal A$ is an abelian category, then it is locally finitely presented if and only if it is a Grothendieck category with a generating set of finitely presented objects (see \cite{Kra97}). In this case, the finitely generated objects are precisely the quotients of the finitely presented ones. Moreover, $\cal A$ is called

\begin{itemize}
\item {\it locally coherent\/} if it is locally finitely presented and $\fp({\cal A})$ is an abelian category.
\end{itemize}

\noindent In Example~\ref{e:qlc-non-lc} we will show that for an additive category the notion of quasi local coherence is independent to the usual notion of local coherence, i.e.~we will exhibit an additive category which is quasi locally coherent and whose finitely presented objects do not form an abelian category.

Throughout this paper, we will set $\fg(R\lMod)\reqdef\mathop{\rm gen}R$ and $\fp(R\lMod)\reqdef R\lmod$.

\subsection{Torsion theories}%
Let $\cal A$ be an abelian category. A {\it torsion pair\/} in $\cal A$ is a pair $({\cal T},{\cal F})$ of full subcategories such that $\Hom_{\cal A}({\cal T},{\cal F})=0$ and for every $M\in{\cal A}$ there exists a (functorial) short exact sequence $0\to X\to M\to Y\to0$ such that $X\in{\cal T}$ and $Y\in{\cal F}$. Generally, such approximating exact sequence will be expressed by means of two functors, say $x$ and $y$, involved in the following adjoint pairs
$$
	{\cal T} \coreflective[x] {\cal A}
		\reflective[y] {\cal F} \;.
$$
In view of this display, the induced endofunctors of $\cal A$, denoted by $x$ and $y$ again, are called respectively the {\it torsion radical\/} and the {\it torsion coradical\/}. $\cal T$ is called the {\it torsion class\/} of the torsion pair and its objects are the {\it torsion objects\/} of $\cal A$ (w.r.t.~the torsion pair), whereas $\cal F$ is the {\it torsionfree class\/} and its objects are the {\it torsionfree objects\/} of $\cal A$.

When $\cal C$ is a Grothendieck category, then two full subcategories $\cal T$ and $\cal F$ form a torsion pair $({\cal T},{\cal F})$ in $\cal C$ if and only if ${\cal T}^{\bot_0}={\cal F}$ and ${\cal T}={}^{\bot_0}{\cal F}$. On the other hand, a non-empty class $\cal T$ of objects of $\cal C$ is a torsion class iff it is closed under quotient objects, extensions, and coproducts; dually, a non-empty class $\cal F$ of objects of $\cal C$ is a torsionfree class iff it is closed under subobjects, extensions and products.

A torsion pair $({\cal T},{\cal F})$ in $\cal C$ is said to be:

\begin{itemize}
\item {\it hereditary\/} if $\cal T$ is closed under taking subobjects (equivalently, if $\cal F$ is closed under taking injective envelopes);

\item {\it stable\/} if it is hereditary and also $\cal T$ is closed under injective envelopes;

\item {\it of finite type\/} if $\cal F$ is closed under taking direct limits.
\end{itemize}

\noindent Moreover, we say that the torsion pair $({\cal T},{\cal F})$ {\it restricts to $\fp({\cal C})$\/} if for every $B\in\fp({\cal C})$ we have $x(B),y(B)\in\fp({\cal C})$.

\begin{rem}
Let $\cal C$ be a Grothendieck category.%
\label{r:fp_additive_category}%

\begin{enumerate}
\item By the properties of a torsion pair, it is clear that any torsion class of $\cal C$ is an additive category with direct limits; moreover it has cokernels and kernels, the former computed as in $\cal C$, the latter by taking the torsion radical of the kernels of $\cal C$ (thus, it makes sense to ask whether or when it fulfils some finiteness condition). 

However, a torsion class needs not to be an abelian category. For instance, given a tilting (non projective) object $V\in{\cal C}$ (see e.g.~\cite{PS20}), the induced tilting torsion class $\mathop{\rm Gen}V=\ker\Ext^1_{\cal C}(V,-)$ is not abelian. Indeed, since $\cal C$ provides injective envelopes, then the tilting class is cogenerating, meaning that for any $M\in{\cal C}$ there exists a short exact sequence $0\to M\to T\buildrel q\over\to T'\to0$ for some  $T,T'\in\mathop{\rm Gen}V$. Now, we can choose $0\ne M\in\ker\Hom_{\cal C}(V,-)$, and if $\mathop{\rm Gen}V$ would be abelian, then $q$ would be an isomorphism in $\mathop{\rm Gen}V$, in particular a split epimorphism in $\cal C$, contradiction.

\item When $({\cal T},{\cal F})$ is a hereditary torsion pair of $\cal C$ such that ${\cal T}=\mathop{\rm Gen}M$ for some object $M\in{\cal T}$, then $\cal T$ is a Grothendieck category with the same exact structure of $\cal C$.
\end{enumerate}
\end{rem}

We are particularly interested in (hereditary) torsion pairs of finite type (mostly in view of Theorem~\ref{t:lc-TTF_finite_type}).

In the case of $R\lMod$ (see \cite[Theorem~VI.5.1]{Ste75} and \cite[Appendix]{GP08}), a hereditary torsion pair (of finite type) corresponds bijectively to a {\it Gabriel filter\/} (of finite type) of $R$; that is, to a set $\cal G$ of ideals of $R$ fulfilling the following axioms:

\begin{enumerate}
\item[(i)] for any $I,J\in{\cal G}$, $I\cap J\in{\cal G}$;

\item[(ii)] if $I\in{\cal G}$ and $J$ is an ideal such that $J\supseteq I$, then $J\in{\cal G}$;

\item[(iii)] if $I\in{\cal G}$ and $r\in R$, then $(I:r)\leqdef\{\gamma\in R\mid \gamma r\in I\}\in{\cal G}$;

\item[(iv)] for any ideal $J$, if there exists an ideal $I\in{\cal G}$ such that $(J:a)\in{\cal G}$ for all $a\in I$, then $J\in{\cal G}$.
\end{enumerate}

\noindent Recall that a Gabriel filter $\cal G$ is {\it of finite type\/} if it has a basis of finitely generated ideals; that is, if every ideal in $\cal G$ contains a finitely generated ideal in $\cal G$. The bijective correspondence between hereditary torsion pairs of finite type $({\cal T},{\cal F})$ in $R\lMod$ and Gabriel filters of finite type $\cal G$ of $R$ is given by the mutually inverse assignments
\begin{align*}
	{\cal T} &\longmapsto {\cal G}_{\cal T}\leqdef\{I\le R\mid R/I\in{\cal T}\} \cr
	\noalign{\hbox{and}}
	{\cal G} &\longmapsto {\cal T}_{\cal G}\leqdef\{M\in R\lMod\mid \mathop{\rm Ann}\nolimits_R(x)\in{\cal G}\ \forall x\in M\} \;.
\end{align*}

Another particular case of torsion theories in a Grothendieck category $\cal C$ we are interested in is given by the {\it TTF~triples\/}, namely triples $({\cal E},{\cal T},{\cal F})$ such that both $({\cal E},{\cal T})$ and $({\cal T},{\cal F})$ are torsion pairs of $\cal C$ (see \cite{BR07} for a detailed reference). The middle term $\cal T$ is called {\it TTF~class\/} of the triple; by the closure properties of torsion and torsionfree classes, it follows that a full subcategory $\cal T$ of $\cal C$ is a TTF~class if and only if $\cal T$ is closed under subobjects, quotients, coproducts, products and extensions. In this case, since the right constituent $({\cal T},{\cal F})$ is hereditary, as well as the left constituent $({\cal E},{\cal T})$ is of finite type, a TTF~triple is {\it hereditary\/} resp.~{\it of finite type\/} in case its left, resp.~right, constituent is so.

TTF triples over a commutative ring $R$ are well-understood (see \cite[VI.8]{Ste75}): they are in bijection with idempotent ideals of $R$, and $\cal T$ is a TTF~class in $R\lMod$ if, and only if, there is an idempotent ideal $J\le R$ such that $\cal T$ consists of the modules annihilated by $J$, i.e.\ ${\cal T}=R/J\lMod$, so that in the left constituent $({}^{\bot_0}{\cal T},{\cal T})$ of the triple the torsion modules are precisely the $J$-divisible modules, i.e.~those $M\in R\lMod$ such that $JM=M$.

\subsection{$\rm t$-structures}%
\label{ss:t-structures}
The corresponding notion of torsion pair for a triangulated category is the one of t-structure, introduced in the celebrated work \cite{BBD82}, to which we will refer to. t-structures provide a useful approximation theory in their ambient triangulated category, as well as torsion pairs do in their ambient abelian category. The most powerful feature of such approximation theory is that each t-structure makes a ``homological algebra'' available within its triangulated category, and the relevant cohomologies belong to a suitable abelian category naturally associated with the t-structure.

Let $({\cal D},(-)[1])$ be a triangulated category. A {\em t-structure\/} in $\cal D$ is a pair $({\cal U},{\cal V})$ of full subcategories closed under direct summands and satisfying the following conditions:

\begin{enumerate}
\item[(i)] ${\cal U}[1]\subseteq{\cal U}$; that is, $\cal U$ is closed under positive shifts;

\item[(ii)] $\Hom_{\cal D}({\cal U},{\cal V}[-1])=0$;

\item[(iii)] For any object $M\in{\cal D}$, there exists an exact triangle $U\to M\to V\buildrel+\over\to$ with $U\in{\cal U}$ and $V\in{\cal V}[-1]$.
\end{enumerate}

\noindent The assignments $M\mapsto U$ and $M\mapsto V$ provided by axiom~(iii) underlie the so-called {\it truncation functors\/} $\tau^\le_{\cal U}$ and $\tau^>_{\cal U}$ of the t-structure, which are adjoint to the relevant inclusions:
$$
	{\cal U}\coreflective[] {\cal D}:\tau^\le_{\cal U}\qquad\hbox{and}
	\qquad\tau^>_{\cal U}:{\cal D}\reflective[]{\cal V[-1]} \;.
$$
By the axioms of a triangulated category, it is readily seen that any t-structure $({\cal U},{\cal V})$ can be expressed by means of the first component $\cal U$ via the equality ${\cal V}={\cal U}^{\bot_0}[1]$. $\cal U$ is called the {\em aisle\/} and its orthogonal ${\cal U}^{\bot_0}$ is the {\em coaisle\/} of the t-structure. We recall that $({\cal U},{\cal V})$ is a t-structure if, and only if, $({\cal U}[n],{\cal V}[n])$ is a t-structure for every $n\in\Z$.

Let us recall the well-known and most important results from \cite{BBD82} on a t-structure $({\cal U},{\cal V})$ we are going to use in the sequel. The main one is that the intersection ${\cal H}\leqdef{\cal U}\cap{\cal V}$ turns out to be an abelian category, called the {\em heart\/} of the t-structure. This said, the ``homological algebra'' we referred to on $\cal D$ is provided by the naturally isomorphic cohomological functors $H_{\cal H},\tilde H_{\cal H}\colon{\cal D}\to{\cal H}$ defined as
$$
	H_{\cal H}\leqdef \tau^>_{{\cal U}[1]}\circ\tau^\le_{\cal U}\cong
	\tau^\le_{\cal U}\circ\tau^>_{{\cal U}[1]}\reqdef \tilde H_{\cal H} \;.
$$

The abelian structure of $\cal H$ is described as follows. Given a morphism $f\colon M\to N$ in $\cal H$, embed it in an exact triangle of $\cal D$ by means of a cone $C$. Consider the approximation of $C[-1]$ within $({\cal U},{\cal V})$, then the following octahedron provided by a cone $W$ of the morphism $\tau^\le_{\cal U}(C[-1])\to M$,
\[
\xymatrix{%
	\tau^\le_{\cal U}(C[-1]) \ar[r]\ar@{=}[d] & C[-1]\ar[r]\ar[d] &
		\tau^>_{\cal U}(C[-1]) \ar[r]^-{+} \ar@{.>}[d] & {} \\
	\tau^\le_{\cal U}(C[-1]) \ar[r]\ar[d] & M \ar[r]\ar[d]^f &
		W \ar[r]^+\ar@{.>}[d] & {} \\
	0 \ar[r] & N \ar@{=}[r]\ar[d] & N \ar@{.>}[d] \\
	& C \ar[r] & \tau^>_{\cal U}(C[-1])[1]
}
\]
Define:
\begin{align*}
	\ker^{({\cal H})}(f) &\leqdef \tau^\le_{\cal U}(C[-1]) \cr
	\im^{({\cal H})}(f) &\leqdef W \cr
	\coker^{({\cal H})}(f) &\leqdef \tau^>_{\cal U}(C[-1])[1]=H_{\cal H}(C) \;.
\end{align*}
Moreover, the short exact sequences of $\cal H$ are precisely the exact triangles of $\cal D$ whose vertices belong to $\cal H$. Consequently, we have the following crucial correspondences, valid for all $M,N\in{\cal H}$:
\begin{align*}
	\Ext^1_{\cal H}(M,N) &\buildrel\cong\over\longrightarrow \Hom_{\cal D}^{\vphantom1}(M,N[1]) \cr
	\Ext^2_{\cal H}(M,N) &\lhook\joinrel\longrightarrow \Hom_{\cal D}^{\vphantom2}(M,N[2]) \;.
\end{align*}

If the ambient triangulated category $\cal D$ admits coproducts, which we will denote by the symbol $\coprod$, then the heart $\cal H$ has coproducts as well, generally distinct to those of $\cal D$; indeed, given a family $(M_i)_{i\in I}$ of objects of $\cal H$, it is not difficult to see that
$$
	\bigoplus_{i\in I}\adjust M_i\leqdef H_{\cal H}\Bigl(\coprod_{i\in I}M_i\Bigr)
$$
is the coproduct of the family in $\cal H$. It is now clear how to compute direct limits. When the coaisle ${\cal V}[-1]$ is closed under coproducts in $\cal D$, then the coproducts of $\cal H$ coincide with those of $\cal D$, and in this case the t-structure $({\cal U},{\cal V})$ is said to be {\em smashing\/}. The dual notion of products and inverse limits are also available in case $\cal D$ has products.

\begin{exmpl}
Let $\cal A$ be an abelian category. For every $n\in\Z$, let
\begin{align*}
	\De^{\le n}({\cal A}) &\leqdef \{M\in\De({\cal A})\mid H^j(M)=0\ \forall j>n\}, \cr
	\De^{\ge n}({\cal A}) &\leqdef \{M\in\De({\cal A})\mid H^j(M)=0\ \forall j<n\}
\end{align*}
be the subcategories of bounded below resp.~above complexes over $\cal A$ (notice that $\De^{\le n}({\cal A})=\De^{\le0}({\cal A})[-n]$). Then $({\cal D}^{\le n}({\cal A}),{\cal D}^{\ge n}({\cal A}))$ is a t-structure of $\De({\cal A})$, called the {\it (shifted) standard t-structure\/}, and its heart is equivalent to ${\cal A}[-n]$.
\end{exmpl}

\begin{exmpl}
Let $\cal C$ be a Grothendieck category and $({\cal T},{\cal F})\reqdef\boldsymbol{\tau}$ be a torsion pair in $\cal C$. The {\it Happel-Reiten-Smal\o\ t-structure\/} associated with $\boldsymbol{\tau}$ (introduced in \cite{HRS96}) is the t-structure of the bounded derived category $\De({\cal C})$, whose members are defined respectively as
\begin{align*}
	{\cal U}_{\boldsymbol{\tau}} &\leqdef \{M\in\De^{\le0}({\cal C}) \mid H^0(M)\in{\cal T}\} \cr
	\noalign{\hbox{and}}
	{\cal V}_{\boldsymbol{\tau}} &\leqdef \{M\in\De^{\ge-1}({\cal C}) \mid H^{-1}(M)\in{\cal F}\} \;.
\end{align*}
Therefore, the associated {\it HRS heart\/} ${\cal H}_{\boldsymbol{\tau}}$ consists of the cochain complexes $0\to Y\buildrel d\over\to X\to0$ over $\cal C$ concentrated in degrees $-1$ and $0$ having $\ker d\in{\cal F}$ and $\coker d\in{\cal T}$. Such heart admits $({\cal F}[1],{\cal T}[0])$ as torsion pair. We recall that in \cite{PS15,PS16a} it is proved that ${\cal H}_{\boldsymbol{\tau}}$ is a Grothendieck category if and only if $\boldsymbol{\tau}$ is of finite type.%
\label{e:HRS-heart}%
\end{exmpl}

\begin{rem}
We will deal with (hearts of) certain {\it compactly generated\/} t-structures in the derived category of a commutative ring $R$. In relation to our instance of providing finiteness conditions (in particular, the local coherence) on a given abelian category, the interest in compactly generated t-structures of $\De(R)$ is motivated by the recent paper \cite{SS20}, in which it is proved that the heart of a compactly generated t-structure in a triangulated category with coproducts is a locally finitely presented Grothendieck category (see Theorem~8.20 therein).%
\label{r:cg}%

We recall that $\De(R)$ admits coproducts (and products) and it is {\it compactly generated\/}, meaning that there exists a set $\cal S$ of complexes such that:

\begin{enumerate}
\item[(i)] every $S\in{\cal S}$ is a {\em compact object\/} of $\De(R)$, i.e.~each functor $\Hom_{\De(R)}(S,-)\colon\De(R)\to\Ab$ commutes with coproducts;

\item[(ii)] $\cal S$ is a {\em set of generators\/} of $\De(R)$, i.e.~given $M\in\De(R)$, it is $M=0$ if and only if $\Hom_{\De(R)}(S[k],M)=0$ for all $k\in\mathbb{Z}$ and all $S\in{\cal S}$.
\end{enumerate}

\noindent The full subcategory of $\De(R)$ formed by the {\it compact objects\/}, i.e.~those satisying~(i), will be denoted by $\De^c(R)$. A t-structure $({\cal U},{\cal V})$ of $\De(R)$ is {\it compactly generated\/} if there is a set $\cal S$ of compact objects of $\De(R)$ such that $\cal U$ coincides with the smallest suspended subcategory of $\De(R)$ closed under coproducts and containing $\cal S$ (in this case we will write ${\cal U}=\mathop{\rm aisle}{\cal S}$); equivalently, see \cite[Lemma~3.1]{AJSo03}, if the set $\cal S$ is such that ${\cal V}=\bigcap_{k\ge0}\ker\Hom_{\De(R)}({\cal S}[k],-)$.
\end{rem}

\subsection{Derivators}%
We briefly recall some terminology and basic facts concerning Grothendieck prederivators, more precisely the strong and stable derivators, following \cite{Gro13,Sto14,Lak18,SSV17}. The aim is to remind that to any such derivator it is naturally associated a triangulated category, called its base, in which homotopy limits and colimits are defined; furthermore, there is a strong and stable derivator whose base is equivalent to the derived category of a fixed ring, so that the homotopy colimits of this latter may be managed (and understood) in the base instead.

Let $\bf Cat$ be the $2$-category of all categories, $\bf cat$ be the $2$-category of small categories, and ${\bf cat}^{\rm op}$ be the $2$-category obtained by $\bf cat$ reversing the arrows of the $1$-cells and letting the $2$-cells unchanged. A {\it prederivator\/} is a strict $2$-functor $\pDer\colon {\bf cat}^{\rm op}\to{\bf Cat}$. Let $\bf1$ be the discrete small category consisting of one object, and let $I$ be any small category; then each object $i\in I$ may be regarded as the functor ${\bf1}\to I$, $0\mapsto i$, and similarly any morphism $\lambda\colon i\to j$ in $I$ is a natural transformation of functors $i\to j$. $\pDer({\bf1})$ is called the {\it base\/} of the prederivator $\pDer$, $\pDer({\bf1})^I$ is the category of {\it incoherent diagrams of shape $I$ on $\pDer$\/}, while $\pDer(I)$ is the category of {\it coherent diagrams of shape $I$ on $\pDer$\/}. There is a canonical {\it diagram functor\/} associated with a prederivator $\pDer$,
\begin{align*}
	\mathop{\rm diag}\nolimits_I\colon \pDer(I) &\longrightarrow \pDer({\bf1})^I \cr
	{\cal X} &\longmapsto (i\mapsto \pDer(i)({\cal X})\reqdef{\cal X}_i),
\end{align*}
which in general is not an equivalence of categories. Since $\bf1$ is a terminal object of $\bf cat$, for every small category $I\in{\bf cat}$ there is a unique functor $\mathop{\rm pt}_I\colon I\to{\bf1}$; the {\it homotopy colimit\/} (resp.~{\it homotopy limit\/}) functor is the left (resp.~right) adjoint to the functor $\pDer(\mathop{\rm pt}_I)\colon \pDer({\bf1})\to\pDer(I)$:
\begin{align*}
	\mathop{\rm hocolim}\limits_{i\in I}:\pDer(I) &\longrightleftarrows\pDer({\bf1}):\pDer(\mathop{\rm pt}\nolimits_I) \cr
	\noalign{\hbox{and}}
	\pDer(\mathop{\rm pt}\nolimits_I):\pDer({\bf1}) &\longrightleftarrows\pDer(I):\mathop{\rm holim}\limits_{i\in I} \;.
\end{align*}
In general, a prederivator needs not to admit homotopy (co)limits; in fact, {\it derivators\/} are axiomatised in order to guarantee (also) their existence for all $I\in{\bf cat\/}$. Besides, the axioms of {\it strong and stable derivators\/} provide the conditions in order to equip each of their images with a triangulated structure and, moreover, their homotopy (co)limits into triangulated functors (see \cite[Theorem~4.16, Corollary~4.19]{Gro13}) Furthermore, by \cite[Theorem~A]{SSV17}, given a strong and stable derivator $\pDer$ and a t-structure $({\cal U},{\cal V})$ with heart $\cal H$ in the base $\pDer({\bf1})$, then the classes
\begin{align*}
	{\cal U}_I &\leqdef \{{\cal X}\in\pDer(I)\mid {\cal X}_i\in{\cal U},\, \forall i\in I\} \cr
	\noalign{\hbox{and}}
	{\cal V}_I &\leqdef \{{\cal Y}\in\pDer(I)\mid {\cal Y}_i\in{\cal V},\, \forall i\in I\}
\end{align*}
form a t-structure with heart ${\cal H}_I$ in the category $\pDer(I)$, and the diagram functor induces an equivalence of abelian categories ${\cal H}_I\cong{\cal H}^I$.

\begin{exmpl}
Let $R$ be a ring. For any small category $I\in{\bf cat}$, the assignment
\begin{align*}
	\pDer_R\colon {\bf cat}^{\rm op} &\longrightarrow {\bf Cat} \cr
	I &\longmapsto \smash{\De(R\lMod^I)} \cr
	(u\colon J\to I) &\longmapsto \smash{\bigl(\De(R\lMod^I)\buildrel u^\ast\over\to\De(R\lMod^J)\bigr)},
\end{align*}
where $u^\ast\leqdef\pDer_R(u)$ is induced by the exact functor $R\lMod^I\to R\lMod^J$ given by the precomposition by $u$, well-defines a strong and stable derivator, called the {\it standard derivator of\/ $R$\/}. The base $\pDer_R({\bf1})$ is then equivalent to the derived category of the ring. On the other hand, since $\Ch(R\lMod^I)\cong\Ch(R)^I$ canonically, the objects of $\pDer_R(I)=\De(R\lMod^I)$ can be regarded as $I$-shaped diagrams with values in $\Ch(R)$. In these derived categories, the homotopy (co)limits are naturally isomorphic to the total right (resp.~left) derived functors of the ordinary (co)limits of complexes, i.e.~for every $I\in{\bf cat}$ and ${\cal X}$ in $\De(R\lMod^I)$, say it $(X_i)_{i\in I}$ in $\Ch(R)^I$, we have
$$
	\mathop{\rm holim}\limits_{i\in I}{\cal X}={\bf R}\!\lim_{i\in I} X_i\quad\hbox{and}\quad
		\mathop{\rm hocolim}\limits_{i\in I}{\cal X}=
	{\bf L}\!\mathop{\rm colim}\limits_{i\in I} X_i \;.
$$
We are mostly interested in the case of filtered homotopy colimits, namely when $I$ is a directed poset. In this case, the ordinary colimit functor $\Ch(R)^I\to\Ch(R)$ is exact, hence for any $\cal X$ as above we have a natural isomorphism
$$
	\varhocolim_{i\in I}{\cal X}\cong\varinjlim_{i\in I} X_i \;.
$$%
\end{exmpl}

\section{A criterion for the local coherence}%
\label{s:a_criterion_for_the_local_coherence}%
We prove a general result characterising the local coherence of a Grothendieck categories equipped with a TTF~triple of finite type. The body of the present paper will be prominently focused in specialising this result to the hearts of certain t-structures, as announced in Remark~\ref{r:cg} (cf.~Remark~\ref{r:notation}(1)).

\begin{thm}
Let\/ $\cal C$ be a Grothendieck category equipped with a TTF~triple of finite type $({\cal E},{\cal T},{\cal F})$. Consider the following three statements:%
\label{t:lc-TTF_finite_type}%
\begin{enumerate}
\item[(a)] $\cal C$ is locally coherent;

\item[(b)] The following conditions are satisfied:

\begin{enumerate}
\item[(i)] $\cal E$ and\/ $\cal T$ are quasi locally coherent;

\item[(ii)] For every $P\in\fp({\cal E})$, the functor $\Ext^1_{\cal C}(P,-)$ commutes with direct limits of direct systems of\/ $\cal T$;

\item[(iii)] For every $Q\in\fp({\cal T})$, the functor $\Ext^1_{\cal C}(Q,-)$ commutes with direct limits of direct systems of\/ $\cal E$.
\end{enumerate}

\item[(c)] Conditions~$\rm(i),(ii)$ of part~$\rm(b)$ hold true, and moreover

\begin{enumerate}
\item[(iii)'] The torsion pair $({\cal E},{\cal T})$ restricts to $\fp({\cal C})$.
\end{enumerate}
\end{enumerate}

\noindent Then $\mathchar``\mathchar``\rm(a)\Leftrightarrow(b)\Rightarrow(c)\mathchar`'\mathchar`'$, and the statements are all equivalent in case $\cal C$ is locally finitely presented.
\end{thm}
\begin{proof}
We will denote by
$$
	{\cal E}\coreflective[x]{\cal C}\reflective[y]{\cal T}
$$
the adjunction provided by the left constituent of the TTF~triple. Notice that the hypotheses on the TTF triple imply, by \cite[Lemma~2.4]{PSV19}, that $\fp({\cal E}),\fp({\cal T})\subseteq\fp({\cal C})$. In turn, the local coherence of $\cal C$ always implies conditions~(ii), (iii) and~$\rm(iii)\mathchar`'$: 
the first two conditions are guaranteed by \cite[Proposition~3.5(2)]{Sao17}; to see the third, for every $B\in\fp({\cal C})$ we have $x(B)\in\fp({\cal C})$, since the latter object occurs as the kernel of the epimorphism $B\to y(B)$ in $\fp({\cal C})$.

Let us prove ``${\rm(a)\Rightarrow(b)}$''. By what we just observed, we only have to check condition~(i). $\cal T$ is a locally coherent Grothendieck category thanks to \cite[Theorem~2.16]{Her97} and \cite[Theorem~2.6]{Kra97}, thus it is quasi locally coherent. Now, let us show that $\cal E$ is locally finitely presented. Let $X\in{\cal E}$ and $(B_i)_{i\in I}$ be a direct system in $\fp({\cal C})$ such that $X=\varinjlim_{i\in I}B_i$. Since $({\cal E},{\cal T})$ is of finite type, we have $X=x(X)=\varinjlim_{i\in I}x(B_i)$, thus $\cal E$ is locally finitely presented since each $x(B_i)$ belongs to $\fp({\cal E})$. It remains to show that $\cal E$ is quasi locally coherent. By the previous part, it suffices to check that the kernel in $\cal E$ of an epimorphism $f\colon P\to P'$ in $\fp({\cal E})$ is finitely presented as well. Notice that $f$ is an epimorphism also in $\cal C$, therefore $\ker f\in\fp({\cal C})$ by the local coherence hypothesis. Our claim then follows since $\ker^{({\cal E})}(f)=x(\ker f)$ and $({\cal E},{\cal T})$ restricts to $\fp({\cal C})$.

Now, let us now show that if $\cal C$ is locally finitely presented, then ``${\rm(c)\Rightarrow(a)}$''. We have to prove that the kernel of any epimorphism $f\colon B\to B'$ in $\fp({\cal C})$ is finitely presented as well. Since the torsion pair $({\cal E},{\cal T})$ restricts to $\fp({\cal C})$ by~(iii)', the following commutative diagram with exact rows
\[
\xymatrix{%
	0\ar[r] & x(B) \ar[r]\ar[d]_p & B \ar@{->>}[d]^f\ar[r] & y(B) \ar@{->>}[d]^q\ar[r] & 0 \cr
	0\ar[r] & x(B') \ar[r] & B' \ar[r] & y(B') \ar[r] & 0
}
\]
lives in $\fp({\cal C})$. Besides $q$, also $p$ is an epimorphism, being $\coker p\in{\cal E}\cap{\cal T}=0$. Therefore, $p$ and $q$ are epimorphisms in $\fp({\cal E})$ and $\fp({\cal T})$ respectively, hence by hypothesis~(i) we obtain that $\ker^{({\cal E})}(p)$ and $\ker^{({\cal T})}(q)=\ker q$ are finitely presented objects of $\cal C$. Thus, once we prove that $\ker p\in\fp({\cal C})$, we infer that $\ker f$ is finitely presented by extension-closure (see \cite[Corollary~1.8]{PSV19}), applied on the short exact sequence $0\to\ker p\to\ker f\to\ker q\to0$ provided by the Snake~Lemma. Consider the approximation
$$
	0\longrightarrow\ker^{({\cal E})}(p)\longrightarrow\ker p\longrightarrow y(\ker p)\longrightarrow0
$$
of the relevant kernel within $({\cal E},{\cal T})$, and let us prove that the third term is finitely presented in $\cal C$. We have the following pushout diagram:
\[
\xymatrix{%
	\ker^{({\cal E})}(p) \ar@{>->}[d]\ar@{=}[r] & \ker^{({\cal E})}(p) \ar@{>->}[d] \cr
	\ker p \ar@{->>}[d]\ar@{>->}[r]\ar@{}[dr]|{\rm P.O.} &
		x(B) \ar@{->>}[d]\ar@{->>}[r]^p & x(B') \ar@{=}[d] \cr
	y(\ker p) \ar@{>->}[r] & C \ar@{->>}[r] & x(B')
}
\]
whose second column tells us that the pushout $C$ is finitely presented as well. Eventually, given a direct system $(M_i)_{i\in I}$ of objects of $\cal T$, applying the functors
$$
	\varinjlim_{i\in I}\Ext^k_{\cal C}(-,M_i)\quad\hbox{and}\quad
	\Ext^k_{\cal C}(-,\varinjlim_{i\in I}M_i)\qquad (k\in\N\cup\{0\})
$$
on the second exact row, thanks to hypothesis~(ii), by the Five~Lemma we get that $\Hom_{\cal C}(y(\ker p),-)$ preserves direct limits of $\cal T$; that is, $y(\ker p)$ is a finitely presented object of $\cal T$, hence of $\cal C$, as desired.

In order to conclude the proof, we now show that condition~(b) implies that $\cal C$ is locally finitely presented and the condition~(c). For the first claim we will follow the pattern of the proof of \cite[Lemma~2.5]{PSV19}. Let $M$ be an arbitrary object of $\cal C$ and consider its approximation $0\to x(M)\to M\to y(M)\to0$ within $({\cal E},{\cal T})$. Since $\cal T$ is locally finitely presented by~(i), there exists a direct system $(Q_i)_{i\in I}$ in $\fp({\cal T})$ such that $y(M)=\varinjlim\limits_{i\in I}Q_i$. We have the pullback diagram
\[
\xymatrix{%
	0\ar[r] & x(M) \ar[r]\ar@{=}[d] & M_i \ar[d]\ar[r]\ar@{}[dr]|{\rm P.B.} & Q_i \ar[d]\ar[r] & 0 \cr
	0\ar[r] & x(M) \ar[r] & M \ar[r] & y(M) \ar[r] & 0
}
\]
and the $M_i$'s form a direct system in $\cal C$ whose direct limit is $M$. Once we show that $M_i\in\mathop{\rm Gen}[\fp({\cal C})]$ for all $i\in I$, then we conclude our first claim (see the proof of \cite[Lemma~2.5]{PSV19}). Consider the extension $\xi_i:0\to x(M)\to M_i\to Q_i\to0$ provided by the previous diagram. Since $\cal E$ is locally finitely presented, there exists a direct system $(P_\lambda)_{\lambda\in\Lambda}\subseteq\fp({\cal E})$ such that $x(M)=\smash[b]{\varinjlim\limits_{\lambda\in\Lambda}P_\lambda}$. By hypothesis~(iii), we obtain
$$
	\xi_i\in\Ext^1_{\cal C}(Q_i,\varinjlim_{\lambda\in\Lambda}P_\lambda)\cong\varinjlim_{\lambda\in\Lambda}\Ext^1_{\cal C}(Q_i,P_\lambda),
$$
i.e., by definition of Yoneda ext-group, there is an index $\gamma\in\Lambda$ such that $\xi_i$ factors as the pushout diagram (see again \cite{PSV19} for details)
\[
\xymatrix{%
	0\ar[r] & P_\gamma \ar[r]\ar[d]\ar@{}[dr]|{\rm P.O.} & N_\gamma \ar[d]\ar[r] & Q_i \ar@{=}[d]\ar[r] & 0 \cr
	0\ar[r] & x(M) \ar[r] & M_i \ar[r] & Q_i \ar[r] & 0
}
\]
in which $N_\gamma$ is a finitely presented object of $\cal C$ by \cite[Corollary~1.8]{PSV19}. Moreover, it is
$$
	M_i=\varinjlim_{\lambda\ge\gamma}N_\lambda
$$
so that our first claim is proved. Let us check that condition~(iii)' holds true. Let $B\in\fp({\cal C})$ and let us consider its approximation $0\to x(B)\to B\to y(B)\to0$ within $({\cal E},{\cal T})$. We only have to show that $x(B)\in\fp({\cal E})\subseteq\fp({\cal C})$, since $y(B)\in\fp({\cal T})\subseteq\fp({\cal C})$. The approximation yields the following long exact sequence of covariant functors:
\begin{multline*}
	0\to\Hom_{\cal C}(y(B),-)\to\Hom_{\cal C}(B,-)\to\Hom_{\cal C}(x(B),-)\relbar\joinrel\cdots \cr
	\cdots\joinrel\to\Ext^1_{\cal C}(y(B),-)\to \Ext^1_{\cal C}(B,-)
\end{multline*}
which, when restricted to $\cal E$, by hypothesis~(iii), \cite[Proposition~1.6]{PSV19} and the Five~Lemma, gives that $x(B)\in\fp({\cal E})$. \qedhere
\end{proof}

\section{Basics on Thomason filtrations and hearts}%
\label{s:basics_on_thomason_filtrations_and_hearts}%
Let $R$ be a commutative ring and $\Spec R$ be its {\it prime spectrum\/} i.e.~the set of all the prime ideals of the ring. Let us recall that for every $p\in\Spec R$ one can consider the localisation $\phi\colon R\to R_p$ of $R$ at $p$ and set $M_p\leqdef M\otimes_R R_p$ for every $M\in R\lMod$. This assignment well-defines the so-called extension of scalars functor $\phi^\ast\leqdef -\otimes_RR_p$, which is left adjoint to the scalar restriction $\phi_\ast\colon R_p\lMod\to R\lMod$ induced by $\phi$. Given $M\in R\lMod$, define its {\it support\/} by setting
$$
	\mathop{\rm Supp}M\leqdef\{p\in\Spec R\mid M\otimes_RR_p\ne0\} \;.
$$
Recall that $\Spec R$ is a topological space whose closed subsets are of the form $V(J)\leqdef\{p\in\Spec R\mid p\supseteq J\}=\mathop{\rm Supp}R/J$ for all ideals $J\le R$.

\begin{defn}
A subset $Z$ of $\Spec R$ is said to be {\it Thomason\/} if there exists a family ${\cal B}_Z$ of finitely generated ideals of $R$ such that $Z=\bigcup_{J\in{\cal B}_Z}V(J)$.
\end{defn}

Notice that $\Spec R$ is itself Thomason, for one chooses ${\cal B}_Z$ as the family of principal ideals generated by the elements of $R$, each of which is contained in some maximal ideal.

By \cite[Theorem~2.2]{GP08}, a Thomason subset $Z$ corresponds bijectively to a hereditary torsion pair of finite type $({\cal T}_Z,{\cal F}_Z)$ in $R\lMod$, where
$$
	{\cal T}_Z\leqdef\{M\in R\lMod \mid \mathop{\rm Supp}M\subseteq Z\},
$$
thus in turn it corresponds bijectively to a Gabriel filter of finite type on $R$ defined by
$$
	{\cal G}_Z \leqdef \{J\le R\mid V(J)\subseteq Z\} \;.
$$

\begin{prop}
Let\/ $Z$, ${\cal T}_Z$ and\/ ${\cal G}_Z$ be as above. Then%
\label{p:torsion-Thomason_subset}%

\begin{enumerate}
\item[(i)] ${\cal T}_Z$ is a Grothendieck category, and\/ $\fp({\cal T}_Z)={\cal T}_Z\cap R\lmod$;

\item[(ii)] ${\cal T}_Z=\mathop{\rm Gen}(R/J\mid J\in{\cal G}_Z\cap\mathop{\rm gen}R)$.
\end{enumerate}
\end{prop}
\begin{proof}\mbox{}\vglue0pt%
\noindent(i)\enspace It is well-known that ${\cal T}_Z$ is a Grothendieck category (we deduce it in Proposition~\ref{p:torsion-lfp}). Let us show the equality in the second part of the statement. The inclusion ``$\supseteq$'' is clear, while ``$\subseteq$'' follows by \cite[Lemma~1.11]{PSV19} since $({\cal T}_Z,{\cal F}_Z)$ is a torsion pair of finite type.

\noindent(ii)\enspace The inclusion ``$\supseteq$'' is clear from the properties of a torsion class. Conversely, since ${\cal T}_Z$ is a hereditary torsion class of $R\lMod$, every torsion object is the direct limit of a direct system in ${\cal T}_Z\cap\mathop{\rm gen}R$, hence it suffices to show that each module $M$ in the latter category is the homomorphic image of the direct sum of some $R/J$'s, where each $J$ is a finitely generated ideal in ${\cal G}_Z$. Since $M$ is a finitely generated module, then $\mathop{\rm Supp}M=V(\mathop{\rm Ann}_R(M))$ (see e.g.~\cite[Exercise~23, p.~58]{Lam99}). Therefore, $V(\mathop{\rm Ann}_R(M))\subseteq Z$, and since ${\cal G}_Z$ is a Gabriel filter of finite type, $\mathop{\rm Ann}_R(M)$ contains a finitely generated ideal $J$ of the filter. This means $JM=0$ i.e.~$M$ is a $R/J$-module, in fact finitely generated over $R/J$ as well, so that there exists an epimorphism $(R/J)^n\to M$ for some positive integer $n$. \qedhere
\end{proof}

\begin{cor}
Let\/ $Z=\bigcup_{J\in{\cal B}_Z}V(J)$ be a Thomason set, let ${\cal G}_Z$ be the associated Gabriel filter and set ${\cal J}_Z={\cal G}_Z\cap\mathop{\rm gen}R$. Then $Z=\bigcup_{J\in{\cal J}_Z}V(J)$.%
\label{c:torsion-Thomason_subset}
\end{cor}
\begin{proof}
The right-ward inclusion $Z\subseteq\bigcup_{J\in{\cal J}_Z}V(J)$ is clear (notice that ${\cal B}_Z\subseteq{\cal J}_Z$). Conversely, let $p$ be a prime ideal containing some finitely generated ideal $J$ in ${\cal G}_Z$, and let us prove that $p$ contains an ideal in ${\cal B}_Z$. The module $R/p$ is a torsion by Proposition~\ref{p:torsion-Thomason_subset}, whence $\mathop{\rm Supp}R/p=V(p)\subseteq Z$, so we are done since clearly $p\in V(p)$. \qedhere
\end{proof}

Henceforth, we will always identify a Thomason subset $Z=\bigcup_{J\in{\cal B}_Z}V(J)$ by setting ${\cal B}_Z$ as the the family ${\cal J}_Z$ of all finitely generated ideals in the Gabriel filter associated with $Z$.

\begin{defn}
A {\it Thomason filtration\/} of $\Spec R$ is a decreasing map $\Phi\colon(\Z,\le)\to(2^{\Spec R},\subseteq)$ such that $\Phi(n)$ is a Thomason subset of $\Spec R$ for all $n\in\Z$.

A Thomason filtration $\Phi$ will be called:%
\label{d:bounded_Thomason_filtrations}%

\begin{itemize}
\item {\it weakly bounded below\/} if there is $l\in\Z$ such that $\Phi(l)=\Phi(l-i)$ for all $i\ge0$; if $\Phi(l)=\Spec R$, then $\Phi$ is called {\it bounded below\/};

\item {\it bounded above\/} if there exists $r\in\Z$ such that $\Phi(r+1)=\emptyset$. 
\end{itemize}

\noindent In these cases, we say that $\Phi$ is {\it (weakly) bounded below~$l$\/} and {\it bounded above~$r$}, respectively.

\begin{itemize}
\item A Thomason filtration $\Phi$ weakly bounded below and bounded above will be called {\it of finite lenght\/}.
\end{itemize}

\noindent Without loss of generality, we shall assume a Thomason filtration of finite length to be bounded above~$0$, so that if it is weakly bounded below $-l$, with $\Phi(-l+1)\ne\Phi(-l)$, we will say that it has {\it length~$l$\/}.

\end{defn}

In~\cite{Hrb20} the author classifies all compactly generated t-structures in the derived category of a commutative ring $R$, generalising \cite[Theorem~3.10]{AJS10}, this latter concerning the case of a noetherian commutative ring. More precisely, \cite[Lemma~3.7]{Hrb20} exhibits a bijective correspondence between {\it bounded below\/} compactly generated t-structures in $\De(R)$ and bounded below Thomason filtrations of $\Spec R$, given explicitly by the assignments
$$
	\Phi\longmapsto ({\cal U}_\Phi^{\vphantom{\bot_0}},{\cal U}_\Phi^{\bot_0}[1]) \qquad\hbox{and}
	\qquad ({\cal U},{\cal V})\longmapsto \Phi_{\cal U},
$$
where
\begin{align*}
	{\cal U}_\Phi &\leqdef \{M\in\De(R)\mid \mathop{\rm Supp}H^n(M)\subseteq \Phi(n),\,\forall n\in\Z\} \cr
	&\hphantom{:}=\{M\in\De(R)\mid H^n(M)\in{\cal T}_{\Phi(n)},\,\forall n\in\Z\}, \cr
	\noalign{\hbox{and}}
	\Phi_{\cal U}(n) &\leqdef \bigcup_{%
	\substack{M\in R\lMod \cr M[-n]\in{\cal U}}}
		\mathop{\rm Supp}M,
\end{align*}
for all $n\in\mathbb{Z}$. Another bijective correspondence, involving arbitrary compactly generated t-structures and Thomason filtrations, is provided by \cite[Theorem~5.1]{Hrb20}; however, unlike the noetherian case, the cohomological description of the aisle by means of the supports seems currently unavailable, and for the aisle ${\cal U}_\Phi$ corresponding to a Thomason filtration $\Phi$ we have (see Remark~\ref{r:cg})
$$
	{\cal U}_\Phi=\mathop{\rm aisle}(K(J)[-n]\mid V(J)\subseteq\Phi(n),\,\forall n\in\Z) \;.
$$
Nonetheless, this latter aisle is always contained in the formerly displayed class, which still is an aisle, call it now ${\cal U}^\#$ (see \cite[Lemma~3.6]{HHZ21}), though it is not known whether it induces a compactly generated t-structure; on the other hand, these aisles coincide on the bounded below derived category:
$$
	{\cal U}_\Phi\cap\De^+(R)={\cal U}^\#\cap\De^+(R) \;.
$$
Consequently, the cohomologies of the complexes of ${\cal U}_\Phi$ are always supported on the Thomason subsets; moreover, for any module $X\in{\cal T}_{\Phi(j)}$, $j\in\Z$, by the previous equality we see that its stalk $X[-j]$ belongs to ${\cal U}_\Phi$.

In order to make clearer the distinction between the above two t-structures associated to a Thomason filtration $\Phi$, we will call the t-structure $({\cal U}^\#,({\cal U}^\#)^{\bot_0}[1])$ and its heart ${\cal H}^\#$ the {\it Alonso-Jerem\'\i as--Saor\'\i n t-structure\/}, resp.\ {\it AJS~heart\/}, induced by $\Phi$.

\begin{lem}
Let\/ $\Phi$ be any Thomason filtration of\/ $\Spec R$. Then ${\cal U}_\Phi={\cal U}^\#$ if and only if the AJS~t-structure $({\cal U}^\#,({\cal U}^\#)^{\bot_0}[1])$ is compactly generated in\/ $\De(R)$.%
\label{l:sufficient_aisles}%
\end{lem}
\begin{proof}\mbox{}\vglue0pt%
\noindent``$\Rightarrow$''\enspace There is nothing to prove since ${\cal U}_\Phi$ is the aisle of a compactly generated t-structure.

\noindent``$\Leftarrow$''\enspace By hypothesis, there exists a family ${\cal S}\subseteq\De^c(R)$ of compact complexes such that ${\cal U}^\#=\mathop{\rm aisle}{\cal S}$. By a well-known characterisation of $\De^c(R)$ due to Rickard, the objects of $\cal S$ belong in particular to $\De^+(R)$. In other words, we have ${\cal S}\subseteq{\cal U}^\#\cap\De^+(R)$, and this latter coincides with ${\cal U}_\Phi\cap\De^+(R)$, thus we infer ${\cal U}^\#\subseteq{\cal U}_\Phi$ by minimality in containing $\cal S$, and we are done. \qedhere
\end{proof}

\begin{rem}\mbox{\label{r:notation}}\vglue0pt%
\begin{enumerate}
\item We want to study the local coherence of the hearts associated with certain Thomason filtrations. This instance entails at least two possible strategies, in fact related to the hearts we are referring to, 
namely ${\cal H}_\Phi$ and ${\cal H}^\#$. More precisely, ${\cal H}_\Phi$ is always a locally finitely presented Grothendieck category, by \cite[Theorem~8.31]{SS20}, for its t-structure being compactly generated in $\De(R)$. On the other hand, it is not known so far whether ${\cal U}^\#$ is the aisle of a compactly generated t-structure, so that the cited result of Saor\'in--\v S\v tov\'i\v cek cannot apply on ${\cal H}^\#$. Nonetheless, we will prove in Theorem~\ref{t:when_hearts_coincide} that such hearts coincide under the assumption of $\Phi$ being a weakly bounded below Thomason filtration (though this does not ensure that the AJS~t-structure is compactly generated). Thanks to the announced result, we will be able to achieve a satisfactory characterisation of the local coherence of the heart of a Thomason filtration of finite length (see Theorem~\ref{t:recursive}).

\item Henceforth, when a filtration $\Phi$ is fixed we will denote the t-structure and the heart it gives rise respectively by $({\cal U},{\cal V})$ and $\cal H$, i.e.~omitting any subscript referring to $\Phi$, for it will not create confusion. Moreover, the torsion pair associated with each Thomason subset $\Phi(n)$ will be denoted just by $({\cal T}_n,{\cal F}_n)$; in turn, the relevant adjunctions to the inclusions in $R\lMod$ will be denoted by
$$
	{\cal T}_n \coreflective[x_n] R\lMod
		\reflective[y_n] {\cal F}_n \;.
$$
\end{enumerate}
\end{rem}

As claimed in the previous Remark, our interest will be focused mostly on the weakly bounded below filtrations. Let us start with the following result, which has a vital role throughout the paper.

\begin{prop}
Let\/ $\Phi$ be a weakly bounded below~$k$ Thomason filtration of\/ $\Spec R$.%
\label{p:heart_is_bounded}%

\begin{enumerate}
\item[(i)] The AJS~heart ${\cal H}^\#$ is contained in $\De^{\ge k}(R)$; in particular, when $\Phi$ has length~$l(\empty\ge0)$ we have ${\cal H}\subseteq \De^{[-l,0]}(R)$.

\item[(ii)] ${\cal H}^\#$ is an exact abelian subcategory of\/ $\cal H$ and an $\rm AB\mathchar`-3$ abelian category. Moreover, the coproducts hence the direct limits of\/ ${\cal H}^\#$ are computed as in $\cal H$.

\item[(ii)] ${\cal H}^\#$ is a locally finitely presented Grothendieck category.
\end{enumerate}
\end{prop}
\begin{proof}\mbox{}\vglue0pt%
\noindent(i)\enspace Let us prove that for every $M\in{\cal H}^\#$ we have $\tau^{\le k-1}(M)=0$. Notice that, by definition, ${\cal T}_n={\cal T}_k$ for every $n \le k$. Thus, $\tau^{\le k-1}(M)\in \cal U^\#[1]$ since $$
	H^j(\tau^{\le k-1}(M)[-1])=\begin{cases}
		H^{j-1}(M)\in {\cal T}_{j-1}={\cal T}_k = \cal T_j &\hbox{if $j\le k$}  \cr
		0 &\hbox{if $j>k$.}
	\end{cases}
$$
Therefore, in the exact triangle $\tau^{\le k-1}(M)\to M \to \tau^{>k-1}(M)\buildrel+\over\to$ the first edge is the zero morphism. By \cite[Corollary~1.2.7]{Nee01} we obtain the decomposition $\tau^{>k-1}(M)\cong M\oplus \tau^{\le k-1}(M)[1]$, thus our claim follows at once by additivity of the standard cohomology. Thus we have ${\cal H}^\#=({\cal U}^\#\cap({\cal U}^\#)^{\bot_0}[1])\cap\De^{\ge k}(R)$ and this latter coincides with ${\cal H}\cap\De^{\ge k}(R)$ by \cite[Lemma~3.6]{HHZ21}. The second part of the statement is a consequence of the previous part.

\noindent(ii)\enspace We recall that, by \cite[Proposition~3.2]{PS15}, the heart of any t-structure of $\De(R)$ is an AB-3~abelian category. The fact that the kernels and the cokernels of ${\cal H}^\#$ are the same as in $\cal H$ is clear by the construction of these objects (see subsection~\ref{ss:t-structures}). Let us prove that for any family $(M_i)_{i\in I}$ of objects of ${\cal H}^\#$ the coproduct $\bigoplus_{i\in I}^{({\cal H})} M_i$ belongs to ${\cal H}^\#$, whence the thesis. By \cite[Lemma~3.6]{HHZ21} the coproduct belongs to ${\cal U}^\#$. On the other hand, since $({\cal U},{\cal V})$ is a compactly generated hence smashing t-structure of $\De(R)$, we have $\bigoplus_{i\in I}^{({\cal H})} M_i=\coprod_{i\in I}M_i$, and this latter object belongs to ${\cal V}\cap\De^{\ge k}(R)$, hence to $({\cal U}^\#)^{\bot_0}[1]$ by [ibid.].

\noindent(iii)\enspace Since ${\cal H}^\#$ is an AB-3 category with direct limits computed as in $\cal H$, which is an AB-5 category, then in turn ${\cal H}^\#$ is an AB-5 abelian category. Recall now that, by \cite[Lemma~3.1]{PS15}, for any t-structure $({\cal X},{\cal Y})$ with heart $\cal A$ the restriction $H_{\cal A}\mathclose\upharpoonright_{\cal X}\colon {\cal X}\to{\cal A}$ is left adjoint to the inclusion ${\cal A}\lhook\joinrel\to{\cal X}$. We claim that the restriction $L\leqdef H_{{\cal H}^\#}\mathclose\upharpoonright_{\cal H}\colon{\cal H}\to{\cal H}^\#$ is left adjoint to the inclusion ${\cal H}^\#\lhook\joinrel\to{\cal H}$. Indeed, for all $M\in{\cal H}$ and $M'\in{\cal H}^\#$ we have
\begin{align*}
	\Hom_{{\cal H}^\#}(L(M),M') &= \Hom_{\De(R)}(H_{{\cal H}^\#}\mathclose\upharpoonright_{{\cal U}^\#}(M),M') &(M\in{\cal H}\subseteq{\cal U}\subseteq{\cal U}^\#) \cr
	&\cong \Hom_{{\cal U}^\#}(M,M') & \hbox{(by adjunction)} \cr
	&= \Hom_{\cal H}(M,M') & (M,M'\in{\cal H}).
\end{align*}
This said, let us prove that for any generator $U$ of $\cal H$, the object $L(U)$ is a generator of ${\cal H}^\#$. Let $M\in{\cal H}^\#$. Then there is a set $\alpha$ and an epimorphism $U^{(\alpha)}\to M$ in $\cal H$, whence applying the left adjoint functor $L$ we obtain an epimorphism $L(U)^{(\alpha)}\to L(M)$. Now, by definition of $L$, we have $L(M)=\tau^>_{{\cal U}^\#}(M[-1])[1]$, and being $M[-1]\in({\cal U}^\#)^{\bot_0}$ since $M\in{\cal H}^\#$, then actually $\tau^>_{{\cal U}^\#}(M[-1])=M[-1]$, i.e.\ $L(M)=M$, and we are done.

The previous paragraph also shows that the right adjoint to $L$ is a direct limit preserving functor, hence by \cite[Lemma~2.5]{Kra97} $L({\cal S})$ is a set of finitely presented generators for ${\cal H}^\#$ for any set $\cal S$ of finitely presented generators of $\cal H$. \qedhere
\end{proof}

\section{Some useful results}%
\label{s:some_useful_results}%
Let us introduce some further notions and useful results that will be crucial throughout the rest of the paper.

\begin{lem}[{\cite[Lemma~4.2(3)]{PS17}}]
Let\/ $\Phi$ be a Thomason filtration of\/ $\Spec R$ and let\/ $M$ be a complex in the associated heart $\cal H$. If\/ $m$ is the least integer such that $H^m(M)\ne0$, then%
\label{l:least_nonzero_cohomology}%
$$
	H^m(M)\in{\cal T}_m\cap{\cal F}_{m+1}\cap\ker\Ext^1_R({\cal T}_{m+2},-) \;.
$$
\end{lem}
\begin{proof}
By what we recalled before of Lemma~\ref{l:sufficient_aisles}, we only need to check that $H^m(M)$ belongs to the last two classes of the displayed intersection. By hypothesis, $H^m(M)[-m]\cong\tau^{\le m}(M)$, hence for every $X\in{\cal T}_{m+1}$ we obtain 
\begin{align*}
	\Hom_R(X,H^m(M)) &\cong\Hom_{\De(R)}(X[-m],H^m(M)[-m]) \cr
	&\cong\Hom_{\De(R)}(X[-m],\tau^{\le m}(M)) \;.
\end{align*}
The latter group is zero, indeed we have
$
	X[-m]\in{\cal U}^\#[1]\cap\De^+(R)={\cal U}[1]\cap\De^+(R)\subseteq{\cal U}[1],
$
so that the covariant hom~functor of the stalk complex applied on exact triangle $\tau^{>m}(M)[-1]\to\tau^{\le m}(M)\to M\buildrel+\over\to$ yields the claimed vanishing by the axioms of t-structure. Therefore, the least nonzero cohomology of $M$ is an object of ${\cal F}_{m+1}$.

On the other hand, by Verdier's thesis~\cite{Ver}, for every $X\in{\cal T}_{m+2}$ we have
$$
	\Ext^1_R(X,H^m(M))\cong\Hom_{\De(R)}(X[-m],H^m(M)[-m+1]),
$$
and the right-hand group is zero by the previous argument, i.e.~by applying the hom functor of $X[-m]\in{\cal U}[2]$ on the rotation of the above triangle. \qedhere
\end{proof}

In view of the previous proof, in particular of what we recalled before of Lemma~\ref{l:sufficient_aisles}, we see that the same statement holds true for complexes $M$ belonging to the AJS~heart ${\cal H}^\#$ of $\Phi$, namely for they belong to $\De^{\ge m}(R)$.

\begin{exmpl}
We show that the HRS~heart of a hereditary torsion pair of finite type of $R\lMod$ can be realised as the heart of a Thomason filtration bounded both below and above; as a byproduct of \cite{SS20}, such heart is automatically a locally finitely presented Grothendieck category. This example will be resumed in the last part of the paper.

Let $Z$ be a proper Thomason subset, and let ${\cal H}_{\boldsymbol{\tau}}$ the HRS~heart of the torsion pair $\boldsymbol{\tau}\leqdef({\cal T}_Z,{\cal F}_Z)$ corresponding to $Z$ (see Example~\ref{e:HRS-heart}). Consider now the Thomason filtration
$$
	\Phi:\Spec R\supset Z\supset\emptyset,
$$
and let us prove that ${\cal H}={\cal H}^\#={\cal H}_{\boldsymbol{\tau}}$. The first equality is guaranteed by \cite[Lemma~3.7]{Hrb20}.%
\label{e:HRS-heart&filtration}%

Let us prove that ${\cal H}\subseteq{\cal H}_{\boldsymbol{\tau}}$. For every $M\in{\cal H}$ we have $H^0(M)\in{\cal T}_Z$, so it remains to verify that $H^{-1}(M)\in{\cal F}_Z$. This follows by Lemma~\ref{l:least_nonzero_cohomology} and Proposition~\ref{p:heart_is_bounded}.

Conversely, let us prove the inclusion ${\cal H}_{\boldsymbol{\tau}}\subseteq{\cal H}$ by showing that both the torsion and torsionfree classes ${\cal F}_Z[1]$ and ${\cal T}_Z[0]$ approximating ${\cal H}_{\boldsymbol{\tau}}$ are contained in ${\cal H}$, whence the conclusion by the extension-closure of the heart. The fact that ${\cal T}_Z[0]\subseteq{\cal H}$ is clear by definition of the t-structure $({\cal U},{\cal V})$.

On the other hand, let $Y\in{\cal F}_Z$. Since $\mathop{\rm Supp}H^{-1}(Y[1])=\mathop{\rm Supp}Y$ is contained in the spectrum i.e.~in $\Phi(-1)$, whereas $\mathop{\rm Supp}H^j(Y[1])=\emptyset$ for all $j\ne-1$, we have ${\cal F}_Z[1]\subseteq{\cal U}$. 
Let now $M\in{\cal U}\subseteq\De^{\le0}(R)$; out of the exact triangle $\tau^{\le-1}(M)\to M\to H^0(M)[0]\buildrel+\over\to$ provided by the standard t-structure of $\De(R)$, applying the cohomological functor $\Hom_{\De(R)}(-,Y[0])$ we obtain the exact sequence
$$
	\Hom_{\De(R)}(H^0(M)[0],Y[0])\to\Hom_{\De(R)}(M,Y[0])
	\to\Hom_{\De(R)}(\tau^{\le-1}(M),Y[0])
$$
whence we obtain the remaining inclusion ${\cal F}_Z[1]\subseteq{\cal U}^{\bot_0}[1]$, for the left hand term is zero by \cite{Ver} and since $({\cal T}_Z,{\cal F}_Z)$ is a torsion pair in $R\lMod$, and for the right hand term being clearly zero as well.
\end{exmpl}

Let us recall some basic facts concerning Koszul complexes and their cohomology (see e.g.~\cite[Chap.~8]{Nor68}). For any finitely generated ideal $J$ one has:

\begin{itemize}
\item $K(J)\in\De^{[-n,0]}(R)\cap\De^c(R)$, where $n=\mathop{\rm rank}J$;

\item $H^0(K(J))\cong R/J$;

\item $JH^{-j}(K(J))=0$ or, equivalently, $\mathop{\rm Supp}H^{-j}(K(J))\subseteq V(J)$, for all $j=0,\ldots,n$.
\end{itemize}

\noindent The last two items tell us that the Koszul cohomologies are torsion modules w.r.t.~the torsion pair associated with $V(J)$.

The following result will be resumed in Section~\ref{s:arbitrary_thomason_filtrations} and stengthened in Theorem~\ref{t:when_hearts_coincide}.

\begin{prop}
Let\/ $\Phi$ be a Thomason filtration of\/ $\Spec R$. For any $n\in\Z$ and for any ideal\/ $J\in{\cal B}_n$, we have $H_{{\cal H}^{\vphantom\#}}(K(J)[-n])=H_{{\cal H}^\#}(K(J)[-n])$.%
\label{p:Koszul_complex_in_hearts}%
\end{prop}
\begin{proof}
For all $n$ and $J$ as stated, by \cite[Lemma~3.6]{HHZ21} we have $K(J)[-n]\in{\cal U}\cap\De^+(R)={\cal U}^\#\cap\De^+(R)$. On the other hand, we have $\tau^{\le n-1}(K(J)[-n])\in{\cal U}^\#[1]\cap\De^+(R)={\cal U}[1]\cap\De^+(R)\subseteq{\cal U}[1]$, so out of the triangle $\tau^{\le n-1}(K(J)[-n])\to K(J)[-n]\to R/J[-n]\buildrel+\over\to$, provided by the basic properties of the Koszul complex, we infer that
\begin{align*}
	H_{{\cal H}^{\vphantom\#}}(K(J)[-n]) &\cong H_{{\cal H}^{\vphantom\#}}(R/J[-n]) \cr
	\noalign{\hbox{and}}
	H_{{\cal H}^\#}(K(J)[-n]) &\cong H_{{\cal H}^\#}(R/J[-n]) \;.
\end{align*}
We now anticipate an argument which will be formalised in Section~\ref{s:arbitrary_thomason_filtrations}. There exists a (weakly bounded below) Thomason filtration $\Phi_n$ associated to $\Phi$ whose t-structure has its aisle ${\cal U}^\#_n$ and the heart ${\cal H}^\#_n$ contained, respectively, in ${\cal U}^\#_{\vphantom n}$ and ${\cal H}^\#_{\vphantom n}$ (see Lemma~\ref{l:inclusion-H_*}); moreover, $R/J[-n]\in{\cal U}^\#_n$, thus out of the triangle
$$
	\tau^\le_{{\cal U}^\#_n}(R/J[-n-1])[1]\longrightarrow
		R/J[-n]\longrightarrow H_{{\cal H}^\#_n}(R/J[-n])\buildrel+\over\longrightarrow
$$
we infer that $H_{{\cal H}^\#_{\vphantom n}}(R/J[-n])=H_{{\cal H}^\#_n}(R/J[-n])$. On the other hand, the first vertex of the triangle belongs to $\De^+(R)$ for the other two vertices do, therefore it also belongs to ${\cal U}[1]$ by the usual argument of \cite{HHZ21}. Thus, the displayed triangle also shows that $H_{{\cal H}^{\vphantom\#}_{\vphantom n}}(R/J[-n])=H_{{\cal H}^\#_n}(R/J[-n])$, and we are done. \qedhere
\end{proof}

\begin{lem}
Let\/ $\Phi$ be any Thomason filtration of\/ $\Spec R$, let\/ $\cal A$ be either $\cal H$ or the AJS heart ${\cal H}^\#$, and let $(M_i)_{i\in I}$ be a family of objects of the heart $\cal A$. Then, for all\/ $n\in\Z$,%
\label{l:cohomologies_commute}%
$$
	\bigoplus_{i\in I}H^n(M_i)\cong
		H^n\Bigl(\bigoplus_{i\in I}\vphantom{\displaystyle\bigoplus}^{\!({\cal A})}M_i\Bigr) \;.
$$
Moreover, if\/ $I$ is directed, then
$$
	\varinjlim_{i\in I} H^n(M_i)\cong H^n(\varinjlim_{i\in I}\adjust M_i) \;.
$$
\end{lem}
\begin{proof}
Since each stalk complex of the ring $R$ is a compact object of $\De(R)$, we have
\begin{align*}
	H^n\Bigl(\bigoplus_{i\in I}\vphantom{\displaystyle\bigoplus}^{\!({\cal A})}M_i\Bigr) &=
		\Hom_{\De(R)}(R[-n],\bigoplus_{i\in I}\vphantom{\displaystyle\bigoplus}^{\!({\cal A})}M_i) \cr
		&\cong \bigoplus_{i\in I}\Hom_{\De(R)}(R[-n],M_i)=\bigoplus_{i\in I}H^n(M_i) \;.
\end{align*}
Notice that this result in fact holds true for the heart of any smashing t-structure of $\De(R)$.

Let now $I$ be a directed set. By \cite[Proposition~5.4]{SSV17}, $R[-n]$ is a homotopically finitely presented object of $\De(R)$, meaning that its covariant hom~functor commutes with direct homotopy colimits; by [op.~cit., Corollary~5.8], the direct homotopy colimits of the heart $\cal H$ are canonically isomorphic to the underlying direct limits; consequently, we can repeat the proof above by replacing the coproducts of $(M_i)_{i\in I}$ with its direct limit, and we are done. \qedhere
\end{proof}

Slightly diverting from \cite{PS17}, we fix the following notation: given any Thomason filtration $\Phi$, for any $n\in\Z$ we set
\begin{align*}
	\TF_{\!n} &\leqdef {\cal T}_n\cap{\cal F}_{n+1} \cr
	\TFT_{\!n} &\leqdef {\cal T}_n\cap{\cal F}_{n+1}\cap\ker\Ext^1_R({\cal T}_{n+2},-) \;.
\end{align*}
It is readily seen that $\TF_{\!n}$ is closed under subobjects and that $\TFT_{\!n}$ is closed under kernels; moreover, we will show in Remark~\ref{r:split_in_fp(TFT)}(1) that the latter category has direct limits, so it will make sense to consider the subcategory of its finitely presented objects (mostly in order to study its quasi local coherence), which will play a crucial role in the subsequent sections.

\begin{cor}
Let\/ $\Phi$ be a Thomason filtration of\/ $\Spec R$. Then the class
\begin{align*}
	\HTF_{\!n} &\leqdef \{M\in R\lMod\mid M[-n]\in{\cal H}\} \cr
	\noalign{\hbox{coincides with}}
	\HTF^\#_{\!n} &\leqdef \{M\in R\lMod\mid M[-n]\in{\cal H}^\#\} \;.
\end{align*}
Moreover, $\HTF_{\!n}$ is a subcategory of\/ $R\lMod$ closed under direct limits, for every $n\in\mathbb{Z}$.%
\label{c:TF_n-AB5}%
\end{cor}
\begin{proof}
The equality between the two classes of modules follows since the corresponding stalk complexes belong to the relevant hearts and to $\De^+(R)$. Let now $(M_i)_{i\in I}\in\HTF_{\!n}$ be a direct system, so that $(M_i[-n])_{i\in I}$ is a direct system of $\cal H$. By Lemma~\ref{l:cohomologies_commute} we obtain
\begin{align*}
	H^n(\varinjlim_{i\in I}\adjust M_i[-n])
	&\cong 
		\varinjlim_{i\in I}M_i,
\end{align*}
while in any degree different from $n$ the direct limit has no cohomology. Therefore,
$$
	\varinjlim_{i\in I}\adjust M_i[-n]\cong
	\mathopen{\lower4pt\hbox{$\Bigl($}}
		\varinjlim_{i\in I}M_i
	\mathclose{\lower4pt\hbox{$\Bigr)$}}[-n]
$$
i.e.~direct limits of $\HTF_{\!n}$ are computed precisely as in $R\lMod$. \qedhere
\end{proof}

\subsection{Bounded above Thomason filtrations}%
We continue giving useful results concerning the bounded above Thomason filtrations.

\begin{lem}
Let\/ $\Phi$ be a Thomason filtration bounded above~$r$. Then%
\label{l:TF_-1=T_-1F_0}%
$$
	\HTF_{\!r-1}=\TF_{\!r-1} \;.
$$
\end{lem}
\begin{proof}
Notice that, by boundedness of $\Phi$, we have ${\cal U}\subseteq{\cal U}^\#\subseteq\De^{\le r}(R)$.

This said, let $M\in\HTF_{\!r-1}$. Then $M=H^{r-1}(M[-r+1])$, hence by Lemma~\ref{l:least_nonzero_cohomology} we obtain $M\in{\cal T}_{r-1}\cap{\cal F}_r$.

Conversely, let us prove that the stalk concentrated in degree $r-1$ of a module $M\in{\cal T}_{r-1}\cap{\cal F}_r$ belongs to the heart associated with $\Phi$. $M[-r+1]$ belongs to ${\cal U}^\#\cap\De^+(R)$, hence it lands in $\cal U$. On the other hand, we see that $M[-r]$ belongs to the coaisle $({\cal U}^\#)^{\bot_0}$, since for every $U\in{\cal U}^\#$ the standard approximation $\tau^{\le r-1}(U)\to U\to H^r(U)[-r]\buildrel+\over\to$ (provided by the boundedness of $\Phi$) yields, by \cite{Ver}, the desired vanishing $\Hom_{\De(R)}(U,M[-r])=0$. By the usual argument of \cite{HHZ21}, we infer that $M[-r+1]\in{\cal V}$. \qedhere
\end{proof}

\begin{rem}
As we shall deduce by Proposition~\ref{p:torsion-lfp} (which does not depend on the forthcoming results; see however the comments after its proof), the torsion class corresponding to any nonempty Thomason subset is a locally finitely presented Grothendieck category. In particular, for a Thomason filtration as above, by Lemma~\ref{l:TF_-1=T_-1F_0} and \cite[Corollary~5.3]{PSV19} we have%
\label{r:fp(TF_-1)}%
$$
	\fp(\TF_{\!r-1})=\mathop{\rm add} y_r(\fp({\cal T}_{r-1}))
		=\mathop{\rm add} y_r({\cal T}_{r-1}\cap R\lmod) \;.
$$
\end{rem}

\begin{lem}
Let\/ $\Phi$ be a Thomason filtration bounded above~$r$. Then:%
\label{l:Koszul_cohomology}%

\begin{enumerate}
\item[(i)] For every $J\in{\cal B}_r$, it is $H_{{\cal H}^\#}(K(J)[-r])\cong R/J[-r]$;

\item[(ii)] For every $J\in{\cal B}_{r-1}$, it is $H_{{\cal H}^\#}(K(J)[-r+1])\cong y_r(R/J)[-r+1]$.
\end{enumerate}
\end{lem}
\begin{proof}
Notice that, by definition of the aisles $\cal U$ and ${\cal U}^\#$, and by Proposition~\ref{p:Koszul_complex_in_hearts}, it makes no difference in working within the compactly generated t-structure or in the AJS~one. For the sake of clearness, we will continue in working in the AJS~t-structure.

\noindent(i)\enspace Let $J\in{\cal B}_r$. The complexes $K\leqdef K(J)[-r]$ and $M\leqdef H_{{\cal H}^\#}(K)$ fit as the vertexes of the triangle
$$
	U[1]\longrightarrow K\longrightarrow M\buildrel+\over\longrightarrow
$$
provided by the object $U\leqdef\tau^\le_{{\cal U}^\#}(K[-1])$. We will prove that $H^r(M)\cong R/J$ and that $\tau^{\le r-1}(M)=0$. Fix $j \le r-1$ and consider the exact sequence $H^j(K)\to H^j(M)\to H^{j+2}(U)$ in $R\lMod$. The Koszul cohomology $H^j(K)$ is an object of ${\cal T}_r$, hence of ${\cal T}_{j+1}$, so that $H^j(M)\in \cal T_{j+1}$ since in turn ${\cal T}_{j+1}\supseteq{\cal T}_{j+2}$. It follows $\tau^{\le r-1}(M)\in{\cal U}^\#[1]$, and from the triangle
$$
	\tau^{\le r-1}(M)\longrightarrow M\longrightarrow\tau^{>r-1}(M)\buildrel+\over\longrightarrow
$$
we deduce $\tau^{>r-1}(M)\cong M\oplus\tau^{\le r-1}(M)[1]$ by \cite[Corollary~1.2.7]{Nee01}, whence $\tau^{\le r-1}(M)[1]\in\De^{\le r-1}(R)\cap\De^{\ge r}(R)=0$. Now, the first displayed triangle yields the following exact sequence in $R\lMod$:  
$$
	H^{r+1}(U)\longrightarrow H^r(K)(\empty\cong R/J)\longrightarrow H^r(M)\longrightarrow H^{r+2}(U),
$$
whence we obtain $H^r(M)\cong R/J$ since $\Phi(r+1)=\Phi(r+2)=\emptyset$ and $U\in{\cal U}^\#$.

\noindent(ii)\enspace Let $J\in{\cal B}_{r-1}$, $K\leqdef K(J)[-r+1]$ and $M\leqdef H_{{\cal H}^\#}(K[1])$. The thesis follows as in the previous part, namely by proving that $H^j(M)=0$ for every $j\ne r-1$ and that $H^{r-1}(M)\cong y_r(R/J)$. To this aim, look at the long exact cohomology sequence arising from $U[1]\to K\to M\buildrel+\over\to$, in which $U\leqdef\tau^\le_{{\cal U}^\#}(K[-1])$, and use Lemma~\ref{l:least_nonzero_cohomology} again. \qedhere
\end{proof}

\begin{cor}
Let\/ $\Phi$ be a Thomason filtration bounded above~$r$. The following hold true for a module $X\in R\lMod$:%
\label{c:R-mod&T_0-are_in_fp(H)}%

\begin{enumerate}
\item[(i)] $X\in \fp({\cal T}_r)$ if and only if\/ $X[-r]\in\fp({\cal H})$. In particular, $H^r(\fp({\cal H}))[-r]\subseteq\fp({\cal H})$;

\item[(ii)] $X\in\fp(\HTF_{\!r-1})$ if and only if\/ $X[-r+1]\in\fp({\cal H})$.
\end{enumerate}

\noindent Moreover, if\/ $\Phi$ is of finite length, then the assertions~$\rm(i)$ and\/ $\rm(ii)$ hold true replacing $\cal H$ by ${\cal H}^\#$.
\end{cor}
\begin{proof}
Notice that $T[-r]\in{\cal H}\cap{\cal H}^\#$ for all $T\in{\cal T}_r$ (the proof is similar to that in Example~\ref{e:HRS-heart&filtration}).

\noindent(i)\enspace
Let $X$ be a finitely presented object of ${\cal T}_r$ i.e.~an object of $\fp({\cal T}_r)=R\lmod\cap{\cal T}_r$. By (the proof of) Proposition~\ref{p:torsion-Thomason_subset} there exists in ${\cal T}_r\cap R\lmod$ an exact row $(R/J')^n\buildrel\alpha\over\to(R/J)^m\to X\to0$, which can be embedded in the following diagram in $\De(R)$ by taking the stalk complexes:
\[
\xymatrix{
	(\ker\alpha)[-r] \ar[r] & (R/J')^n[-r] \ar[r] & (\im\alpha)[-r] \ar[r]^-{+} \ar[d] & {} \\
	&& (R/J)^m[-r] \ar[d] \\
	&& X[-r] \ar[d]^-{+} \\
	&& {}
}
\]
By Lemma~\ref{l:Koszul_cohomology}(i) and \cite[Lemma~6.3]{SSV17}, for every $I\in{\cal B}_r$ the stalk $R/I[-r]$ is a finitely presented object of $\cal H$. Moreover, since the triangles of the diagram are in $\cal H$ (by what we noted at the beginning of the proof), then they actually are short exact sequences of $\cal H$, hence $X[-r]\cong\coker^{({\cal H})}(\alpha[-r])$ and it is finitely presented being the cokernel of a map between finitely presented complexes.

Conversely, let $X$ be a module whose stalk $X[-r]$ is a finitely presented complex of $\cal H$. Then clearly $X\in{\cal T}_r$; moreover, for all direct systems of modules $(X_i)_{i\in I}$ in ${\cal T}_r$, so that $X_i[-r]\in{\cal H}$ for all $i\in I$, by \cite{Ver} we deduce the natural isomorphism 
$$
	\varinjlim_{i\in I}\Hom_R(X,X_i) \cong
		\Hom_R(X,\varinjlim_{i\in I} X_i),
$$
whence $X\in R\lmod$ since $({\cal T}_r,{\cal F}_r)$ is a torsion pair of finite type (see \cite[Lemma~2.4]{PSV19}).

The second part of the statement readily follows by the previous one, since out of the exact triangle $\tau^{\le r-1}(B)\to B\to H^r(B)[-r]\buildrel+\over\to$ approximating a finitely presented complex $B$ of the heart $\cal H$, by \cite{Ver} we infer that $H^r(B)$ is a finitely presented object of ${\cal T}_r$.

\noindent(ii)\enspace If $X$ is a module whose stalk $X[-r+1]$ is a finitely presented complex of the heart $\cal H$, then by definition of $\HTF_{\!r-1}$ and by Corollary~\ref{c:TF_n-AB5}, Lemma~\ref{l:TF_-1=T_-1F_0} and \cite{Ver}, for every direct system of modules $(M_i)_{i\in I}$ in $\HTF_{\!r-1}$ we obtain the following commutative diagram
\[
\xymatrix{%
	\varinjlim\limits_{i\in I}\Hom_{\cal H}(X[-r+1],M_i[-r+1]) \ar[r]^-{\cong}\ar[d]_-{\cong} &
		\Hom_{\cal H}(X[-r+1],\varinjlim\limits_{i\in I}\adjust M_i[-r+1]) \ar[d]^-{\cong} \cr
	\varinjlim\limits_{i\in I}\Hom_R(X,M_i) \ar[r] &
		\Hom_R(X,\varinjlim\limits_{i\in I} M_i)
}
\]
showing that $X$ is a finitely presented object of $\HTF_{\!r-1}$.

Conversely, let $X$ be a module in $\fp(\TF_{\!r-1})=\mathop{\rm add} y_r(\fp({\cal T}_{r-1}))$ (see Remark~\ref{r:fp(TF_-1)}), so that there exists a finitely presented object $B$ of ${\cal T}_{r-1}$ such that $X\le_\oplus y_r(B)^n$ for some $n\in\mathbb{N}$, hence we shall prove the statement on $y_r(B)^n$, in particular by showing that $y_r(B)[-r+1]\in\fp({\cal H})$. By Proposition~\ref{p:torsion-Thomason_subset} there is an exact sequence $(R/J')^n\buildrel\alpha\over\to(R/J)^m\to B\to0$ in $R\lMod$ for some positive integers $m,n$ and ideals $J',J$ in ${\cal B}_{r-1}$. By Lemma~\ref{l:Koszul_cohomology}(ii) and Proposition~\ref{p:Koszul_complex_in_hearts}, we have the exact row
$$
	H_{{\cal H}^\#}(K(J')[-r+1])^n\buildrel y_r(\alpha)[-r+1]\over{%
		\relbar\joinrel\relbar\joinrel\relbar\joinrel\relbar\joinrel\longrightarrow}
	H_{{\cal H}^\#}(K(J)[-r+1])^m\to \coker^{({\cal H}^\#)}(y_r(\alpha)[-r+1])\to0
$$
in the AJS~heart ${\cal H}^\#$, whose first two terms are finitely presented objects of $\cal H$, so $\coker^{({\cal H}^\#)}(y_r(\alpha)[-r+1])$ turns out to be finitely presented of $\cal H$ once we prove that it belongs to $\De^+(R)$. Let us prove that in fact such object is a stalk complex. To prove this, consider the canonical short exact sequences of ${\cal H}^\#$
\begin{gather*}
	0\to \ker^{({\cal H}^\#)}(y_r(\alpha)[-r+1])\to H_{{\cal H}^\#}(K(J')[-r+1])^n\to
		\im^{({\cal H}^\#)}(y_r(\alpha)[-r+1])\to0 \cr
	\noalign{\hbox{and}}	
	0\to \im^{({\cal H}^\#)}(y_r(\alpha)[-r+1])\to H_{{\cal H}^\#}(K(J)[-r+1])^m\to
		\coker^{({\cal H}^\#)}(y_r(\alpha)[-r+1])\to0,
\end{gather*}
say them $0\to K\to M'\to L\to0$ and $0\to L\to M\to N\to0$ respectively. Since their middle terms are stalk complexes concentrated in degree~$r-1$, they yield $H^r(L)=0$ and $H^r(N)=0$, respectively. On the other hand, from the second exact row, we have $H^{j-1}(N)\cong H^j(L) \in \cal T_j$ for all $j\le r-2$, and $H^{r-2}(N)$ is a submodule of $H^{r-1}(L)\in \cal T_{r-1}$. Hence $\tau^{\le r-2}(N)\in{\cal U}^\#[1]$, so that $N\cong \tau^{>r-2}(N)=\tau^{\ge r-1}(N)=H^{r-1}(N)[-r+1]$. Therefore, the very first displayed exact row $M'\to M\to N\to0$ gives, by exactness,
$$
	N=\coker^{({\cal H}^\#)}(y_r(\alpha)[-r+1])\cong D[-r+1],
$$
for some $D\in\TF_{\!r-1}$; whence $D[-r+1]\in\fp({\cal H})$. Once we prove that $y_r(B)\cong D$, then we get the thesis. By the long exact sequence in cohomology of the previous two short exact sequences, we obtain the commutative diagram with exact rows:
\[
\xymatrix@C=.5em@R=1.5em{%
	& \hphantom{\im\delta} & y_r(R/J')^n \ar[rr]\ar@{->>}[dr]\ar[dd]_-{\delta} && y_r(R/J)^m \ar[rr]\ar@{=}[dd] &&
		\coker y_r(\alpha) \ar[rr]\ar@{->>}[dd]^-{q} && 0 \cr
	&&& \im\delta \ar@{>->}[ur]\ar@{>->}[dl] && \hphantom{\im\delta} && \hphantom{\im\delta} \cr
	0 \ar[rr] && H^{r-1}(L) \ar[rr] && y_r(R/J)^m \ar[rr]^-{p} && D \ar[rr] && 0
}
\]
where $\coker\delta=H^r(K)\in{\cal T}_r$ and $p$ is an epimorphism since $H^r(L)=0$. We deduce that $D\cong y_r(\coker y_r(\alpha))$; set now $C\leqdef\coker y_r(\alpha)$. On the other hand, we have the following commutative diagram with exact rows:
\[
\xymatrix{%
	0\ar[r] & x_r(R/J')^n \ar[d]_{x_r(\alpha)}\ar[r] & (R/J')^n \ar[d]^\alpha\ar[r] &
		y_r(R/J')^n \ar[d]^{y_r(\alpha)}\ar[r] & 0 \cr
	0\ar[r] & x_r(R/J)^m \ar@{->>}[d]\ar[r] & (R/J)^m \ar@{->>}[d]\ar[r] &
		y_r(R/J)^m \ar@{->>}[d]\ar[r] & 0 \cr
	& \coker x_r(\alpha) \ar[r] & B \ar[r]^g & C \ar[r] & 0
}
\]
The short exact sequence $0\to A\to B\buildrel g\over\to C\to0$ provided by the factorisation of the morphism $\coker x_r(\alpha)\to B$ through its image $A$ yields that this latter is an object of ${\cal T}_r$. Consequently, we deduce $D\cong y_r(B)$ by the Snake~Lemma applied on the following commutative diagram
\[
\xymatrix{%
	0 \ar[r] & x_r(B) \ar[r]\ar[d] & B\ar@{->>}[d]^g\ar[r] & y_r(B) \ar[r]\ar[d] & 0 \cr
	0 \ar[r] & x_r(C) \ar[r] & C \ar[r] & y_r(C) \ar[r] & 0
}
\]
and this concludes the first part of the proof.

The second part of the proof, namely that concerning weakly bounded below filtrations, is a direct consequence of Theorem~\ref{t:when_hearts_coincide}. \qedhere
\end{proof}

\subsection{A crucial example}%
\label{ss:a_crucial_example}
We conclude the section by fixing a proper Thomason subset $Z$ of $\Spec R$ and studying the heart ${\cal H}^\#_Z$ given by the weakly bounded Thomason filtration
$$
	\Phi:\cdots =Z=Z=\cdots=Z\supset\emptyset \;.
$$

\begin{prop}
Let $\Phi$ be as above. Then the heart\/ ${\cal H}^\#_Z$ is equivalent to ${\cal T}_Z^{\vphantom\#}$.%
\label{p:torsion-lfp}%
\end{prop}
\begin{proof}
We have ${\cal H}^\#_Z\subseteq{\cal T}_Z^{\vphantom\#}[0]$ by Proposition~\ref{p:heart_is_bounded}; on the other hand, by a similar argument to that of Example~\ref{e:HRS-heart&filtration}, we obtain ${\cal T}_Z^{\vphantom\#}[0]\subseteq{\cal H}^\#_Z$. \qedhere
\end{proof}

Consequently, we get that for any Thomason subset $Z\ne\emptyset$, its torsion class ${\cal T}_Z$ is a locally finitely presented Grothendieck category by Proposition~\ref{p:heart_is_bounded} (see however Theorem~\ref{t:when_hearts_coincide}, in which it is proved that ${\cal H}^\#_Z$ coincides with the heart ${\cal H}_Z^{\vphantom\#}$ of the compactly generated t-structure associated with $Z$, and hence it is a locally finitely presented Grothendieck category by \cite[Theorem~8.31]{SS20}). Now, the following result completely characterises the local coherence of ${\cal H}^\#_Z$; that is, by \cite[Theorem~2.2]{GP08}, it classifies the local coherence of any hereditary torsion class of finite type in $R\lMod$.

\begin{thm}
Let $Z$ be a nonempty Thomason subset. The following statements are equivalent:%
\label{t:comm_non-coherent}%

\begin{enumerate}
\item[(a)] The torsion class ${\cal T}_Z$ is a locally coherent Grothendieck category; that is, ${\cal T}_Z\cap R\lmod$ is an exact abelian subcategory of\/ ${\cal T}_Z$.

\item[(b)] $(J:\gamma)$ is a finitely generated ideal for every $J\in{\cal B}_Z$ and for all\/ $\gamma\in R$;

\item[(c)] $R/J$ is a coherent commutative ring for every $J\in{\cal B}_Z$.
\end{enumerate}
\end{thm}
\begin{proof}
Let us recall that ${\cal B}_Z$ is the family of finitely generated ideals in the Gabriel filter associated with the Thomason subset $Z$.

\noindent``${\rm(a)\Rightarrow(b)}$''\enspace For every $J\in{\cal B}_Z$ and for all $\gamma\in R$, the ideal $J+R\gamma$ is in ${\cal B}_Z$ hence $R/(J+R\gamma)$ is a finitely presented (torsion) module (see Proposition~\ref{p:torsion-Thomason_subset}). In turn, $(J+R\gamma)/J\cong R\gamma/(J\cap R\gamma)$ is so, being the kernel of the epimorphism $R/J\to R/(J+R\gamma)$ in ${\cal T}_Z\cap R\lmod$. The conclusion follows from the short exact sequence $0\to (J:\gamma)\to R\to R\gamma/(J\cap R\gamma)\to0$.

\noindent``${\rm(b)\Rightarrow(a)}$''\enspace 
Let $f\colon M\to M'$ be a $R$-linear map in ${\cal T}_Z\cap R\lmod$. By the well-known closure properties of this latter class of modules, we only need to verify that $\ker f$ is a finitely presented module, and clearly it suffices to consider $f$ as an epimorphism. Furthermore, from the following commutative diagram with exact rows
\[
\xymatrix{%
	0 \ar[r] & K \ar@{->>}[d]\ar[r] &
		(R/J)^n \ar[r]\ar@{->>}[d]^\alpha & M' \ar@{=}[d]\ar[r] & 0 \cr
	0 \ar[r] & \ker f \ar[r] & M \ar[r]^f & M' \ar[r] & 0
}
\]
in which the epimorphism $\alpha$ is provided by (the proof of) Proposition~\ref{p:torsion-Thomason_subset}(ii), we argue that a ``backward'' argument on the extension-closure of the finitely presented modules shows that our claim is equivalent to requiring that $\ker\alpha$ is finitely presented. Indeed, we have the following exact diagram:
\[
\xymatrix{
	0 \ar[r] & \ker(\alpha\circ\mu) \ar@{>->}[d]\ar[r] &
		(R/J)^{n-1} \ar[r]\ar@{>->}[d]^\mu &
		\im(\alpha\circ\mu) \ar@{>->}[d]\ar[r] & 0 \\
	0 \ar[r] & \ker\alpha \ar@{->>}[d]\ar[r] &
		(R/J)^n \ar[r]^-{\alpha}\ar@{->>}[d] &
		M \ar@{->>}[d]\ar[r] & 0 \\
	0 \ar[r] & C \ar[r] & R/J \ar[r] & C' \ar[r] & 0
}
\]
where $\mu$ is the canonical split monomorphism and the third exact row is given by the Snake~Lemma, so that $\ker\alpha$ is finitely presented if $C$ and $\ker(\alpha\circ\mu)$ are so. Now, once we prove that $C$ is finitely presented, we can repeat the previous argument for each $n\ge k\ge2$, achieving the validity at the base $k=2$. In other words, $\ker\alpha$ is finitely presented iff $C$ is finitely presented. Let us prove that $C$ is a finitely presented module. It is finitely generated for $\ker\alpha$ being so.
Consider now the pullback diagram
\[
\xymatrix{%
	0 \ar[r] & J \ar@{=}[d]\ar[r] & J' \ar@{>->}[d]\ar[r]\ar@{}[dr]|{\rm P.B.} &
		C \ar@{>->}[d]\ar[r] & 0 \\
	0 \ar[r] & J \ar[r] & R \ar@{->>}[d]\ar[r] &
		R/J \ar@{->>}[d]\ar[r] & 0 \\
	&& R/J' \ar@{=}[r] & R/J'
}
\]
in which $J'$ is a finitely generated ideal by extension closure, so that an element of ${\cal B}_Z$. Let us prove the claim by induction on the rank of $J'$. If $J'=R\gamma_1$, then it is $J'=J+R\gamma_1$, whence $C\cong J'\!/J\cong R\gamma_1/(J\cap R\gamma_1)$. We conclude by hypothesis~(b) applied on the short exact sequence $0\to (J:\gamma_1)\to R\to C\to0$. Notice that we just proved that for any ideal $J\in{\cal B}_Z$ and any $\gamma\in R$, then $(J+R\gamma)/J$ is finitely presented. This said, assume $J'=R\gamma_1+R\gamma_2$. Then again $J'=J+R\gamma_1+R\gamma_2$, with $J+R\gamma_1$ being in ${\cal B}_Z$, and from the exact commutative diagram
\[
\xymatrix{
	0 \ar[r] & J \ar@{=}[d]\ar[r] & J+R\gamma_1 \ar@{>->}[d]\ar[r] &
		(J+R\gamma_1)/J \ar@{>->}[d]\ar[r] & 0 \\
	0 \ar[r] & J \ar[r] & J' \ar@{->>}[d]\ar[r] &
		C \ar@{->>}[d]\ar[r] & 0 \\
	&& J'\!/(J+R\gamma_1) \ar@{=}[r] & J'\!/(J+R\gamma_1)
}
\]
we obtain, by the inductive base, that $(J+R\gamma_1)/J$ and $J'\!/(J+R\gamma_1)$ are finitely presented, hence $C$ is so by extension-closure. This argument clearly applies at every finite rank of $J'$, so $C$ is finitely presented.

\noindent``${\rm(b)\Rightarrow(c)}$''\enspace Let $J'\!/J$ be a finitely generated ideal of $R/J$ (so that $J'\!/J$ is a finitely generated module over $R$) and let us prove that it is finitely presented. $J'$ is in ${\cal G}_Z$, and by the short exact sequence $0\to J'\!/J\to R/J\to R/J'\to0$ in $R\lMod$ we deduce that $R/J'$ is a finitely presented $R$-module. By the hypothesis ``${\rm(b)\Leftrightarrow(a)}$'' we get that $J'\!/J$ is finitely presented over $R$, hence over $R/J$.

\noindent``${\rm(c)\Rightarrow(b)}$''\enspace Assume that $R/J$ is a coherent ring for each $J\in{\cal B}_Z$, and let $\gamma\in R$. By the short exact sequence $0\to(J:\gamma)\to R\to R\gamma/(J\cap R\gamma)\to0$ we shall prove that $R\gamma/(J\cap R\gamma)\cong (J+R\gamma)/J$ is a finitely presented $R$-module. $(J+R\gamma)/J$ is a finitely generated hence a finitely presented ideal of $R/J$, so there is a presentation $0\to K\to(R/J)^n\to (J+R\gamma)/J\to0$ with $n\in\N$ and $K$ a finitely generated $R/J$-module. Since the scalar restriction functor $R/J\lMod\to R\lMod$ is exact, and since $K$ is also a finitely generated $R$-module, such presentation lifts to $R\lMod$ so that $(J+R\gamma)/J$ is finitely presented, as desired. \qedhere
\end{proof}

\begin{cor}
Let\/ $R$ be a coherent commutative ring and\/ $Z$ be a Thomason subset. Then ${\cal T}_Z$ is a locally coherent Grothendieck category.%
\label{c:coherent_ring=>torsion_LC_Grothendieck}%
\end{cor}
\begin{proof}
It follows by the previous Theorem, since any factor ring $R/J$ is coherent for every finitely generated ideal $J$ (see \cite[(c) p.~143]{Lam99}). \qedhere
\end{proof}

As we will show in the next example, not all non-coherent commutative rings satisfy the equivalent conditions of Theorem~\ref{t:comm_non-coherent}.

\begin{exmpl}
In \cite[Appendix~A]{BP19} the second named author and D.~Bravo consider the ring $R\leqdef\Z\oplus (\Z/2\Z)^{(\N)}$, whose sum is componentwise and its multiplication is defined by%
\label{e:Vasconcelos}%
$$
	(m,a)\cdot (n,b)\leqdef(mn,mb+na+ab),
$$
where $ma\leqdef(ma_1,ma_2,\ldots)$ and $ab\leqdef(a_1b_1,a_2b_2,\ldots)$. In [op.~cit., Lemma~A.1] it is proved that $R$ is a commutative non-coherent ring, namely for the ideal generated by any $(2m,a)$ is finitely generated and not finitely presented. This fact entails at once a (somehow trivial) example of a Thomason subset whose torsion class is not a locally coherent Grothendieck category, namely $\Spec R$ itself, since the resulting torsion class is $R\lMod$.

Nonetheless, let us show that, over $R$ as above, there are proper Thomason subsets of $\Spec R$ and finitely generated ideals which do not satisfy Theorem~\ref{t:comm_non-coherent}(b). For instance, consider
$$
	J\leqdef R(0,e_1),\quad Z\leqdef V(J),\quad\hbox{and}\quad \gamma\leqdef(2,e_2)\in R,
$$
where $e_n$ is the standard basis vector of $(\Z/2\Z)^{(\N)}$, so that $J\in{\cal B}_Z$. We compute:
\begin{align*}
	J &= \{(m,a)\cdot(0,e_1)\mid (m,a)\in R\} \cr
	&= \{(0,(m+a_1)e_1)\mid (m,a)\in R\} \cr
	\noalign{\hbox{and}}
	(J:\gamma) &= \{(m,a)\in R\mid (m,a)(2,e_2)\in J\} \cr
	&= \{(m,a)\in R\mid (2m,me_2+ae_2)\in J\} \cr
	&= \{(m,a)\in R\mid (2m,(m+a_2)e_2)\in J\} \cr
	&= \mathop{\rm Ann}\nolimits_R(\gamma) \;.
\end{align*}
Now, out of the presentation $0\to\mathop{\rm Ann}\nolimits_R(\gamma)\to R\to R\gamma\to0$,
since $R\gamma$ is not finitely presented (\cite[Lemma~A.1]{BP19}), then $(J:\gamma)$ is not finitely generated, as claimed.
\end{exmpl}

\section{Arbitrary Thomason filtrations}%
\label{s:arbitrary_thomason_filtrations}%
Let $\Phi$ be any Thomason filtration of $\Spec R$, let $({\cal U},{\cal V})$ be the corresponding compactly generated t-structure with heart $\cal H$, and let ${\cal U}^\#$ be the induced AJS~aisle with heart ${\cal H}^\#$. For any $k\in\Z$ we define
$$
	\Phi_k(n)\leqdef \begin{cases}
		\Phi(k) &\hbox{if $n<k$} \cr
		\Phi(n) &\hbox{if $n\ge k$}.
	\end{cases}
$$
Thus, $\Phi_k$ is a weakly bounded below~$k$ Thomason filtration, naturally associated with $\Phi$. We will denote by $({\cal U}_k,{\cal V}_k)$ and ${\cal H}_k$, respectively, the corresponding compactly generated t-structure and the heart, and by ${\cal U}^\#_k$ and ${\cal H}^\#_k$ the induced AJS~aisle and heart. It is clear that at each degree in which $\Phi_k$ and $\Phi$ have the same Thomason subsets, namely for all $n\ge k$, their corresponding torsion pairs coincide as well; in this case we will denote these latter just as $({\cal T}_n,{\cal F}_n)$, i.e.~as those associated with $\Phi(n)$.

\begin{lem}
Let\/ $\Phi$ be a Thomason filtration of\/ $\Spec R$. Then ${\cal H}^\#_k\subseteq{\cal H}^\#_{\vphantom k}$.%
\label{l:inclusion-H_*}%
\end{lem}
\begin{proof}
Given $M\in{\cal H}^\#_k$, then clearly $M\in{\cal U}^\#_{\vphantom k}$ so it remains to prove that $M[-1]\in({\cal U}^\#)^{\bot_0}$. This follows immediately by applying the functor $\Hom_{\De(R)}(-,M[-1])$ on the approximation $\tau^{\le k-1}(U)\to U\to\tau^{>k-1}(U)\buildrel+\over\to$ of an arbitrary object $U\in{\cal U}^\#$ within the shifted standard t-structure of $\De(R)$, bearing in mind that ${\cal H}^\#_k\subseteq\De^{\ge k}(R)$, by Proposition~\ref{p:heart_is_bounded}, and that $\tau^{>k-1}(U)\in{\cal U}^\#_k$. \qedhere
\end{proof}

The following is the second main result of the paper.

\begin{thm}
Let\/ $\Phi$ be a Thomason filtration of\/ $\Spec R$. The following hold:%
\label{t:when_hearts_coincide}%

\begin{enumerate}
\item[(i)] For any compact object $S$ either in ${\cal U}$ or\/ ${\cal U}^\#$, we have $H_{{\cal H}^{\vphantom\#}}(S)=H_{{\cal H}^\#}(S)$;

\item[(ii)] If\/ $\Phi$ is weakly bounded below, then the heart ${\cal H}$ of the compactly generated t-structure associated with $\Phi$ coincides with the corresponding AJS~heart ${\cal H}^\#$.
\end{enumerate}
\end{thm}
\begin{proof}
Recall that by \cite[Lemma~3.6]{HHZ21} we have $\De^c(R)\cap{\cal U}=\De^c(R)\cap{\cal U}^\#$.

\noindent(i)\enspace Let $S\in\De^c(R)\cap{\cal U}$. Then there is an integer $k$ such that $S\in{\cal U}^\#_k$. Consequently, by the triangle
$$
	\tau^\le_{{\cal U}^\#_k}(S[-1])[1]\longrightarrow S\longrightarrow
		H_{{\cal H}^\#_k}(S)\buildrel+\over\longrightarrow
$$
we see that the first vertex belongs to $\De^+(R)$ for the other two do so (the third vertex belongs to $\De^+(R)$ by Proposition~\ref{p:heart_is_bounded}), moreover we deduce
\begin{align*}
	H_{{\cal H}^{\vphantom\#}_{\vphantom k}}(S) &=
		H_{{\cal H}^{\vphantom\#}_{\vphantom k}}(H_{{\cal H}^\#_k}(S)) \cr
	\noalign{\hbox{and}}
	H_{{\cal H}^\#_{\vphantom k}}(S) &=
		H_{{\cal H}^\#_{\vphantom k}}(H_{{\cal H}^\#_k}(S)),
\end{align*}
since ${\cal U}^\#_k\cap\De^+(R)\subseteq {\cal U}^\#_{\vphantom k}\cap\De^+(R)={\cal U}\cap\De^+(R)$ by [ibid.]. By Lemma~\ref{l:inclusion-H_*}, the object of the previous display coincides with $H_{{\cal H}^\#_k}(S)$, which is a complex of ${\cal H}^\#\cap\De^+(R)$, i.e.~of ${\cal H}\cap\De^+(R)$ by [ibid.]\ again. Consequently, the objects in the previous displays coincide, as desired.

\noindent(ii)\enspace Recall that, by \cite[Theorem~8.31]{SS20}, ${\cal S}\leqdef\mathop{\rm add}H_{\cal H}(\De^c(R)\cap{\cal U})$ is a set of (finitely presented) generators for $\cal H$, and by the previous part we obtain the equality
$$
	\mathop{\rm add}H_{{\cal H}^{\vphantom\#}}(\De^c(R)\cap{\cal U}^{\vphantom\#})=
	\mathop{\rm add}H_{{\cal H}^\#}(\De^c(R)\cap{\cal U}^\#) \;.
$$
Thanks to the hypothesis on $\Phi$ we have, by Proposition~\ref{p:heart_is_bounded}, that ${\cal H}^\#\subseteq{\cal H}$, whereas by the proof of the Proposition we see that $H_{{\cal H}^\#}({\cal S})$ is a set of generators for ${\cal H}^\#$. Therefore, we infer ${\cal H}={\cal H}^\#$ since $H_{{\cal H}^\#}({\cal S})={\cal S}$, ${\cal H}^\#$ is an exact abelian subcategory of $\cal H$, and the coproducts in ${\cal H}^\#$ are computed as in $\cal H$. \qedhere
\end{proof}

\begin{cor}
For any Thomason filtration $\Phi$ of\/ $\Spec R$, we have ${\cal H}_k^\#={\cal H}_k^{\vphantom\#}$ and\/ ${\cal H}_k^{\vphantom\#}\subseteq{\cal H}$.%
\label{c:inclusion-H_*}%
\end{cor}
\begin{proof}
Recall that ${\cal H}_k^\#$ and ${\cal H}_k^{\vphantom\#}$ are the hearts associated with the filtration $\Phi_k^{\vphantom\#}$, which is weakly bounded below $k$ (see the beginning of the present section). By the previous Theorem, we deduce the stated identity among the two hearts. In order to prove the stated inclusion, we will imitate the proof of Lemma~\ref{l:inclusion-H_*}. Let $M\in{\cal H}_k$, then $M\in{\cal U}_k\subseteq{\cal U}$, since the compact objects generating ${\cal U}_k$ belong to $\cal U$. Let us prove that $M[-1]\in {\cal U}^{\bot_0}$. This follows immediately by applying the functor $\Hom_{\De(R)}(-,M[-1])$ on the approximation $\tau^{\le k-1}(U)\to U\to\tau^{>k-1}(U)\buildrel+\over\to$ of an arbitrary object $U\in{\cal U}$ within the shifted standard t-structure of $\De(R)$, bearing in mind that ${\cal H}_k^{\vphantom\#}={\cal H}_k^\#\subseteq\De^{\ge k}(R)$, by Proposition~\ref{p:heart_is_bounded}, and that $\tau^{>k-1}(U)\in{\cal U}^\#_k\cap\De^+(R)\subseteq{\cal U}_k^{\vphantom\#}$, by \cite[Lemma~3.6]{HHZ21}. \qedhere
\end{proof}

\begin{cor}
Let $\Phi$ be a Thomason filtration. For every $M\in{\cal H}^\#$ the following assertions hold:%
\label{c:approximation-H_*}%

\begin{enumerate}
\item[(i)] there exists in ${\cal H}^\#$ a short exact sequence $0\to A\to M\to B\to0$ with $A\in{}^{\bot_0}_{\vphantom k}({\cal H}^\#_k)$ and\/ $B\in{\cal H}^\#_k$ (the orthogonal being computed w.r.t.~${\cal H}^\#_{\vphantom k}$);

\item[(ii)] there exists in ${\cal H}^\#$ a short exact sequence $0\to A\to M\to B\to0$ with $A\in{\cal H}^\#_k$ and\/ $B\in({\cal H}^\#_k)^{\bot_0}_{\vphantom k}$ (the orthogonal being computed w.r.t.~${\cal H}^\#$).
\end{enumerate}
\end{cor}
\begin{proof}\mbox{}\vglue0pt%
\noindent(i)\enspace Let $M\in{\cal H}^\#$, and consider the octahedron:
\[
\xymatrix{%
	\tau^{\le k-1}(M) \ar@{=}[d]\ar[r] & A \ar[d]\ar[r] & U[1] \ar[d]\ar[r]^-{+} & {} \cr
	\tau^{\le k-1}(M) \ar[r] & M \ar[d]\ar[r] & \tau^{>k-1}(M) \ar[d]\ar[r]^-{+} & {} \cr
	& H_{{\cal H}^\#_k}(\tau^{>k-1}(M)) \ar[d]_-{+}\ar@{=}[r] &
		H_{{\cal H}^\#_k}(\tau^{>k-1}(M)) \ar[d]^-{+} \cr
	&{} & {}
}
\]
provided by $U\leqdef\tau^\le_{{\cal U}^\#_k}(\tau^{>k-1}_{\vphantom{\cal U}}(M)[-1])$ and a cone $A$ (notice that $\tau^{>k-1}(M)\in{\cal U}^\#_k$). Since $B\leqdef H_{{\cal H}^\#_k}(\tau^{>k-1}(M))$ actually is in ${\cal H}^\#_k$, hence in ${\cal H}^\#_{\vphantom k}$ by Lemma~\ref{l:inclusion-H_*}, we only have to check that $A$ belongs to ${\cal H}^\#_{\vphantom k}$ and that it is left orthogonal to ${\cal H}^\#_k$ in ${\cal H}^\#_{\vphantom k}$. From the first vertical triangle we see that $A\in({\cal U}^\#)^{\bot_0}[1]$, whereas by the first horizontal one we deduce that $A\in{\cal U}^\#$. Moreover, using once again the first horizontal triangle, we infer that $A\in{}^{\bot_0}_{\vphantom k}({\cal H}^\#_k)$ since ${\cal H}^\#_k\subseteq \De^{\ge k}(R)$, as desired. Thus, the first vertical triangle yields the stated short exact sequence of ${\cal H}^\#$.

\noindent(ii)\enspace Consider the approximation $A\to M\to B\buildrel+\over\to$ of $M$ within the t-structure $({\cal U}^\#_k,({\cal U}^\#_k)^{\bot_0}_{\vphantom k}[1])$, thus surely $B$ is right orthogonal to ${\cal H}^\#_k$ in ${\cal H}^\#_{\vphantom k}$. It remains to check that $A\in({\cal U}^\#_k)^{\bot_0}_{\vphantom k}[1]$ and that $B\in{\cal H}^\#_{\vphantom k}$. The first claim holds true by extension-closure of the coaisle applied on the rotated triangle $B[-2]\to A[-1]\to M[-1]\buildrel+\over\to$, and since ${\cal U}^\#_k\subseteq{\cal U}^\#_{\vphantom k}$. On the other hand, $B$ belongs to the aisle ${\cal U}^\#_{\vphantom k}$ in view of the rotated triangle $M\to B\to A[1]\buildrel+\over\to$, while for every $U\in{\cal U}^\#$, by the approximation
$$
	\tau^{\le k-1}(U)\longrightarrow U\longrightarrow \tau^{>k-1}(U)\buildrel+\over\longrightarrow,
$$
we have $\tau^{>k-1}(U)\in{\cal U}^\#_k$, whence $\Hom_{\De(R)}(\tau^{>k-1}(U),B[-1])=0$. Therefore, once we show that $\Hom_{\De(R)}(\tau^{\le k-1}(U),B[-1])=0$, we conclude the proof. Our claim follows at once by applying the covariant hom functor of $\tau^{\le k-1}(U)$ on the triangle $M[-1]\to B[-1]\to A\buildrel+\over\to$, bearing in mind that $A\in{\cal H}^\#_k\subseteq\De^{\ge k}(R)$. \qedhere
\end{proof}

\begin{cor}
Let\/ $\Phi$ be a Thomason filtration. Then the heart ${\cal H}^\#_k$ is closed in ${\cal H}^\#_{\vphantom k}$ under taking products and coproducts.%
\label{c:closure_by_products_and_coproducts}%
\end{cor}
\begin{proof}
Let $(M_i)_{i\in I}$ be a family of objects of ${\cal H}^\#_k$ with product $(\prod_{i\in I}M_i,(\pi_i)_{i\in I})$ in ${\cal H}^\#_k$. We have to prove that such pair satisfies the universal property of the product in ${\cal H}^\#$. So, let $M\in{\cal H}^\#$ and $(f_i)_{i\in I}$ be a family of morphisms $f_i\colon M\to M_i$ in ${\cal H}^\#$. By Corollary~\ref{c:approximation-H_*}(i) we obtain the following commutative diagram,
\[
\xymatrix{%
	A \ar[d]\ar[r]^\alpha & M \ar[d]_{f_i}\ar[r]^\beta & B \ar@{.>}[d]^{g_i}\ar[r]^-{+} & {} \cr
	0 \ar[r] & M_i \ar@{=}[r] & M_i \ar[r]^-{+} & {}
}
\]
hence a family of morphisms $g_i\colon B\to M_i$ in ${\cal H}^\#_k$ inducing a unique morphism $g\colon B\to\prod_{i\in I}M_i$ such that $\pi_i\circ g=g_i$ for all $i\in I$. The composition $g\circ\beta$ yields the existence of a morphism $M\to\prod_{i\in I}M_i$ in ${\cal H}^\#$ such that $\pi_i\circ(g\circ\beta)=f_i$ for all $i\in I$. Uniqueness of $g\circ\beta$ w.r.t.~the latter property is a byproduct of the construction of the triangle made in Corollary~\ref{c:approximation-H_*}, namely for both $A$ and $B$ are uniquely determined up to isomorphism, together with the fact that $\beta$ is an epimorphism in ${\cal H}^\#$.

The proof concerning the coproduct is dual. \qedhere
\end{proof}

\begin{thm}
For any Thomason filtration $\Phi$ of\/ $\Spec R$ and for any integer $k$, the following assertions hold true:%
\label{t:H_*-TTF_finite_type}%

\begin{enumerate}
\item[(i)] the heart\/ ${\cal H}^\#_k$ ($\empty=\mathcal{H}_k^{\vphantom\#}$, see Theorem~\ref{t:when_hearts_coincide}) is a TTF class in ${\cal H}^\#_{\vphantom k}$;
	
\item[(ii)] If\/ $\mathcal{H}^{\vphantom\#}=\mathcal{H}^{\#}$ (e.g.~when $R$ is noetherian or $\Phi$ is weakly bounded below, see Theorem~\ref{t:when_hearts_coincide}), then ${\cal H}^\#_k$ is a TTF class of finite type in ${\cal H}^\#_{\vphantom k}$.
\end{enumerate}
\end{thm}
\begin{proof}
Recall that ${\cal H}_k^\#$ and ${\cal H}_k^{\vphantom\#}$ are the hearts associated with the filtration $\Phi_k$, which is weakly bounded below $k$ (see the beginning of the present section).

\noindent(i)\enspace In order to prove that ${\cal H}^\#_k$ is a TTF~class in ${\cal H}^\#_{\vphantom k}$, by Corollary~\ref{c:closure_by_products_and_coproducts} we only have to show that the former heart is closed under subobjects, quotient objects and extensions. The closure under extensions is obvious since the relevant aisle and coaisle fulfil it. So, let $0\to L\to M\to N\to0$ be a short exact sequence in ${\cal H}^\#_{\vphantom k}$ with $M\in{\cal H}^\#_k$. Clearly, $L$ and $N$ belong to $({\cal U}^\#_k)^{\bot_0}_{\vphantom k}[1]$. By Proposition~\ref{p:heart_is_bounded}, applied on the Thomason filtration $\Phi_k$, we deduce $H^{j}(M)=0$ for all $j<k-1$. Thus we have $H^j(N)\cong H^{j+1}(L)\in{\cal T}_{j+1}\ (\ast)$ for all $j<k-2$, and that $H^{k-1}(N)$ is a submodule of $H^k(L)$, i.e.~it belongs to ${\cal T}_k$. Moreover, it follows $\tau^{\le k-1}(N)\in{\cal U}^\#[1]$ and consequently, by the usual argument of the proof of Proposition~\ref{p:heart_is_bounded}, that $\tau^{\le k-1}(N)=0$. By $(\ast)$, we infer $\tau^{\le k-1}(L)=0$ as well. Therefore, $N,L\in {\cal H}^\#_{\vphantom k}\cap \De^{\ge k}(R)\subseteq{\cal U}^\#_{\vphantom k}\cap \De^{\ge k}(R)\subseteq{\cal U}^\#_k$, and this concludes the proof (as stated, the equality ${\cal H}^\#_k={\cal H}_k^{\vphantom\#}$ follows by Theorem~\ref{t:when_hearts_coincide}).

\noindent(ii)\enspace Suppose that ${\cal H}^\#={\cal H}$. We have to prove that the class $({\cal H}^\#_k)^{\bot_0}_{\vphantom k}$, right orthogonal of ${\cal H}_k^\#$ in ${\cal H}^\#_{\vphantom k}(\empty={\cal H})$, is closed under taking direct limits of the latter heart. We claim that
$$
	{\cal H}_k^{\bot_0}={\cal U}_k^{\bot_0}\cap{\cal H}_{\vphantom k},
$$
where the orthogonal of ${\cal U}_k$ is computed in $\De(R)$, so that the inclusion ``$\supseteq$'' is clear. 
Let $M\in{\cal H}_k^{\bot_0}$ and apply the functor $\Hom_{\De(R)}(-,M)$ on the exact triangle $\tau^\le_{{\cal U}_k}(U[-1])[1]\to U\to H_{{\cal H}_k}(U)\buildrel+\over\to$ associated with an arbitrary object $U\in{\cal U}_k$, to get $\Hom_{\De(R)}(U,M)=0$ since ${\cal U}_k^{\vphantom{\bot_0}}\subseteq{\cal U}$ and ${\cal H}_k^{\bot_0}\subseteq{\cal H}$. This proves our claim. Now, for any direct system $(M_i)_{i\in I}$ in ${\cal H}_k^{\bot_0}$, by \cite[Corollary~5.8]{SSV17} we have the natural isomorphism
$$
	\varinjlim_{i\in I}\adjust M_i
	\cong\varhocolim_{i\in I} M_i;
$$ 
moreover, since $({\cal U}_k,{\cal V}_k)$ is a compactly generated hence homotopically smashing t-structure of $\De(R)$, by \cite[Corollary~5.6]{SSV17} we infer that $\varinjlim_{i\in I}^{({\cal H})}M_i\in{\cal H}\cap{\cal U}_k^{\bot_0}={\cal H}_k^{\bot_0}$. Therefore, the class ${\cal H}_{ k}^{\bot_0}$ is closed under direct limits taken in $\cal H$. Eventually, we conclude by our assumption together with the identities ${\cal H}_k^\#={\cal H}_k^{\vphantom\#}$ and 
\begin{align*}
	({\cal H}_k^{\#})^{\bot_0}_{\vphantom k} &=
		\{M\in {\cal H}^\#_{\vphantom k}\mid \Hom_{\De(R)}({\cal H}_k^{\#},M)=0\} \cr
	&= \{M\in {\cal H}\mid \Hom_{\De(R)}({\cal H}_k^{\vphantom{\bot_0}},M)=0\}={\cal H}_k^{\bot_0} \;. \qedhere
\end{align*}
\end{proof}

\begin{cor}
Let\/ $\Phi$ be a Thomason filtration bounded above~$r$ and weakly bounded below. Then ${\cal T}_r[-r]$ is a TTF~class of finite type in ${\cal H}^\#$.%
\label{c:T_0-TTF_in_H}
\end{cor}
\begin{proof}
Thanks to the boundedness of $\Phi$, we have ${\cal H}^\#_r={\cal T}_r^{\vphantom\#}[-r]$, so the conclusion follows by the previous Theorem. Notice that in this case the left constituent of the TTF~triple is $(\tau^{\le r-1}({\cal H}^\#), {\cal T}_r[-r])$, for there are no nonzero morphisms between the members of the pair and, by Corollary~\ref{c:approximation-H_*}(i), for every $M\in{\cal H}^\#$ its standard approximation $\tau^{\le r-1}(M)\to M\to H^r(M)[-r]\buildrel+\over\to$ yields a functorial short exact sequence in ${\cal H}^\#$. \qedhere
\end{proof}

\begin{rem}
The existence in ${\cal H}^\#$ of TTF~triples carries useful information, both on the members of the triples and on the local coherence of ${\cal H}^\#$ itself. More precisely:%
\label{r:lc-TTF_finite_type}%

\begin{enumerate}
\item By Corollary~\ref{c:approximation-H_*}(i), the torsion class ${}^{\bot_0}({\cal H}^\#_k)$ consists of those complexes $M$ of ${\cal H}^\#_{\vphantom k}$ which fit in an exact triangle $\tau^{\le k-1}(M)\to M\to U[1]\buildrel+\over\to$ for some object $U\in{\cal U}^\#_k$.

\item The torsion class ${\cal H}^\#_k$ is a locally finitely presented category by Proposition~\ref{p:heart_is_bounded}. Moreover, if $\Phi$ is weakly bounded below, then we have $\fp({}^{\bot_0}_{\vphantom k}({\cal H}^\#_k)),\fp({\cal H}^\#_k)\subseteq\fp({\cal H}^\#_{\vphantom k})$ by \cite[Lemma~2.4]{PSV19}. Furthermore, by Theorem~\ref{t:lc-TTF_finite_type}, both  ${}^{\bot_0}_{\vphantom k}({\cal H}^\#_k)$ and ${\cal H}^\#_k$ are quasi locally coherent categories in case ${\cal H}^\#$ is locally coherent.

\item If $\Phi$ is weakly bounded below, then thanks to Theorem~\ref{t:when_hearts_coincide} and \cite[Lemma~3.6]{HHZ21}, for any module $X\in{\cal T}_j$ the triangle
$$
	U[1]\longrightarrow X[-j]\longrightarrow H_{\cal H}(X[-j])\buildrel+\over\longrightarrow
$$
can be taken either with $U\in{\cal U}$ or $U\in{\cal U}^\#$, since both $X[-j]$ and $H_{\cal H}(X[-j])$ belong to $\De^+(R)$.

\item In order to distinguish the torsion radicals and coradicals of each torsion pair $({}^{\bot_0}_{\vphantom k}({\cal H}^\#_k),{\cal H}^\#_k)$ of ${\cal H}^\#_{\vphantom k}$ to those of each torsion pair $({\cal T}_k,{\cal F}_k)$ of $R\lMod$ we dealt with so far, we will use the following notation
$$
	{}^{\bot_0}_{\vphantom k}({\cal H}^\#_k)\coreflective[{\bf x}_k]
		{\cal H}^\#_{\vphantom k}\reflective[{\bf y}_k]{\cal H}^\#_k;
$$
furthermore, we will drop the index in case the value of the integer is clear from the context.
\end{enumerate}
\end{rem}

We conclude this section with another remark that, together with Remark~\ref{r:lc-TTF_finite_type}(3), will be frequently used in the subsequent section.

\begin{rem}
Let $\Phi$ be any Thomason filtration, and $k\in\Z$. Then the composition $H^{-k}\circ H_{\cal H}\circ[k]$ defines a functor%
\label{r:functorial_monomorphism_sigma}%
\begin{align*}
	\Sigma^{-k}\colon \TF_{\!-k} &\longrightarrow \TFT_{\!-k} \cr
	X &\longmapsto H^{-k}(H_{\cal H}(X[k]))
\end{align*}
equipped with a functorial monomorphism $\sigma\colon\id\to\Sigma^{-k}$ such that $\coker\sigma_X\in{\cal T}_{-k+2}$. Indeed, for every $X\in\TF_{\!-k}$, i.e.\ $X\in{\cal T}_{-k}\cap{\cal F}_{-k+1}$, its stalk $X[k]$ is an object of ${\cal U}\cap{\cal U}_{-k}$ (see Remark~\ref{r:lc-TTF_finite_type}(3)), hence $H_{{\cal H}_{\vphantom{-k}}}(X[k])\cong H_{{\cal H}_{-k}}(X[k])$, so that the least nonzero cohomology of the latter complex is at degree $-k$ by Proposition~\ref{p:heart_is_bounded}. Therefore, by Lemma~\ref{l:least_nonzero_cohomology}, $\Sigma^{-k}$ is well-defined on objects. Let now $f\colon X\to X'$ be a morphism in $\TF_{\!-k}$. Then we have a diagram
\[
\xymatrix{%
	U[1] \ar[r]\ar@{.>}[d] & X[k] \ar[d]^-{f[k]}\ar[r] & H_{\cal H}(X[k]) \ar[r]^-{+}\ar@{.>}[d]^h & \cr
	U'[1] \ar[r] & X'[k] \ar[r] & H_{\cal H}(X'[k]) \ar[r]^-{+} & 
}
\]
for some $U,U'\in{\cal U}$, which can be completed to a morphism of triangles since the composition $U[1]\to X[k]\buildrel f[k]\over\to X'[k]\to H_{\cal H}(X'[k])$ is the zero map. We have $h=H_{\cal H}(f[k])$, hence $\Sigma^{-k}$ actually is a functor. This said, apply the standard cohomology $H^{-k}$ on the first triangle of the previous commutative diagram, to obtain the exact sequence
$$
	0\longrightarrow H^{-k}(U[1])\longrightarrow X\buildrel\sigma_X\over\longrightarrow\Sigma^{-k}(X)
		\longrightarrow H^{-k+1}(U[1])\longrightarrow0
$$
in which $H^{-k}(U[1])\in{\cal T}_{-k+1}\cap{\cal F}_{-k+1}=0$ by assumption on $X$. Therefore, the $\sigma_X$'s are monomorphisms, moreover they form a natural transformation in view of the construction of the functor $\Sigma^{-k}$. Finally, $\coker\sigma_X\in{\cal T}_{-k+2}$ being isomorphic to $H^{-k+1}(U[1])$.
\end{rem}

\begin{lem}
Let\/ $\Phi$ be a Thomason filtration, $k\in\Z$ and\/ $X\in{\cal T}_{-k}$. Consider the following assertions:%
\label{l:SigmaT}%

\begin{enumerate}
\item[(a)] $H_{\cal H}(X[k])\in\fp({\cal H})$;

\item[(b)] $H_{\cal H}(y_{-k+1}(X)[k])\in\fp({\cal H})$;

\item[(c)] $H_{\cal H}(\Sigma^{-k}(y_{-k+1}(X))[k])\in\fp({\cal H})$;

\item[(d)] $\Sigma^{-k}(y_{-k+1}(X))\in\fp(\TFT_{\!-k})$.
\end{enumerate}

\noindent Then ``${\rm(a)\Leftrightarrow(b)\Leftrightarrow(c)\Rightarrow(d)}$''; moreover, if\/ $\Phi$ is weakly bounded below, then all the assertions are equivalent. In this case, the subclass of\/ ${\cal T}_{-k}$ of modules satisfying the previous equivalent conditions will be denoted by $\Sigma{\cal T}_{-k}$.
\end{lem}
\begin{proof}\mbox{}\vglue0pt%
\noindent``${\rm(a)\Leftrightarrow(b)}$''\enspace Consider the approximation $0\to x_{-k+1}(X)\to X\to y_{-k+1}(X)\to0$ of $X$ within the torsion pair $({\cal T}_{-k+1},{\cal F}_{-k+1})$ of $R\lMod$. Then $x_{-k+1}(X)[k]\in{\cal U}[1]$ hence applying the functor $H_{\cal H}$ on the triangle involving the stalk complexes of the sequence, we obtain $H_{\cal H}(X[k])\cong H_{\cal H}(y_{-k+1}(X)[k])$, and we are done.

\noindent``${\rm(b)\Leftrightarrow(c)}$''\enspace Since $y_{-k+1}(X)\in\TF_{\!-k}$, in view of Remark~\ref{r:functorial_monomorphism_sigma} we have a short exact sequence
$$
	0\longrightarrow y_{-k+1}(X)\longrightarrow \Sigma^{-k}(y_{-k+1}(X))\longrightarrow
		\coker\sigma_{y_{-k+1}(X)}\longrightarrow0
$$
say it $0\to Y\to S\to C\to0$ for short, in which $C\in{\cal T}_{-k+2}$. By applying the functor $H_{\cal H}$ on the triangle involving the stalk complexes of such sequence, we obtain $H_{\cal H}(Y[k])\cong H_{\cal H}(S[k])$, whence the thesis.

\noindent``${\rm(c)\Rightarrow(d)}$''\enspace The proof is similar to that of Theorem~\ref{t:when_hearts_coincide}. In view of the previous notation, we have $S[k]\in{\cal U}^\#_{-k}$, hence $H_{{\cal H}\vphantom{{}^\#_{-k}}}(S[k])=H_{{\cal H}^\#_{-k}}(S[k])$. We now apply Lemma~\ref{l:fp(H)&TFT} on the heart ${\cal H}^\#_{-k}$, obtaining that
$
	H^{-k}(H_{{\cal H}^\#_{-k}}(S[k]))\in\fp(\TFT_{\!-k})
$, but this latter object coincides with $S$ by Lemma~\ref{l:on_length_l+1}(ii). Notice that, even if our filtration $\Phi_{-k}$ is not of finite length, the arguments of the cited forthcoming results still hold true for the cohomologies involved in the present proof.

If $\Phi$ is weakly bounded below, we have ${\cal H}^\#_{-k}={\cal H}_{-k}^{\vphantom\#}$ and ${\cal H}^\#_{\vphantom k}={\cal H}$, by Theorem~\ref{t:when_hearts_coincide}. By Remark~\ref{r:lc-TTF_finite_type}(2) we have that $\fp({\cal H}_{-k})\subseteq\fp({\cal H})$, eventually by Lemma~\ref{l:fp(H)&TFT} applid on ${\cal H}_{-k}$ we infer that ``${\rm(d)\Rightarrow(c)}$''. \qedhere
\end{proof}

\section{Thomason filtrations of finite length}%
\label{s:thomason_filtrations_of_finite_length}%
The present section is devoted to deepen the approximation theory of the AJS~heart ${\cal H}^\#$ associated with a Thomason filtration $\Phi$ of finite length (see Definition~\ref{d:bounded_Thomason_filtrations}), in order to characterise its local coherence. In this vein, the main tools we have at our disposal are given by the TTF~classes of finite type ${\cal H}^\#_k$ detected in Theorem~\ref{t:H_*-TTF_finite_type}, for they allow to specialise Theorem~\ref{t:lc-TTF_finite_type}, and by Theorem~\ref{t:when_hearts_coincide}, which ensures that all these hearts coincide with those of the compactly generated t-structures corresponding to $\Phi$; in view of this latter result, we will make the identifications
$$
	{\cal H}={\cal H}^\#_{\vphantom k}\qquad\hbox{and}\qquad
	{\cal H}_k^{\vphantom\#}={\cal H}^\#_k\quad \hbox{for all $k\in\Z$},
$$
so that, within ${\cal H}={\cal H}^\#$, we also have ${}^{\bot_0}({\cal H}^\#_k)={}^{\bot_0}_{\vphantom k}{\cal H}_k$ and $({\cal H}^\#_k)^{\bot_0}_{\vphantom k}={\cal H}_k^{\bot_0}$. Bearing in mind Remark~\ref{r:lc-TTF_finite_type}, it is then natural to seek for a recursive argument, namely a result which takes in account the local coherence of each heart ${\cal H}_k$. Therefore, we set $l+1$ to be the length of $\Phi$.

\begin{lem}
Let\/ $\Phi$ be a Thomason filtration of length~$l+1$. Then:%
\label{l:on_length_l+1}%

\begin{enumerate}
\item[(i)] For every $X\in{\cal T}_{-l-1}$, we have $H_{\cal H}(X[l+1])\in{}^{\bot_0}{\cal H}_{-l}$;

\item[(ii)] For every $X\in\TFT_{\!-l-1}$ there exist $U\in{\cal U}_{-l+2}$ and a triangle $U[1]\to X[l+1]\to H_{\cal H}(X[l+1])\buildrel+\over\to$. In particular, $H^{-l-1}(H_{\cal H}(X[l+1]))=X$.

\item[(iii)] for all\/ $M\in{}^{\bot_0}{\cal H}_{-l}$, there exists in $\cal H$ a functorial short exact sequence $0\to L\to W\to M\to0$, in which  $L\in{\cal H}_{-l+1}$ and\/ $W\cong H_{\cal H}(X[l+1])$, where $X=H^{-l-1}(M)$;

\item[(iv)] ${}^{\bot_0}{\cal H}_{-l}=\mathop{\rm Gen}(H_{\cal H}(K(J)[l+1])\mid J\in{\cal B}_{-l-1})$.
\end{enumerate}
\end{lem}
\begin{proof}
We will often exploit the characterisation of the torsion class ${}^{\bot_0}{\cal H}_{-l}$ deduced from Corollary~\ref{c:approximation-H_*} (see Remark~\ref{r:lc-TTF_finite_type}(1)).

\noindent(i)\enspace Given $X\in{\cal T}_{-l-1}$, let $M\leqdef H_{\cal H}(X[l+1])$ and consider the exact triangle $U[1]\to X[l+1]\to M\buildrel+\over\to$ given by some object $U\in{\cal U}$ (see Remark~\ref{r:lc-TTF_finite_type}(3)). Let us show that $M$ satisfies the aforementioned characterisation of the torsion class ${}^{\bot_0}{\cal H}_{-l}$. Applying the standard cohomology on the above triangle we obtain $H^j(M)\cong H^{j+2}(U)$ for all $j\ge-l+1$ and that $H^{-l}(M)\le H^{-l+2}(U)$, where the former two are modules in the torsion class ${\cal T}_{j+2}$. We claim that $\tau^{\ge-l}(M)[-1]\in{\cal U}_{-l}$, whence the conclusion thanks to the triangle
$$
	H^{-l-1}(M)[l+1]\longrightarrow M\longrightarrow \tau^{\ge-l}(M)\buildrel+\over\longrightarrow \;.
$$
Indeed, we have
$$
	H^j(\tau^{\ge-l}(M)[-1])=H^{j-1}(\tau^{\ge-l}(M))=
	\begin{cases}
		0 &\hbox{if $j-1<-l$}, \cr
		H^{j-1}(M) &\hbox{if $j-1\ge-l$,}
	\end{cases}
$$
hence, when $j-1\ge-l$, we have $H^{j-1}(M)\cong H^{j+1}(U)\in{\cal T}_{j+1}\subseteq{\cal T}_j$, as desired.

\noindent(ii)\enspace Let $X\in\TFT_{\!-l-1}$ and $U[1]\to X[l+1]\to M\buildrel+\over\to$ as in part~(i). The long exact sequence in standard cohomology yields
$$
	0\longrightarrow H^{-l}(U)\longrightarrow X\longrightarrow
		H^{-l-1}(M)\longrightarrow H^{-l+1}(U)\longrightarrow0
$$
in which in fact $H^{-l}(U)=0$ for it belongs simultaneously to ${\cal T}_{-l}$ and ${\cal F}_{-l}$ by assumption on $X$. Moreover, the resulting extension of $H^{-l-1}(M)$ is split by assumption on $X$ again, meaning that $H^{-l+1}(U)=0$ as well. Consequently, $U\in\De^{\ge-l+2}(R)\cap{\cal U}$, as desired.

\noindent(iii)\enspace Let $M\in{}^{\bot_0}{\cal H}_{-l}$, so that by Proposition~\ref{p:heart_is_bounded} and Remark~\ref{r:lc-TTF_finite_type}(1) there exists $U\in{\cal U}_{-l}$ and an exact triangle $H^{-l-1}(M)[l+1]\to M\to U[1]\buildrel+\over\to$, in which we set $X\leqdef H^{-l-1}(M)$. The long exact sequence in standard cohomology yields in particular $U\in{\cal U}_{-l}\cap\De^{[-l+1,0]}(R)\subseteq{\cal U}_{-l+1}$. The usual triangle of $U$ w.r.t.\ the heart of $({\cal U}_{-l+1},{\cal V}_{-l+1})$ gives the following octahedron
\[
\xymatrix{%
	U'[1] \ar@{=}[r]\ar[d] & U'[1] \ar[d] \cr
	U \ar[r]\ar[d] & X[l+1] \ar[r]\ar[d] & M \ar[r]^-{+}\ar@{=}[d] & \cr
	L \ar[r]\ar[d]_+ & W \ar[d]^+\ar[r] & M \ar[r]^-{+} & \cr
	&
}
\]
for some $U'\in{\cal U}_{-l+1}$, so that $L\cong H_{{\cal H}_{-l+1}}(U)$, and a cone $W$, which actually belongs to $\cal H$ by extension-closure applied on the second horizontal triangle. Applying the t-cohomological functor $H_{\cal H}$ on the second vertical triangle we obtain $W\cong H_{\cal H}(X[l+1])$, hence the former triangle is a functorial short exact sequence of $\cal H$; indeed, it is the image under $H_{\cal H}$ of the first horizontal triangle, which is in turn functorial.

\noindent(iv)\enspace Let $M$ and $X$ be as in part~(iii). By Proposition~\ref{p:torsion-Thomason_subset} we know that there exist a family $(J_i)_{i\in I}$ of finitely generated ideals in  the Gabriel filter associated with the torsion class ${\cal T}_{-l-1}$, and an epimorphism $\varphi\colon\bigoplus_{i\in I}(R/J_i)^{(\alpha_i)}\to X$. Applying $H_{\cal H}$ on the associated triangle of the stalk complexes concentrated in degrees $-l-1$, bearing in mind that it commutes with coproducts of $\De(R)$, we obtain the exact sequence of $\cal H$
$$
	H_{\cal H}(\ker(\varphi)[l+1])\longrightarrow \bigoplus_{i\in I}H_{\cal H}(R/J_i[l+1])^{(\alpha_i)}\longrightarrow
		\smash{\overbrace{H_{\cal H}(X[l+1])}^{\textstyle\empty\cong W}}\longrightarrow0 \;.
$$
Thus, our claim follows once we prove that $H_{\cal H}(K(J)[l+1])\cong H_{\cal H}(R/J[l+1])$ for all $J\in{\cal B}_{-l-1}$, since $M$ is an epimorphic image of $W$ in $\cal H$. Shifting by $l+1$ the standard approximation $\tau^{\le-1}(K(J))\to K(J)\to R/J[0]\buildrel+\over\to$ of the Koszul complex $K(J)$, we see that $\tau^{\le-1}(K(J))[l+1]=(\tau^{\le-1}(K(J)[l]))[1]$ is an object of the aisle ${\cal U}[1]$. Therefore, applying the functor $H_{\cal H}$ on the resultin triangle, we conclude. \qedhere
\end{proof}

\begin{rem}\mbox{\label{r:split_in_fp(TFT)}}\vglue0pt%
\begin{enumerate}
\item For all Thomason filtrations of finite length and $k\in\Z$, the class $\TFT_{\!k}$ is closed under direct limits (of $R\lMod$). Indeed, let $(X_i)_{i\in I}$ be a direct system in $\TFT_{\!k}$, and for all $i\in I$ consider $H_{{\cal H}_k}(X_i[-k])\cong H_{{\cal H}_{\vphantom k}}(X_i[-k])$; since $\Phi_k$ is weakly bounded below, by using the proof of Lemma~\ref{l:on_length_l+1}(ii) we get that $H^k(H_{\cal H}(X_i[-k]))\cong X_i$. On the other hand, $\varinjlim_{i\in I}^{({\cal H})} H_{\cal H}(X_i[-k])$ belongs to ${\cal H}_k$, which in turn is contained in $\De^{\ge k}(R)$, and consequently $H^k(\varinjlim_{i\in I}^{({\cal H})} H_{\cal H}(X_i[-k]))\in\TFT_{\!k}$ by Lemma~\ref{l:least_nonzero_cohomology}. But this latter module is isomorphic to $\varinjlim_{i\in I} H^k(H_{\cal H}(X_i[-k]))\cong \varinjlim_{i\in I} X_i$, as desired (see also the proof of Corollary~\ref{c:TF_n-AB5}).

\item For every $X\in\TFT_{\!-l-1}$ and $M\in{\cal H}_{-l+1}$ we have $\Ext^1_{\cal H}(H_{\cal H}(X[l+1]),M)=0$. Indeed, by Lemma~\ref{l:on_length_l+1}(ii) there are $U\in{\cal U}_{-l+2}$ and a triangle $U[1]\to X[l+1]\to H_{\cal H}(X[l+1])\buildrel+\over\to$, hence applying $\Hom_{\De(R)}(-,M[1])$ on the triangle we obtain, by \cite{Ver}, the desired vanishing of the ext-group since in the exact sequence
$$
	\Hom_{\De(R)}(U[2],M[1])\longrightarrow \Ext^1_{\cal H}(H_{\cal H}(X[l+1]),M)\longrightarrow
	\Hom_{\De(R)}(X[l+1],M[1])
$$
the first term is zero by axioms of t-structure, as well as the third since $M[1]\in\De^{\ge-l}(R)$.
\end{enumerate}
\end{rem}

\begin{lem}
Let\/ $\Phi$ be a Thomason filtration of length $l+1$, and $X\in{\cal T}_{-l-1}$. Then $H_{\cal H}(X[l+1])\in\fp({\cal H})$ if and only if the functor $\Hom_R(X,-)$ commutes with direct limits of direct systems in $\TFT_{\!-l-1}$.

In particular, for all\/ $B\in\fp(\TFT_{\!-l-1})$ we have $H_{\cal H}(B[l+1])\in\fp({\cal H})$.%
\label{l:fp(H)&TFT}%
\end{lem}
\begin{proof}\mbox{}\vglue0pt%
\noindent``$\Rightarrow$''\enspace Let $X\in{\cal T}_{-l-1}$ and suppose that $H_{\cal H}(X[l+1])$ is a finitely presented object of $\cal H$. Let $(X_i)_{i\in I}$ be a direct system in $\TFT_{\!-l-1}$. For each complex $X_i[l+1]\in{\cal U}$ we have a triangle $U_i[1]\to X_i[l+1]\to H_{\cal H}(X_i[l+1])\buildrel+\over\to$, say $M_i$ its last vertex, for some $U_i\in{\cal U}$,
in which $X_i\cong H^{-l-1}(M_i)$ for all $i\in I$, by Lemma~\ref{l:on_length_l+1}(ii). On the other hand, out of the triangle $U[1]\to X[l+1]\to M\buildrel+\over\to$ corresponding to $X[l+1]$, we obtain the commutative diagram with exact rows
\[
\xymatrix{%
	0\ar[r] & \smash[b]{\varinjlim\limits_{i\in I}}\Hom_{\De(R)}(M,M_i) \ar[r]\ar[d]_-{\cong} &
		\smash[b]{\varinjlim\limits_{i\in I}}\Hom_{\De(R)}(X[l+1],M_i) \ar[r]\ar[d] & 0 \cr
	0 \ar[r] & \Hom_{\De(R)}(M,\varinjlim\limits_{i\in I}\adjust M_i) \ar[r] &
		\Hom_{\De(R)}(X[l+1],\varinjlim\limits_{i\in I}\adjust M_i) \ar[r] & 0
}
\]
in which the left hand vertical homomorphism is bijective by hypothesis, thus the right hand one is so. Eventually, by Proposition~\ref{p:heart_is_bounded} we have a triangle
$$
	H^{-l-1}(\varinjlim_{i\in I}\adjust M_i)[l+1]\longrightarrow
		\varinjlim_{i\in I}\adjust M_i\longrightarrow
		\tau^{>-l-1}(\varinjlim_{i\in I}\adjust M_i)\buildrel+\over\longrightarrow
$$
whose first vertex is $(\varinjlim_{i\in I}H^{-l-1}(M_i))[l+1]$ since the standard cohomologies commute with direct limits, hence applying $\Hom_{\De(R)}(X[l+1],-)$ on such triangle we see, by \cite{Ver}, that the right hand isomorphism of the previous diagram actually is
$$
	\varinjlim_{i\in I}\Hom_R(X,H^{-l-1}(M_i))\longrightarrow\Hom_R(X,\varinjlim_{i\in I}H^{-l-1}(M_i))
$$
i.e.~the desired one showing that $\Hom_R(X,-)$ commutes with direct limits of direct systems in $\TFT_{\!-l-1}$.

\noindent``$\Leftarrow$''\enspace Let $X\in{\cal T}_{-l-1}$ be a module whose functor $\Hom_R(X,-)$ commutes with direct limits of direct systems in $\TFT_{\!-l-1}$. Let $(M_i)_{i\in I}$ be a direct system in $\cal H$, and consider the direct system of approximating triangles $(H^{-l-1}(M_i)[l+1]\to M_i\to\tau^{>-l-1}(M_i)\buildrel+\over\to)_{i\in I}$ in $\De(R)$. Applying $\Hom_{\De(R)}(X[l+1],-)$ we obtain, as in the previous part of the proof, the commutative diagram with exact rows
\[
\xymatrix{%
	0 \ar[r] & \smash[b]{\varinjlim\limits_{i\in I}}\Hom_R(X,H^{-l-1}(M_i)) \ar[r]\ar[d]_-{\cong} &
		\smash[b]{\varinjlim\limits_{i\in I}}\Hom_{\De(R)}(X[l+1],M_i) \ar[r]\ar[d] & 0 \cr
	0 \ar[r] & \Hom_R(X,\varinjlim\limits_{i\in I}H^{-l-1}(M_i)) \ar[r] &
		\Hom_{\De(R)}(X[l+1],\varinjlim\limits_{i\in I}\adjust M_i) \ar[r] & 0
}
\]
in which the left hand vertical homomorphism is bijective by hypothesis. Eventually, applying the functors $\Hom_{\De(R)}(-,M_i)$'s on the usual triangle $U[1]\to X[l+1]\to M\buildrel+\over\to$ (see the previous part of the proof), we obtain again that the right hand isomorphism of the previous diagram is the desired one. \qedhere
\end{proof}

\begin{cor}
Let\/ $\Phi$ be a Thomason filtration of length $l+1$. For every $B\in\fp(\TFT_{\!-l-1})$, there exist $n\in\N$, ideals $J_1,\ldots,J_n\in{\cal B}_{-l-1}$, and%
\label{c:fp(H)&TFT}%

\begin{enumerate}
\item[(i)] an epimorphism in $\cal H$
$$
	\bigoplus_{k=1}^nH_{\cal H}(K(J_k)[l+1])
		\relbar\joinrel\twoheadrightarrow H_{\cal H}(B[l+1]);
$$

\item[(ii)] integers $k_1,\ldots,k_n$, and a homomorphism in $R\lMod$
$$
	f\colon\bigoplus_{i=1}^n \Sigma^{-l-1}(y_{-l}(R/J_i)^{k_i})\longrightarrow B
$$
with $\coker f\in{\cal T}_{-l}$.
\end{enumerate}
\end{cor}
\begin{proof}\mbox{}\vglue0pt
\noindent(i)\enspace By Lemma~\ref{l:fp(H)&TFT} we know that $H_{\cal H}(B[l+1])$ is a finitely presented object of the heart. On the other hand, by Lemma~\ref{l:on_length_l+1}(iv) there are families $(J_i)_{i\in I}$ of ideals in ${\cal B}_{-l-1}$, a set $\varLambda$ and an epimorphism
$$
	p\colon \Bigl(\bigoplus_{i\in I}H_{\cal H}(K(J_i)[l+1])\Bigr)^{\!(\varLambda)}
		\relbar\joinrel\twoheadrightarrow H_{\cal H}(B[l+1])\;.
$$
For every finite subset $\bar I\subset I$, every $i\in\bar I$, and every finite subset $A\subset\varLambda$, consider the composition
$$
	H_{\cal H}(K(J_i)[l+1])^{(A)}
		\buildrel\varepsilon_i^A\over\longrightarrow
	\Bigl(\bigoplus_{i\in I}H_{\cal H}(K(J_i)[l+1])\Bigr)^{\!(\varLambda)}
		\buildrel p\over{\relbar\joinrel\twoheadrightarrow}
	H_{\cal H}(B[l+1])
$$
where $\varepsilon_i^A$ is the split monomorphism. Then
$$
	H_{\cal H}(B[l+1])=\im p=
	\sum_{\substack{i\in\bar I \cr \bar I\subset I\cr \hidewidth A\subset\varLambda\hidewidth}}
	\im(p\circ\varepsilon_i^A),
$$
hence being the former a finitely presented complex, there exist finite subsets $\bar I\subset I$ and $A\subset\varLambda$ such that $H_{\cal H}(B[l+1])=\sum_{i\in\bar I}\im(p\circ\varepsilon_i^A)$, as desired.

\noindent(ii)\enspace Let $p$ be as in part~(i) and define $f\leqdef H^{-l-1}(p)$. In view of the proof of Lemma~\ref{l:on_length_l+1}(iv), in the heart $\cal H$ we have exact rows
\begin{gather*}
	H_{\cal H}(\ker(f)[l+1])\longrightarrow \bigoplus_{i=1}^n H_{\cal H}(R/J_i[l+1])
		\buildrel\beta\over\longrightarrow H_{\cal H}(\im(f)[l+1])\longrightarrow0 \cr
	\noalign{\hbox{and}}
	H_{\cal H}(\im(f)[l+1])\buildrel\alpha\over\longrightarrow H_{\cal H}(B[l+1])\longrightarrow
		H_{\cal H}(\coker(f)[l+1])\longrightarrow0
\end{gather*}
in which $\alpha\circ\beta=p$, whence $\alpha$ is an epimorphism, so that $H_{\cal H}(\coker(f)[l+1])=0$. Consequently, the usual triangle of $\De(R)$ ending in this latter complex of $\cal H$ shows that $\coker(f)[l+1]$ is isomorphic to the object $U[1]$ for some $U\in{\cal U}$, meaning that $\coker f\cong H^{-l-1}(U[1])\in{\cal T}_{-l}$. \qedhere
\end{proof}

We now pass to consider some necessary conditions to the local coherence of the heart of a Thomason filtration of finite length. 

\begin{prop}
Let\/ $\Phi$ be a Thomason filtration of length $l+1$. If\/ $\cal H$ is a locally coherent Grothendieck category and\/ $P\in{}^{\bot_0}{\cal H}_{-l}$, then $P\in\fp({}^{\bot_0}{\cal H}_{-l})$ if and only if the following conditions hold true:%
\label{p:lc-fp(C)_char}%

\begin{enumerate}
\item[(i)] $H^{-l-1}(P)\in\fp(\TFT_{\!-l-1})$;

\item[(ii)] $\Hom_{\De(R)}(\tau^{\ge-l}(P)[-1],-)$ commutes with direct limits of direct systems in ${\cal H}_{-l+1}$.
\end{enumerate}
\end{prop}
\begin{proof}\mbox{}\vglue0pt%
\noindent``$\Rightarrow$''\enspace Let $P\in\fp({}^{\bot_0}{\cal H}_{-l})$. By Lemma~\ref{l:on_length_l+1}(iii) there exists $L\in{\cal H}_{-l+1}$ and a short exact sequence $0\to L\to H_{\cal H}(X[l+1])\to P\to0$ in $\cal H$, in which $X\leqdef H^{-l-1}(P)$. Set $W\leqdef H_{\cal H}(X[l+1])$, and consider the exact sequence of covariant functors
$$
	0\to\Hom_{\cal H}(P,-)\to \Hom_{\cal H}(W,-)\to\Hom_{\cal H}(L,-)\to \Ext^1_{\cal H}(P,-)\to\Ext^1_{\cal H}(W,-) \;.
$$
When we restrict these functors to ${\cal H}_{-l+2}$, we obtain $\Hom_{\cal H}(W,-)\mathclose\upharpoonright=0$ by Lem\-ma~\ref{l:on_length_l+1}(i), hence $\Hom_{\cal H}(P,-)\mathclose\upharpoonright=0$, moreover $\Ext^1_{\cal H}(W,-)\mathclose\upharpoonright=0$ by Remark~\ref{r:split_in_fp(TFT)}(2). Therefore, there is a natural isomorphism $\Hom_{\cal H}(L,-)\mathclose\upharpoonright\cong\Ext^1_{\cal H}(P,-)\mathclose\upharpoonright$, and by local coherence of $\cal H$ together with \cite[Proposition~3.5(2)]{Sao17} we get that $L\in\fp({\cal H}_{-l+2})\subseteq\fp({\cal H})$. By extension-closure of $\fp({\cal H})$ (see \cite[Corollary~1.8]{PSV19}), we have that $W$ is a finitely presented object of $\cal H$, whence $X\in\fp(\TFT_{\!-l-1})$ by Lemma~\ref{l:fp(H)&TFT}. This proves part~(i), so let us show part~(ii). By Proposition~\ref{p:heart_is_bounded} we have an exact triangle $H^{-l-1}(P)[l+1]\to P\to\tau^{\ge-l}(P)\buildrel+\over\to$, say $X[l+1]$ the first vertex, as in part~(i). By Remark~\ref{r:lc-TTF_finite_type}(1), we know that $\tau^{\ge-l}(P)\in{\cal U}[1]$. Thus, by applying the functor $H_{\cal H}$ on such triangle we obtain the exact row
$$
	0\longrightarrow H_{\cal H}(\tau^{\ge-l}(P)[-1])\longrightarrow H_{\cal H}(X[l+1])\longrightarrow P\longrightarrow0,
$$
which actually coincides with the short exact sequence $0\to L\to W\to P\to0$ provided by Lemma~\ref{l:on_length_l+1}(iii). Therefore, by local coherence of $\cal H$ together with Lemma~\ref{l:fp(H)&TFT}, $H_{\cal H}(\tau^{\ge-l}(P)[-1])$ is a finitely presented complex of the heart. Moreover, we have the triangle $U[1]\to\tau^{\ge-l}(P)[-1]\to H_{\cal H}(\tau^{\ge-l}(P)[-1])\buildrel+\over\to$ provided by $U\leqdef\tau^\le_{\cal U}(\tau^{\ge-l}_{\vphantom{\cal U}}(P)[-2])$, so given a direct system $(M_i)_{i\in I}$ of complexes in ${\cal H}_{-l+1}$ and applying the functors
$$
	F\leqdef \varinjlim_{i\in I}\Hom_{\De(R)}(-,M_i)\qquad\hbox{and}\quad
	G\leqdef \Hom_{\De(R)}(-,\varinjlim_{i\in I}\adjust M_i)
$$
on the previous triangle, we obtain the commutative diagram with exact rows
\[
\xymatrix{%
	0 \ar[r] & F(H_{\cal H}(\tau^{\ge-l}(P)[-1])) \ar[d]_-{\cong}\ar[r] &
		F(\tau^{\ge-l}(P)[-1]) \ar[d]\ar[r] & 0 \cr
	0 \ar[r] & G(H_{\cal H}(\tau^{\ge-l}(P)[-1])) \ar[r] & G(\tau^{\ge-l}(P)[-1]) \ar[r] & 0
}
\]
yielding the thesis. Notice that in this way we proved that our condition~(ii) is equivalent to $L\in\fp({\cal H})$.

\noindent``$\Leftarrow$''\enspace Let $P\in{}^{\bot_0}{\cal H}_{-l}$ and consider the short exact sequence $0\to L\to W\to P\to0$ of $\cal H$ provided by Lemma~\ref{l:on_length_l+1}(iii). Then $L\in\fp({\cal H})$ by what we said at the end of the proof of the previous part~(ii), whereas $W\in\fp({\cal H})$ by Lemma~\ref{l:fp(H)&TFT}. Therefore, $P$ is finitely presented as well, being a cokernel of a morphism in $\fp({\cal H})$. \qedhere
\end{proof}

\begin{cor}
Let\/ $\Phi$ be a weakly bounded below Thomason filtration such that its heart $\cal H$ is a locally coherent Grothendieck category. If\/ $B\in\fp({\cal H})$ and\/ $m$ is the least nonzero cohomology degree of\/ $B$, then we have $H^m(B)\in  \fp(\TFT_{\!m})$.%
\label{c:lc-fp(C)_char}%
\end{cor}
\begin{proof}
By definition of $m$ and by Proposition~\ref{p:heart_is_bounded}, we have $B\in{\cal H}_m$. Moreover, since $\cal H$ is locally coherent, so is ${\cal H}_m$ being a TTF class of finite type. In particular, the torsion pair $({}^{\bot_0}{\cal H}_{m+1},{\cal H}_{m+1})$ of ${\cal H}_m$ restricts to $\fp({\cal H}_m)$ (see Theorem~\ref{t:lc-TTF_finite_type}), hence the approximation $0\to {\bf x}(B)\to B\to {\bf y}(B)\to0$ of $B$ within the torsion pair (see Remark~\ref{r:lc-TTF_finite_type}(4)) actually is in $\fp({\cal H}_m)$. By the proof of Proposition~\ref{p:lc-fp(C)_char}, we get $H^m({\bf x}(B))\in\fp(\TFT_{\!m})$, and being ${\bf y}(B)\in{\cal H}_{m+1}\subseteq\De^{\ge m+1}(R)$, it follows $H^m(B)\cong H^m({\bf x}(B))$, and we are done. \qedhere
\end{proof}

\begin{prop}
Let\/ $\Phi$ be a Thomason filtration of length~$l+1$. If the heart $\cal H$ is locally coherent, then%
\label{p:lc-necessary_conditions_TFT}%

\begin{enumerate}
\item[(i)] $\fp(\TFT_{\!-l-1})$ is closed under kernels {\rm(}in $R\lMod${\rm)};

\item[(ii)] For all\/ $B\in\fp(\TFT_{\!-l-1})$, there exists a $R$-linear map
$$
	f\colon\bigoplus_{i=1}^n \Sigma^{-l-1}(y_{-l}(R/J_i)^{k_i})\longrightarrow B
$$
with $\coker f\in\Sigma{\cal T}_{-l}$;

\item[(iii)] For all morphisms $f$ in $\fp(\TFT_{\!-l-1})$ with $\coker f\in{\cal T}_{-l}$, then $\coker f\in\Sigma{\cal T}_{-l}$.
\end{enumerate}
\end{prop}
\begin{proof}\mbox{}\vglue0pt%
\noindent(i)\enspace Given $f\colon B\to B'$ a homomorphism in $\fp(\TFT_{\!-l-1})$, we have to show that $\ker f\in\fp(\TFT_{\!-l-1})$. Consider the following diagram in $\De(R)$ obtained by approximating the stalk complexes of the modules within the t-structure $({\cal U},{\cal V})$:
\[
\xymatrix{%
	U[1] \ar[r]^-{a}\ar@{.>}[d] & B[l+1] \ar[r]\ar[d]^{f[l+1]} & H_{\cal H}(B[l+1]) \ar[r]^-{+}\ar@{.>}[d]^q & \cr
	U'[1] \ar[r] & B'[l+1] \ar[r]^-{b'} & H_{\cal H}(B'[l+1]) \ar[r]^-{+} & 
}
\]
Since $b'\circ f[l+1]\circ a=0$, the dotted vertical maps actually exist and they complete the diagram to a morphism of triangles (see e.g.\ \cite[Proposition~1.4.5]{Mil}). Now, by Lemma~\ref{l:on_length_l+1}(i) we have that $H_{\cal H}(B[l+1])\reqdef M$ and $H_{\cal H}(B'[l+1])\reqdef M'$ are complexes of ${}^{\bot_0}{\cal H}_{-l}$, whereas by Lemma~\ref{l:fp(H)&TFT} we have that $q$ is a morphism in $\fp({}^{\bot_0}{\cal H}_{-l})$. This said, by the hypothesis of quasi local coherence of ${}^{\bot_0}{\cal H}_{-l}$ we infer that ${\bf x}(\ker^{({\cal H})}(q))$ is a finitely presented object of ${}^{\bot_0}{\cal H}_{-l}$, so that of $\cal H$. Moreover, notice that $H^{-l-1}(q)=f$, and that the standard cohomology sequences associated with the following two sequences of $\cal H$
\[
\xymatrix{%
	0 \ar[r] & {\bf x}(\ker^{({\cal H})}(q)) \ar[r] & \ker^{({\cal H})}(q) \ar[r]\ar@{>->}[d] &
		{\bf y}(\ker^{({\cal H})}(q)) \ar[r] & 0 \cr
	&& M \ar[d]^-{q} \cr
	&& M'
}
\]
yield
\begin{align*}
	H^{-l-1}({\bf x}(\ker^{({\cal H})}(q)) &= H^{-l-1}(\ker^{({\cal H})}(q)) \cr
	&=\ker H^{-l-1}(q)=\ker f,
\end{align*}
where the second equality follows by applying the functor $H^{-l-1}$ to commutative diagram of $\cal H$ obtained by the factorisation of $q$ through its kernel and image. Now, since $\ker f=H^{-l-1}({\bf x}(\ker^{({\cal H})}(q)))$, by Proposition~\ref{p:lc-fp(C)_char} we infer that $\ker f\in\fp(\TFT_{\!-l-1})$, as desired.

\noindent(ii)\enspace We already know the existence of a homomorphism $f\colon\bigoplus_{i=1}^n\Sigma^{-l-1}(y_{-l}(R/J_i)^{k_i})\to B$ having cokernel in ${\cal T}_{-l}$ (see Corollary~\ref{c:fp(H)&TFT}(ii)). Let us rename the corresponding canonical exact sequence by $0\to K\to N\buildrel f\over\to B\to C\to0$, and let $L\leqdef\im f$. Since $N,B\in\fp(\TFT_{\!-l-1})$, by part~(i) we know that $K\in\fp(\TFT_{\!-l-1})$ as well. In turn, $L\in\TF_{\!-l-1}$ and $H_{\cal H}(L[l+1])$ is finitely presented being a cokernel in $\fp({\cal H})$, by Lemma~\ref{l:fp(H)&TFT}. On the other hand, since $C\in{\cal T}_{-l}$, we have $C[l+1]\in{\cal U}[1]$ whence $H_{\cal H}(C[l+1])=0$. This said, in the heart we have the commutative diagram with exact row
\[
\xymatrix@C=.5em{%
	H_{\cal H}(B[l]) \ar[rr]\ar@{->>}[dr] && H_{\cal H}(C[l]) \ar[rr]\ar@{->>}[dr] &&
		H_{\cal H}(L[l+1]) \ar[rr] && H_{\cal H}(B[l+1]) \ar[rr] && 0 \cr
	& H^{\hphantom\prime} \ar@{>->}[ur] && H' \ar@{>->}[ur] && \hphantom{H'} && \hphantom{H}
}
\]
in which $H_{\cal H}(B[l])\in{\cal H}_{-l+2}$ (as we will show at the end of the proof), so that also $H$ belongs to such TTF class of $\cal H$; then $H_{\cal H}(C[l])\in{}^{\bot_0}{\cal H}_{-l+1}$ by an adaptation of Lemma~\ref{l:on_length_l+1}(i), so that also $H'$ belongs to such torsion class of $\cal H$; eventually, the remaining terms of the diagram belong to ${}^{\bot_0}{\cal H}_{-l}\cap\fp({\cal H})$ by Lemmata~\ref{l:on_length_l+1}(i) and~\ref{l:fp(H)&TFT}. Since $\cal H$ is locally coherent by assumption, we infer that $H'\in\fp({\cal H})$. Now we take the standard cohomologies of the extension of $H_{\cal H}(C[l])$ to see that $H^{-l}(H_{\cal H}(C[l]))\cong H^{-l}(H')\in\fp(\TFT_{\!-l})$ by Corollary~\ref{c:lc-fp(C)_char}. On the other hand, the standard cohomology sequence of the triangle $U[1]\to C[l]\to H_{\cal H}(C[l])\buildrel+\over\to$ provided by some object $U\in{\cal U}$ yields 
$$
	H^{-l+1}(U)\longrightarrow C \longrightarrow H^{-l}(H_{\cal H}(C[l])) \longrightarrow
		H^{-l+2}(U) \longrightarrow 0
$$
and the canonical factorisations of the first two homomorphisms give indeed the approximation of $C$ within the torsion pair $({\cal T}_{-l+1},{\cal F}_{-l+1})$. Therefore, $H^{-l}(H_{\cal H}(C[l]))=\Sigma^{-l}(y_{-l+1}(C))$ by means of the monomorphism induced by the natural transformation $\sigma$ (see Remark~\ref{r:functorial_monomorphism_sigma}), and we are done.

As announced above, let us now show that $H_{\cal H}(B[l])\in{\cal H}_{-l+2}$. By Lemma~\ref{l:on_length_l+1}(ii) there exist $U\in{\cal U}_{-l+2}$ and a triangle $U[1]\to B[l+1]\to H_{\cal H}(B[l+1])\buildrel+\over\to$, whence $H_{\cal H}(H_{\cal H}(B[l+1])[j])=0$ for $j=-1,-2$, meaning that $H_{\cal H}(B[l])=H_{\cal H}(U)=H_{{\cal H}_{-l+2}}(U)$, as claimed.

\noindent(iii)\enspace This is a consequence of the proof of part~(ii). \qedhere
\end{proof}

\subsection{A characterisation of the local coherence}%
We are now ready to state and prove the characterisation of the local coherence of the heart associated with a Thomason filtration of finite length of the prime spectrum of a commutative ring.

\begin{thm}
Let\/ $\Phi$ be a Thomason filtration of length~$l+1$. Then $\cal H$ is a locally coherent Grothendieck category if and only if the following conditions hold true:%
\label{t:recursive}%

\begin{enumerate}
\item ${\cal H}_{-l}$ is a locally coherent Grothendieck category;

\item For every $B\in\fp(\TFT_{\!-l-1})$, the functor $\Ext^1_{\cal H}(H_{\cal H}(B[l+1]),-)$ commutes with direct limits of direct systems in ${\cal H}_{-l}$;

\item For every $B\in\fp(\TFT_{\!-l-1})$, the functor $\Ext^1_{\cal H}(H_{\cal H}(B[l+1]),-)$ commutes with direct limits of direct systems in ${}^{\bot_0}{\cal H}_{-l}$;

\item For all\/ $P\in\fp({}^{\bot_0}{\cal H}_{-l})$, in the functorial short exact sequence $0\to L\to W\to P\to0$ of Lemma~\ref{l:on_length_l+1}, we have $L\in\fp({\cal H})$;

\item The torsion pair $({}^{\bot_0}{\cal H}_{-l},{\cal H}_{-l})$ restricts to $\fp({\cal H})$.
\end{enumerate}
\end{thm}
\begin{proof}
Let us assume that the heart $\cal H$ associated with $\Phi$ is a locally coherent Grothendieck category, and let us show that the five stated conditions hold true. (1) ${\cal H}_{-l}$ is a locally coherent Grothendieck category since it is a TTF class of finite type in $\cal H$. (2) and~(3) follow by Lemma~\ref{l:fp(H)&TFT} and \cite[Proposition~3.5(2)]{Sao17}. (4)~follows by the proof of Proposition~\ref{p:lc-fp(C)_char}. (5)~holds true by hypothesis on $\cal H$ and since the torsion pair $({\cal H}_{-l}^{\vphantom{\bot_0}},{\cal H}_{-l}^{\bot_0})$ is of finite type.

Conversely, let us show that the five stated conditions imply the local coherence of the heart $\cal H$. More in details, we want to exploit Theorem~\ref{t:lc-TTF_finite_type} which characterises the local coherence of an arbitrary Grothendieck category equipped with a TTF~triple of finite type. Notice that hypothesis~(iii)' of Theorem~\ref{t:lc-TTF_finite_type} coincides with our hypothesis~(5)

Concerning condition~(i) of Theorem~\ref{t:lc-TTF_finite_type}, thanks to our hypothesis~(1) we need to check that the torsion class ${}^{\bot_0}{\cal H}_{-l}$ is quasi locally coherent. We know that $\cal H$ is a locally finitely presented Grothendieck category by Proposition~\ref{p:heart_is_bounded} and Theorem~\ref{t:when_hearts_coincide}, and by imitating the proof of ``${\rm(a)\Rightarrow(b)}$'' in Theorem~\ref{t:lc-TTF_finite_type} we deduce that ${}^{\bot_0}{\cal H}_{-l}$ is locally finitely presented as well, thus it remains to prove that $\fp({}^{\bot_0}{\cal H}_{-l})$ is closed under taking kernels; in particular, it suffices to check that for every epimorphism $p\colon P\to P'$ in $\fp({}^{\bot_0}{\cal H}_{-l})$, we have ${\bf x}(\ker^{({\cal H})}(f))\in\fp({}^{\bot_0}{\cal H})\subseteq\fp({\cal H})$. The following diagram provided by Lemma~\ref{l:on_length_l+1}(iii)
\[
\xymatrix{%
	0 \ar[r] & L \ar@{.>}[d]_\alpha\ar[r] & W \ar@{.>}[d]^\beta\ar[r] & P \ar[r]\ar@{->>}[d]^p & 0 \cr
	0 \ar[r] & L' \ar[r] & W' \ar[r] & P' \ar[r] & 0
}
\]
can be completed to a commutative diagram, since in $\De(R)$ the composition $W\to P\buildrel p\over\to P'\to L[1]$ yields an element of $\Ext^1_{\cal H}(W,L)$, which is zero by Remark~\ref{r:split_in_fp(TFT)}(2); consequently $\alpha$ is defined by the universal property of the kernel. By condition~(4), the objects $L,L'$ are finitely presented complexes of $\cal H$, while $W,W'$ are so by extension-closure. Moreover, since $W'\in{}^{\bot_0}{\cal H}_{-l}$ by Lemma~\ref{l:on_length_l+1}(i), then $\beta$ is an epimorphism since its cokernel in $\cal H$ is a quotient in both the torsion classes ${}^{\bot_0}{\cal H}_{-l}$ and ${\cal H}_{-l}$. The Snake~Lemma applied on the previous commutative diagram gives us the exact row
$$
	0\longrightarrow \ker^{({\cal H})}(\alpha)\longrightarrow \ker^{({\cal H})}(\beta)\longrightarrow
		\ker^{({\cal H})}(p)\longrightarrow \coker^{({\cal H})}(\alpha)\longrightarrow0
$$
in which the outer terms are finitely presented by hypotheses~(1). This said, by \cite[Corollary~1.8]{PSV19} and hypothesis~(5), our claim will follow once we check that $H\leqdef\ker^{({\cal H})}(\beta)$ is a finitely presented object. Let $X=H^{-l-1}(P)$ so that $W=H_{\cal H}(X[l+1])$ (similarly for $W'$), and consider $f\leqdef H^{-l-1}(\beta)$, with $K\leqdef\ker f$, $N\leqdef\im f$ and $C\leqdef\coker f$. By applying the functor $H_{\cal H}$ to the triangles $K[l+1]\to X[l+1]\to N[l+1]\buildrel+\over\to$ and $C[l]\to N[l+1]\to X'[l+1]\buildrel+\over\to$ obtained out of the canonical short exact sequences in $R\lMod$ associated to $f$, we get the commutative diagram of $\cal H$ with exact rows
\[
\xymatrix@C=.35em@R=1em{%
	&& \hphantom{H_{\cal H}(K)} && H_{\cal H}(K[l+1]) \ar[dd]\ar@{->>}[ld] \cr
	& M' \ar@{=}[rr]\ar@{>->}[dr] && M' \ar@{>->}[dr] && \hphantom{H_{\cal H}(K)}  && \hphantom{H_{\cal H}(K)} \cr
	0 \ar[rr] && H^{\hphantom\prime} \ar[rr]\ar[dd]\ar@{->>}[rd] && W \ar[rr]^-{\beta}\ar@{->>}[dd]^-{\delta} &&
		W' \ar[rr]\ar@{=}[dd] && 0 \cr
	& \hphantom{H_{\cal H}(K)} && M \ar@{>->}[dr] \cr
	&& *+[l]{H_{\cal H}(C[l])} \ar[rr]\ar@{->>}[ur] && H_{\cal H}(N[l+1]) \ar[rr]^-{\alpha} && W' \ar[rr] && 0\cr	
}
\]
in which $\beta=\alpha\circ\delta$ (this also implies $H_{\cal H}(C[l+1])=0$ i.e.~that $C\in{\cal T}_{-l}$, whence $H_{\cal H}(C[l])\in{\cal H}_{-l}$) and the epimorphism $H\to M$ is provided by the universal property of the kernel. Moreover, by the Snake~Lemma, the image $M'$ of the morphism $H_{\cal H}(K[l+1])\to W$ induces the short exact sequence $0\to M'\to H\to M\to0$, which actually is the approximation of $H$ within the torsion pair $({}^{\bot_0}{\cal H}_{-l},{\cal H}_{-l})$ (see Lemma~\ref{l:on_length_l+1}(i)). Thus, we reduced our claim to check that $M',M\in\fp({\cal H})$. We have $M'\in\fp({}^{\bot_0}{\cal H}_{-l})\subseteq\fp({\cal H})$ by hypothesis~(3) applied on the short exact sequence $0\to M'\to W\to H_{\cal H}(N[l+1])\to0$, whereas $M\in\fp({\cal H}_{-l})\subseteq\fp({\cal H})$ thanks to hypothesis~(2) applied on the short exact sequence $0\to M\to H_{\cal H}(N[l+1])\to W'\to0$.

Eventually, let us prove that Theorem~\ref{t:lc-TTF_finite_type}(ii) holds true. Given $P\in\fp({}^{\bot_0}{\cal H}_{-l})$ we have a short exact sequence $0\to L\to W\to P\to0$ by Lemma~\ref{l:on_length_l+1}(iii), in which $L\in\fp({\cal H})$ by hypothesis~(4). The sequence yields an exact row
\begin{multline*}
	0\to \Hom_{\cal H}(P,-)\to \Hom_{\cal H}(W,-)\to \Hom_{\cal H}(L,-)\relbar\joinrel\cdots \cr
		\cdots\joinrel\to\Ext^1_{\cal H}(P,-)\to\Ext^1_{\cal H}(W,-)\to\Ext^1_{\cal H}(L,-)
\end{multline*}
of covariant functors ${\cal H}\to\Ab$. When restricted to ${\cal H}_{-l}$, the first three functors commute with direct limits, whereas the last two do so respectively by hypotheses~(2) and~(1), and by \cite[Proposition~3.5(2)]{Sao17}, so that $\Ext^1_{\cal H}(P,-)\mathclose\upharpoonright_{{\cal H}_{-l}}$ commutes with the desired direct limits.
\end{proof}

\begin{rem}
The previous Theorem provides a recursive argument for the construction of a Thomason filtration of finite length whose heart is a locally coherent Grothendieck category. However, one practical issue is to check conditions~(2) and~(3) when the length of the filtration, i.e.~$l$, is greater than $2$. Nonetheless, for $0\le l\le2$ (that are values involving interesting classes of abelian categories, e.g.~torsion classes of $R\lMod$ and certain HRS~hearts, as we have already seen), the conditions of the Theorem simplify so that most of them can be rephrased in module-theoretic ones, as we will show in the following results.
\end{rem}

The length zero case has been treated in subsection~\ref{ss:a_crucial_example}, and it consists in a characterisation of the local coherence of an arbitrary hereditary torsion class of finite type in $R\lMod$.

\begin{cor}
Let $\Phi$ be a Thomason filtration of length~$1$. Then $\cal H$ is a locally coherent Grothendieck category if and only if the following conditions are satisfied:%
\label{c:recursive-length1}%

\begin{enumerate}
\item ${\cal T}_0$ is locally coherent;

\item For all\/ $P\in\fp(\TFT_{\!-1})$, the functor $\Hom_R(P,-)$ commutes direct limits of direct systems in ${\cal T}_0$;

\item For all\/ $P\in\fp(\TFT_{\!-1})$, the functor $\Ext^1_R(P,-)$ commutes direct limits of direct systems in $\TFT_{\!-1}$;

\item For all\/ $Q\in\fp({\cal T}_0)$, the functor $\Ext^2_R(Q,-)$ commutes with direct limits of direct systems in $\TFT_{\!-1}$.
\end{enumerate}
\end{cor}
\begin{proof}
First, notice that $H_{\cal H}(P[1])=P[1]$ for all $P\in\TFT_{\!-1}$, that ${\cal H}_0={\cal T}_0[0]$ and that ${}^{\bot_0}{\cal H}_0=\TFT_{\!-1}[1]$. Thus, the stated conditions~(1), (2) and~(3) are exactly the corresponding ones of Theorem~\ref{t:recursive}, since $l=0$. In turn, condition~(4) of the Theorem is clearly satisfied since $L\in{\cal H}_1=0$ (see Lemma~\ref{l:on_length_l+1}). Let us check condition~(5) of the Theorem. We claim that it is implied by our condition~(4). Let $B\in\fp({\cal H})$ and consider its approximation $0\to H^{-1}(B)[1]\to B\to H^0(B)[0]\to0$ within the torsion pair $({}^{\bot_0}{\cal H}_0,{\cal H}_0)=(\TFT_{\!-1}[1],{\cal T}_0[0])$; we have to prove that the outer terms are finitely presented objects of $\cal H$. We recall that $H^0(B)[0]\in\fp({\cal H})$ by Corollary~\ref{c:R-mod&T_0-are_in_fp(H)}(i); in particular, we have $H^0(B)\in\fp({\cal T}_0)$. Let $(X_i)_{i\in I}$ be a direct system of modules in $\TFT_{\!-1}$. Applying the functors
$$
	F^k\leqdef \varinjlim_{i\in I}\Ext^k_{\cal H}(-,X_i[1])\quad\hbox{and}\quad
	G^k\leqdef \Ext^k_{\cal H}(-,\varinjlim_{i\in I}X_i[1])\qquad (k\in\N\cup\{0\})
$$
on the previous approximation, say it $0\to Y[1]\to B\to X[0]\to0$ for short, we obtain the following commutative diagram with exact rows
\[
\xymatrix{%
	0 \ar[r] & F^0(X[0]) \ar[d]_-{f_1}\ar[r] & F^0(B) \ar[d]^-{f_2}\ar[r] &
		F^0(Y[1]) \ar[d]^-{f_3}\ar[r] & F^1(X[0]) \ar[d]^-{f_4}\ar[r] & F^1(B) \ar[d]^-{f_5} \cr
	0 \ar[r] & G^0(X[0]) \ar[r] & G^0(B) \ar[r] & G^0(Y[1]) \ar[r] & G^1(X[0]) \ar[r] & G^1(B)
}
\]
in which, using \cite{Ver,BBD82}, $f_1$ is an isomorphism by Corollary~\ref{c:R-mod&T_0-are_in_fp(H)}(i), $f_2$ is iso and $f_5$ is monic, and $f_4$ is an isomorphism by hypothesis~(4), so we are done by the Five~Lemma.

In order to conclude, it remains to prove that if $\cal H$ is locally coherent, then our hypothesis~(4) is satisfied. Let $Q\in\fp({\cal T}_0)$. By Corollary~\ref{c:R-mod&T_0-are_in_fp(H)}(i) again, we have $Q[0]\in\fp({\cal H})$, hence $\Ext^1_{\cal H}(Q[0],-)$ preserves direct limits by \cite[Proposition~3.5(2)]{Sao17}; in particular, it commutes with direct limits of $\TFT_{\!-1}[1]$, which is our thesis by \cite{Ver,BBD82}. \qedhere
\end{proof}

\begin{cor}
Let $\Phi$ be a Thomason filtration of length~$2$. Then $\cal H$ is a locally coherent Grothendieck category if and only if the following conditions are satisfied:%
\label{c:recursive-length2}%

\begin{enumerate}
\item ${\cal H}_{-1}$ is locally coherent (cf.~Corollary~\ref{c:recursive-length1});

\item For all\/ $P\in\fp(\TFT_{\!-2})$, the functor $\Hom_R(P,-)$ preserves direct limits of direct systems in $\TFT_{\!-1}$;

\item The following conditions hold true:

\begin{enumerate}
\item[(3.i)] For all\/ $J\in{\cal B}_{-2}$, the functor $\Ext^1_{\cal H}(\Sigma^{-2}(y_{-1}(R/J))[2],-)$ preserves direct limits of direct systems in ${}^{\bot_0}{\cal H}_{-1}$;

\item[(3.ii)] $\fp(\TFT_{\!-2})$ is closed under kernels in $R\lMod$.

\item[(3.iii)] For all morphisms $f$ in $\fp(\TFT_{\!-2})$, we have $\Sigma^{-2}(\im f)/{\im f}\in R\lmod$;
\end{enumerate}

\item For all exact sequences of\/ $R\lMod$ of the form $0\to Y\to M\buildrel f\over\to N\to X\to0$ such that $Y\in\fp(\TFT_{\!-2})$, $X\in\fg({\cal T}_0)$ and\/ $\mathop{\rm Cone}(f[1])\in{\cal H}$, we have $X\in\fp({\cal T}_0)$.

\item For all\/ $P\in\fp({\cal H})$, the following conditions hold true:

\begin{enumerate}
\item[(5.i)] $H^{-2}(P)\in\fp(\TFT_{\!-2})$;

\item[(5.ii)] $x_0(H^{-1}(P))\in\fp({\cal T}_0)$.
\end{enumerate}%
\end{enumerate}
\end{cor}
\begin{proof}
It is clear that our hypothesis~(1) corresponds exactly to condition~(1) of Theorem~\ref{t:recursive}.

Let us prove that our hypothesis~(2) is equivalent to Theorem~\ref{t:recursive}(2). Notice again that for all $P\in\TFT_{\!-2}$ we have $H_{\cal H}(P[2])=P[2]$. This said, any direct system $(M_i)_{i\in I}$ of ${\cal H}_{-1}$ is approximated by $(0\to H^{-1}(M_i)[1]\to M_i\to H^0(M_i)[0]\to0)_{i\in I}$ within the left constituent of the TTF~triple given by the TTF~class ${\cal H}_0$ (see the proof of Corollary~\ref{c:recursive-length1}). Thus, by applying the cohomological functor $\Hom_{\De(R)}(P[2],-)$ on the direct limit of the previous approximation and using \cite{Ver,BBD82}, we obtain the commutative diagram with exact rows
\[
\xymatrix{%
	0 \ar[r] & \varinjlim\limits_{i\in I}\Hom_R(P,H^{-1}(M_i)) \ar[d]\ar[r] &
		\varinjlim\limits_{i\in I}\Ext^1_{\cal H}(P[2],M_i) \ar[d]\ar[r] & 0 \cr
	0 \ar[r] & \Hom_R(P,\varinjlim\limits_{i\in I}H^{-1}(M_i)) \ar[r] &
		\Ext^1_{\cal H}(P[2],\varinjlim\limits_{i\in I}M_i) \ar[r] & 0
}
\]
which shows the desired equivalence, since for all $M\in{\cal H}_{-1}$ and $Y\in\TFT_{\!-1}$, we have $H^{-1}(M)\in\TFT_{\!-1}$ and $Y[1]\in{\cal H}_{-1}$.

Let us show that Theorem~\ref{t:recursive} implies our condition~(3).

\noindent(3.i)\enspace Let $J\in{\cal B}_{-2}$. By the approximating triangle $\tau^{\le-1}(K(J)[2])\to K(J)[2]\to R/J[2]\buildrel+\over\to$ of the Koszul complex $K(J)[2]$ within the standard t-structure of $\De(R)$, since the first vertex belongs to ${\cal U}[3]$ by the proof of Lemma~\ref{l:on_length_l+1}(iv), we obtain $H_{\cal H}(K(J)[2])\cong H_{\cal H}(R/J[2])$, and these are finitely presented objects of $\cal H$ by \cite[Lemma~6.3]{SSV17}. Let us call $M$ such complex; it fits in an exact triangle $U[1]\to K(J)[2]\to M\buildrel+\over\to$ provided by some $U\in{\cal U}$, whose standard cohomology exact sequence yields
$$
	0\longrightarrow H^{-1}(U)\longrightarrow R/J\buildrel d\over\longrightarrow H^{-2}(M)\longrightarrow
		H^0(U)\longrightarrow0 \;.
$$
On the one hand we infer that $M$ is a stalk, i.e.\ $M\cong H^{-2}(M)[2]$, whence in turn $H^{-2}(M)\in\fp(\TFT_{\!-2})$ by Lemma~\ref{l:fp(H)&TFT}; on the other hand, we have $\im d\in\TF_{\!-2}$ and $H^0(U)\in{\cal T}_0$, thus
$$
	H^{-2}(M)\cong\Sigma^{-2}(\im d)\cong\Sigma^{-2}(y_{-1}(R/J))
$$
and we conclude by Lemma~\ref{l:SigmaT} and \cite[Proposition~3.5(2)]{Sao17}.

\noindent(3.ii)\enspace It follows by Proposition~\ref{p:lc-necessary_conditions_TFT}(i).

\noindent(3.iii)\enspace Let $f\colon B\to B'$ be a morphism in $\fp(\TFT_{\!-2})$. In view of Remark~\ref{r:functorial_monomorphism_sigma}, we have to prove that $\coker\sigma_{\im f}$ is a finitely presented $R$-module. We have $\ker f\in\fp(\TFT_{\!-2})$ by part~(3.ii), so by the exact sequence
$$
	0\longrightarrow H_{\cal H}(\im(f)[1])\longrightarrow\ker(f)[2]\longrightarrow B[2]\longrightarrow H_{\cal H}(\im(f)[2])\longrightarrow0
$$
of the heart $\cal H$ we obtain that the outer terms are finitely presented, in particular we infer $\Sigma^{-2}(\im f)\in\fp(\TFT_{\!-2})$ by Lemma~\ref{l:SigmaT}. On the other hand, from the short exact sequence $0\to\im f\to\Sigma^{-2}(\im f)\to\coker\sigma_{\im f}\to0$ we obtain the triangle
$$
	\Sigma^{-2}(\im f)[0]\longrightarrow\coker(\sigma_{\im f})[0]\longrightarrow
		\im(f)[1]\longrightarrow \Sigma^{-2}(\im f)[1]
$$
whence
$$
	H_{\cal H}(\im(f)[1])\cong H_{\cal H}(\coker(\sigma_{\im f})[0])=\coker(\sigma_{\im f})[0]
$$
and the latter term belongs to $\fp({\cal H})$. Then, by Corollary~\ref{c:R-mod&T_0-are_in_fp(H)}(i) we obtain that $\coker\sigma_{\im f}\in R\lmod$, as desired.

Conversely, let us prove that our hypotheses~(2) and~(3) imply Theorem~\ref{t:recursive}(3). Let $B\in\fp(\TFT_{\!-2})$. By Corollary~\ref{c:fp(H)&TFT}(ii) there exists an $R$-linear map
$$
	f\colon\bigoplus_{i=1}^n\Sigma^{-2}(y_{-1}(R/J_i)^{k_i})\longrightarrow B,
$$
which we rename $f\colon N\to B$, with cokernel $C\in{\cal T}_{-1}$ and kernel $K\in\fp(\TFT_{\!-2})$ by hypothesis~(3.ii). Let $f=\mu\circ\beta$ be the canonical factorisation of $f$ through its image $L$. Consider the following commutative diagram with exact rows in $\cal H$
\[
\xymatrix{%
	&&&& H_{\cal H}(C[1]) \ar@{>->}[d]^-{\gamma} \cr
	0 \ar[r] & H_{\cal H}(L[1]) \ar[r]^-{\lambda} & K[2] \ar[d]\ar[r]^-{v[2]} & N[2] \ar@{=}[d]\ar[r]^-{H_{\cal H}(\beta[2])} &
		H_{\cal H}(L[2]) \ar@{->>}[d]^-{H_{\cal H}(\mu[2])}\ar[r] & 0 \cr
	& 0 \ar[r] & H \ar[r] & N[2] \ar[r]^-{f[2]} & B[2] \ar[r] & 0
}
\]
in which $f[2]$ is an epimorphism since its cone in $\De(R)$ belongs to ${\cal U}[1]$, whereas $\lambda$ and $\gamma$ are monomorphisms since $H_{\cal H}(N[1])=0$ and $H_{\cal H}(B[1])=0$, respectively. Moreover, notice that $H_{\cal H}(C[1])\cong y_{-1}(C)[1]$, in particular it belongs to ${\cal H}_{-1}$. The Snake~Lemma yields a short exact sequence $0\to \im^{({\cal H})}(v[2])\to H\to y_{-1}(C)[1]\to0$ in which the outer terms are finitely presented objects, as we now show. On the one hand, $\im^{({\cal H})}(v[2])$ is finitely presented for being a cokernel in $\fp({\cal H})$; indeed, $H_{\cal H}(L[1])\cong\coker(\sigma_L)[0]$ is finitely presented by hypothesis~(3.iii) and Corollary~\ref{c:R-mod&T_0-are_in_fp(H)}(i). On the other hand, we have
$$
	H_{\cal H}(L[2])=H_{\cal H}(\Sigma^{-2}(L)[2])=\Sigma^{-2}(L)[2]
$$
and $\Sigma^{-2}(L)\in\fp(\TFT_{\!-2})$ by Lemma~\ref{l:SigmaT}; moreover, by our condition~(2) (i.e.~Theorem~\ref{t:recursive}(2)) in view of the exact column of the previous diagram, we infer that $y_{-1}(C)[1]\in\fp({\cal H}_{-1})\subseteq\fp({\cal H})$. By extension-closure, we have $H\in\fp({\cal H})$ as well. Thus, the second exact row of the previous diagram induces the exact sequence of covariant functors
\begin{multline*}
	0\to\Hom_{\cal H}(B[2],-)\to\Hom_{\cal H}(N[2],-)
		\to\Hom_{\cal H}(H,-)\relbar\joinrel\cdots \cr
	\cdots\joinrel\to \Ext^1_{\cal H}(B[2],-)\to
			\Ext^1_{\cal H}(N[2],-)\to\Ext^1_{\cal H}(H,-)
\end{multline*}
in which, since $\Ext^1_{\cal H}(N[2],-)$ restricted to ${}^{\bot_0}{\cal H}_{-1}$ preserves direct limits by~(3.i), then also $\Ext^1_{\cal H}(B[2],-)\mathclose\upharpoonright$ does so, as desired.

Let us prove that Theorem~\ref{t:recursive} implies our condition~(4). First notice that if $X\in\fg({\cal T}_0)$, then there exists $B\in\fp({\cal T}_0)$ and an epimorphism $p\colon B\to X$, whence a short exact sequence $0\to\ker(p)[0]\to B[0]\to X[0]\to0$ in $\cal H$, which shows that $X[0]\in\fg({\cal H})$. Let now $0\to Y\to M\buildrel f\over\to N\to X\to0$ be as in the statement. Then we obtain the following diagram of $\De(R)$
\[
\xymatrix{%
	&& Y[2] \ar[d] \cr
	M[1] \ar[r]^-{f[1]} & N[1] \ar[r] & \mathop{\rm Cone}(f[1]) \ar[r]^-{+}\ar[d] & {} \cr
	&& X[1] \ar[d]^-{+} \cr
	&& {}
}
\]
and the rotation of the vertical triangle is a short exact sequence of $\cal H$ by hypothesis on the cone. In particular, by $0\to X[0]\to Y[2]\to\mathop{\rm Cone}(f[1])\to0$, being $X[0]\in\fg({\cal H})$ and $Y[2]\in\fp({\cal H})$ (see Lemma~\ref{l:SigmaT}(d)), we infer that $\mathop{\rm Cone}(f[1])$ is a finitely presented object of $\cal H$. By \cite[Proposition~3.5(2)]{Sao17}, the functor $\Ext^1_{\cal H}(\mathop{\rm Cone}(f[1]),-)$ commutes with direct limits, in particular those of ${\cal T}_0[0]$, but the relevant restriction of the functor is naturally isomorphic to $\Hom_R(X,-)\mathclose\upharpoonright_{{\cal T}_0}$, and we are done.

Let us prove that our conditions~(4) and~(5.i) imply Theorem~\ref{t:recursive}(4). Let $P\in\fp({}^{\bot_0}{\cal H}_{-1})$, and consider the associated short exact sequence $0\to L\buildrel\varepsilon\over\to W\to P\to0$ as in Lemma~\ref{l:on_length_l+1}(iii), so that with $L\in{\cal H}_0$ and $W=H^{-2}(P)[2]$. By hypothesis~(5.i) and Lemma~\ref{l:SigmaT} we know that $W\in\fp({\cal H})$, thus $L\in\fg({\cal H})$. Therefore, there exists an epimorphism $Q\to L$ originating in a finitely presented complex $Q$ of $\cal H$, whence we have the epimorphism $H^0(Q)\to H^0(L)$ originating in $H^0(Q)\in\fp({\cal T}_0)$, whence $H^0(L)\in\fg({\cal T}_0)$. Now, since $\varepsilon$ is a morphism in
$$
	\Hom_{\De(R)}(L,W)\cong \Hom_{\De(R)}(H^0(L)[0],H^{-2}(W)[2])\cong \Ext^2_R(H^0(L),H^{-2}(W)),
$$
then it is represented by an exact sequence
$$
	0\longrightarrow H^{-2}(W)\longrightarrow X_2\buildrel f\over\longrightarrow X_1\longrightarrow
		H^0(L)\longrightarrow0
$$
of $R\lMod$, in which $\mathop{\rm Cone}(f[1])\cong B$. By~(4.ii), we deduce that $H^0(L)\in\fp({\cal T}_0)$, i.e.\ $L\cong H^0(L)[0]\in\fp({\cal H})$ by Corollary~\ref{c:R-mod&T_0-are_in_fp(H)}(i).

It remains to treat condition~(5). Part~(5.i) has been proved in Corollary~\ref{c:lc-fp(C)_char}. On the other hand, for any $P\in\fp({\cal H})$ consider the approximation $0\to{\bf x}(P)\to P\to{\bf y}(P)\to0$ within the torsion pair $({}^{\bot_0}{\cal H}_{-1},{\cal H}_{-1})$. Its cohomology long exact sequence breaks up in the following exact rows of $R\lMod$:
\begin{gather*}
	0\longrightarrow H^{-2}({\bf x}(P))\longrightarrow H^{-2}(P)\longrightarrow 0 \cr
	0\longrightarrow H^{-1}({\bf x}(P))\longrightarrow H^{-1}(P)\longrightarrow
		H^{-1}({\bf y}(P))\longrightarrow0 \cr
	0\longrightarrow H^0(P)\longrightarrow H^0({\bf y}(P))\longrightarrow 0
\end{gather*}
where the only non-trivial fact is that $H^0({\bf x}(P))=0$, but this follows since the epimorphism ${\bf x}(P)\to H^0({\bf x}(P))[0]$ is zero by axiom of torsion pair. This said, we have $H^{-1}({\bf x}(P))\in{\cal T}_0$ since ${\bf x}(P)\in{}^{\bot_0}{\cal H}_{-1}$, and $H^{-1}({\bf y}(P))\in\TFT_{\!-1}\subseteq{\cal F}_0$. Therefore, by the second displayed exact row we deduce $H^{-1}({\bf x}(P))\cong x_0(H^{-1}(P))$. Moreover, by rotating the approximation of ${\bf x}(P)$ within the standard t-structure of $\De(R)$ we obtain the short exact sequence
$$
	0\longrightarrow x_0(H^{-1}(P))[0]\longrightarrow H^{-2}(P)[2]\longrightarrow
		{\bf x}(P)\longrightarrow0
$$
of $\cal H$. Now, bearing in mind part~(5.i), if $\cal H$ is locally coherent, then our condition~(5.ii) holds true by Corollary~\ref{c:R-mod&T_0-are_in_fp(H)}(i); conversely, if $x_0(H^{-1}(P))\in\fp({\cal T}_0)$, then ${\bf x}(P)\in\fp({\cal H})$ for being a cokernel of a morphism in $\fp({\cal H})$. \qedhere
\end{proof}

\section{Applications}%
\label{s:applications}%
We apply Corollary~\ref{c:recursive-length1} in the case of the HRS~heart ${\cal H}_{\boldsymbol{\tau}}$ associated with a hereditary torsion pair of finite $\boldsymbol{\tau}\leqdef({\cal T},{\cal F})$ in $R\lMod$; indeed, in Example~\ref{e:HRS-heart&filtration} we saw that ${\cal H}_{\boldsymbol{\tau}}$ can be realised as the AJS~heart associated with the Thomason filtration length~$1$
$$
	\Phi:\Spec R\supset Z\supset\emptyset,
$$
where $Z$ is the Thomason subset that corresponds to the torsion class $\cal T$.

Firstly, we see that for an arbitrary torsion pair $\boldsymbol{\tau}$ of $R\lMod$, one necessary condition for the local coherence of ${\cal H}_{\boldsymbol{\tau}}$ is that $\boldsymbol{\tau}$ must be hereditary of finite type. Notice that this follows by \cite[Proposition~2.6]{HS17} since the locally finite presentability of the heart is equivalent to ${\cal T}=\varinjlim\fp({\cal T})$ by \cite{PSV19}; however, we now achieve such result with a different argument.

\begin{prop}
Let\/ $\boldsymbol{\tau}\leqdef({\cal T},{\cal F})$ be any torsion pair in $R\lMod$. If the associated HRS~heart\/ ${\cal H}_{\boldsymbol{\tau}}$ is a locally finitely presented Grothendieck category, then $\boldsymbol{\tau}$ is hereditary (of finite type).%
\label{p:Ht_LC=>t_her_ft}%
\end{prop}
\begin{proof}
By~\cite{PS15} the torsion pair is necessarily of finite type; moreover, since ${\cal H}_{\boldsymbol{\tau}}$ is locally finitely presented, by \cite[Theorem~6.1]{PSV19} we have in particular ${\cal T}=\varinjlim({\cal T}\cap R\lmod)$. Therefore, ${\cal T}\cap R\lmod$ is a set (up to isomorphism), whose right orthogonal in $R\lMod$ coincides with $\cal F$, hence by \cite[Theorem~3.3]{BP18} $\boldsymbol{\tau}$ is a {\it tCG~torsion pair\/}; that is, its HRS~t-structure $({\cal U}_{\boldsymbol{\tau}},{\cal V}_{\boldsymbol{\tau}})$ in $\De(R)$ is compactly generated. Consequently, by \cite[Lemma~3.7]{Hrb20} there exists a Thomason filtration $\Phi$ such that $({\cal U}_{\boldsymbol{\tau}},{\cal V}_{\boldsymbol{\tau}})=({\cal U}_\Phi,{\cal V}_\Phi)$. We claim that ${\cal T}={\cal T}_{\Phi(0)}$, whence $\cal T$ turns out to be a hereditary torsion class. This readily follows thanks to the equality ${\cal U}_{\boldsymbol{\tau}}={\cal U}_\Phi$, namely by taking the 0th~cohomology of the stalk $X[0]$ for a module $X$ either in ${\cal T}$ or in ${\cal T}_{\Phi(0)}$ (see Remark~\ref{r:lc-TTF_finite_type}(3)).
\end{proof}

We recall that the converse of the previous result is known in the literature (see \cite[Theorem~2.2]{GP08} and \cite{Hrb20,PSV19,SS20}).

\begin{cor}
Let\/ $\boldsymbol{\tau}\leqdef({\cal T},{\cal F})$ be a torsion pair in\/ $R\lMod$, say with adjunctions
$$
	{\cal T} \coreflective[x] R\lMod \reflective[y]{\cal F} \;.
$$
The associated HRS~heart ${\cal H}_{\boldsymbol{\tau}}$ is a locally coherent Grothendieck category if and only if\/ $\boldsymbol{\tau}$ is hereditary of finite type and the following four conditions hold:%
\label{c:HRS_heart_LC_characterisation}%

\begin{enumerate}
\item[(i)] The torsion class $\cal T$ is locally coherent;

\item[(ii)] For every $B\in R\lmod$, the functor $\Hom_R(y(B),-)$ commutes with direct limits of direct systems in $\cal T$;

\item[(iii)] For all\/ $B\in R\lmod$, the functor $\Ext^1_R(y(B),-)$ commutes with direct limits of direct systems of\/ $\cal F$;

\item[(iv)] For every finitely generated ideal $J$ in the Gabriel filter associated with $\cal T$, the functor $\Ext^2_R(R/J,-)$ commutes with direct limits of\/ $\cal F$.
\end{enumerate}
\end{cor}
\begin{proof}
The necessity of the torsion pair being hereditary and of finite type has been proved in Proposition~\ref{p:Ht_LC=>t_her_ft}; this said, we shall prove the present Corollary by showing that the listed four conditions are equivalent to the corresponding ones of Corollary~\ref{c:recursive-length1}.

It is clear that our hypothesis~(i) is precisely Corollary~\ref{c:recursive-length1}(1). On the other hand, we have $\TFT_{\!-1}=\TF_{\!-1}={\cal F}_0$, thus $\fp(\TFT_{\!-1})=\mathop{\rm add}y(R\lmod)$ (see Remark~\ref{r:fp(TF_-1)}). The previous equality together with the additivity of the bifunctors $\Hom_R(-\mathbin,-)$ and $\Ext^1_R(-\mathbin,-)$ shows that also our hypotheses~(ii) and~(iii) are equivalent to the corresponding conditions of Corollary~\ref{c:recursive-length1}. Moreover, it is clear that Corollary~\ref{c:recursive-length1}(4) implies our condition~(iv). Let us prove that our hypotheses~(i) and~(iv) imply Corollary~\ref{c:recursive-length1}(4). Let $Q\in\fp({\cal T})$ and let $(Y_i)_{i\in I}$ be a direct systems of modules in $\cal F$. By Proposition~\ref{p:torsion-Thomason_subset} there exist a finitely generated ideal $J$ in the Gabriel filter of the torsion pair $({\cal T},{\cal F})$ and a short exact sequence $0\to X\to (R/J)^n\to Q\to0$, for some $n\in\N$ and $X$ a torsion module (we have $X\in\fp({\cal T})$ by~(i)). By applying the functors
$$
	L^k\leqdef \varinjlim_{i\in I}\Ext^k_R(-,Y_i)\quad\hbox{and}\quad
	\varGamma^k\leqdef \Ext^k_R(-,\varinjlim_{i\in I}Y_i)\qquad (k\ge1)
$$
on the above short exact sequence, we get the commutative diagram with exact rows
\[
\xymatrix@C=1.5em{%
	0 \ar[r] & L^1(Q) \ar[d]_{g_1}\ar[r] & L^1((R/J)^n) \ar[d]^{g_2}\ar[r] & L^1(X) \ar[d]^{g_3}\ar[r] &
		L^2(Q) \ar[d]^{g_4}\ar[r] & L^2((R/J)^n) \ar[d]^{g_5}\ar[r] & L^2(X) \ar[d]^{g_6} \cr
	0 \ar[r] & \varGamma^1(Q) \ar[r] & \varGamma^1((R/J)^n) \ar[r] & \varGamma^1(X) \ar[r] &
		\varGamma^2(Q) \ar[r] & \varGamma^2((R/J)^n) \ar[r] & \varGamma^2(X)
}
\]
in which $g_1,g_2,g_3$ are isomorphisms since $Q[0],(R/J)^n[0],X[0]\in\fp({\cal H})$ by Corollary~\ref{c:R-mod&T_0-are_in_fp(H)}(i), $g_5$ is iso by condition~(iv), while $g_6$ is a monomorphism by \cite[Proposition~1.6]{PSV19}, so that $g_4$ is iso as well. This concludes the proof.\qedhere
\end{proof}

\begin{rem}
A more general characterisation of the local coherence of the HRS~hearts has been achieved in \cite[Sec.~7]{PSV19} in the context of locally finitely presented Grothendieck categories.
\end{rem}

\begin{exmpl}
Let us exhibit an example of hereditary torsion pair of finite type whose HRS~heart is not locally coherent. Consider the non-coherent commutative ring $R\leqdef\mathbb{Z}\oplus(\mathbb{Z}/2\mathbb{Z})^{(\N)}$ introduced in Example~\ref{e:Vasconcelos}. For any nonzero tuple $a\in(\Z/2\Z)^{(\N)}$, the non~unitary element $e\leqdef(1,a)$ is idempotent, and the ideal $J\leqdef Re$ is idempotent as well.
Therefore, $J$ gives rise to a TTF~triple $({\cal E},{\cal T},{\cal F})$ in $R\lMod$ which is {\it split\/}; that is (see \cite[Proposition~VI.8.5]{Ste75}), in which ${\cal E}={\cal F}$ and both the torsion pairs $({\cal T},{\cal F})$ and $({\cal F},{\cal T})$ are hereditary. In particular, $({\cal T},{\cal F})$ is of finite type for being ${\cal F}=\ker\Hom_R(R/J,-)$; moreover it restricts to $R\lmod$ being split. We claim that it does not have a locally coherent HRS~heart. Assume, by contradiction, that such $\cal H$ is locally coherent. Since the ring $R$ is non-coherent, there exists an $R$-linear epimorphism $f\colon M\to N$ in $R\lmod$ such that $\ker f$ is not finitely presented. In the exact row
$$
	0\longrightarrow \ker x(f)\longrightarrow \ker f\buildrel a\over\longrightarrow
		\ker y(f)\longrightarrow\coker x(f)\longrightarrow0
$$
provided by the Snake~Lemma, we have $\ker y(f)\in\fp({\cal F})$ by conditions~(ii) and~(iii) of the previous Corollary. On the other hand, $\ker x(f),\coker x(f)\in\fp({\cal T})$ since the torsion pair is split, so that $x(M)$ and $x(N)$ are finitely presented objects of $\cal T$, which is locally coherent (as proved in the previous Corollary). By hypotheses~(i) and~(ii) we infer that $\im a\in R\lmod$, thus we get the contradiction $\ker f\in R\lmod$ by the extension-closure of the finitely presented modules.%
\label{e:torsion-pair-NO-coh-HRS}%
\end{exmpl}

\begin{exmpl}
Let us exhibit an example of a quasi locally coherent category in which the finitely presented objects do not form an abelian category. Consider the ordinary torsion pair $({\cal T},{\cal F})$ formed by the torsion and the torsionfree abelian groups. Since $\Z$ is noetherian, the torsion pair is of finite type and restricts to the finitely presented $\Z$-modules, hence by \cite[Theorem~5.2]{Sao17} its HRS~heart $\cal H$ is locally coherent. By Example~\ref{e:HRS-heart&filtration}, we know that there exists a proper Thomason subset $Z$ of $\Spec\Z$ such that $\cal H$ is the heart of the Thomason filtration $\Spec\Z\supset Z\supset\emptyset$. We then obtain a TTF~triple of finite type
$$
	({}^{\bot_0}_{\vphantom Z}{\cal H}_Z^{\vphantom{\bot_0}},{\cal H}_Z^{\vphantom{\bot_0}},{\cal H}_Z^{\bot_0})
$$
in which ${}^{\bot_0}{\cal H}_Z\;(\empty\cong{\cal F}[1])$ is a quasi locally coherent category by Theorem~\ref{t:lc-TTF_finite_type}. We claim that this torsion class is not locally coherent, i.e.~that the subcategory $\fp({}^{\bot_0}{\cal H}_Z)$ is not abelian. By contradiction, assume that this is the case and let $M\in{\cal T}$ be a finitely presented nonzero $\Z$-module, say it presented by the exact sequence $\Z^n\buildrel f\over\to\Z^m\to M\to0$. Now, $\Z^n[1]$ and $\Z^m[1]$ are finitely presented objects of $\cal H$ (\cite[Lemma~6.3]{SSV17}), therefore from the exact row
$$
	0\longrightarrow \ker^{({}^{\bot_0}{\cal H}_Z)}(f[1])\longrightarrow \Z^n[1]\buildrel f[1]\over\longrightarrow
		\Z^m[1]\longrightarrow \coker^{({\cal H})}(f[1])\longrightarrow0
$$
which lives in $\fp({}^{\bot_0}{\cal H}_Z)$ by quasi local coherence, we would have
$$
	\coker^{({\cal H})}\ker^{({}^{\bot_0}{\cal H}_Z)}(f[1]) \cong
		\ker^{({}^{\bot_0}{\cal H}_Z)}\coker^{({\cal H})}(f[1]) \;.
$$
On the other hand, the canonical short exact sequences in $\Ab$ given by the factorisation of $f$ through its image yield, once one takes the shifted stalk complexes, the following commutative diagram with exact rows in $\cal H$,
\[
\xymatrix{%
	&&& M[0] \ar@{>->}[d] \cr
	0 \ar[r] & (\ker f)[1] \ar@{>.>}[d]\ar[r] & \Z^n[1] \ar@{=}[d]\ar[r] & (\im f)[1] \ar@{->>}[d]\ar[r] & 0 \cr
	0 \ar[r] & \ker^{({\cal H})}(f[1]) \ar[r] & \Z^n[1] \ar[r]^{f[1]} & \Z^m[1] \ar@{.>}[r] & 0
}
\]
whence we see that $f[1]$ is an epimorphism, so that
\begin{align*}
	\ker^{({}^{\bot_0}{\cal H}_Z)}\coker^{({\cal H})}(f[1]) &\cong \Z^m[1], \cr
	\noalign{\hbox{and being}}
	\coker^{({\cal H})}\ker^{({}^{\bot_0}{\cal H}_Z)}(f[1]) &\cong (\im f)[1],
\end{align*}
we conclude $M[0]=0$, contradiction.%
\label{e:qlc-non-lc}%
\end{exmpl}

\subsection{When the ring is coherent}%
When the ring is coherent, our previous characterisations furtherly lighten, as we shall prove in Corollary~\ref{c:coherent_ring_LC_characterisation}. Let us start with an interesting example, which can be deduced formerly by \cite[Theorem~5.2]{Sao17}.

\begin{exmpl}
Let $R$ be a commutative coherent ring, and let $({\cal X},{\cal Y})$ be a torsion pair in the abelian category $R\lmod$; by \cite[p.~1666]{CB94} $(\varinjlim{\cal X},\varinjlim{\cal Y})\reqdef({\cal T},{\cal F})$ is a torsion pair (of finite type) in $R\lMod$. We claim that the associated HRS~heart in $\De(R)$ is a locally coherent Grothendieck category, namely by showing that the torsion pair is hereditary and satisfies the four conditions of the previous Corollary.

The torsion pair $({\cal T},{\cal F})$ is hereditary by the same argument of the proof of Proposition~\ref{p:Ht_LC=>t_her_ft}, namely for it is a $t$CG~torsion pair.

\noindent(i)\enspace Since $R\lMod$ is a locally coherent Grothendieck category, then $\cal T$ is so (see Remark~\ref{r:lc-TTF_finite_type}).

On the other hand, the torsion pair $({\cal T},{\cal F})$ restricts to $R\lmod$, so $y(B)$ is a finitely presented module for all $B\in R\lmod$, whence it is clear that conditions~(ii), (iii) and~(iv) of Corollary~\ref{c:HRS_heart_LC_characterisation} hold true, since $R$ is coherent.
\end{exmpl}

\begin{cor}
Let\/ $R$ be a commutative coherent ring and\/ $\boldsymbol{\tau}\leqdef({\cal T},{\cal F})$ be a torsion pair in $R\lMod$. Then the HRS~heart ${\cal H}_{\boldsymbol\tau}$ is a locally coherent Grothendieck category if and only if%
\label{c:coherent_ring_LC_characterisation}%

\begin{enumerate}
\item[(i)] The torsion pair $\boldsymbol{\tau}$ is hereditary of finite type;

\item[(ii)] For all\/ $B\in R\lmod$, the functor $\Hom_R(y(B),-)$ commutes with direct limits of direct systems in $\cal T$;

\item[(iii)] For all\/ $B\in R\lmod$, the functor $\Ext^1_R(y(B),-)$ commutes with direct limits of direct systems of\/ $\cal F$.
\end{enumerate}
\end{cor}

\begin{question}
Let $R$ be a commutative coherent ring and $({\cal T},{\cal F})$ be a torsion pair whose HRS~heart is a locally coherent Grothendieck category. Then does the torsion pair necessarily restrict to $R\lmod$?
\end{question}

\subsection{When the torsion pair is stable}%
We equip the torsion pairs of $R\lMod$ with a homological condition, i.e.~we consider the case of stable torsion pairs, so that even the torsion classes are closed under taking injective envelopes. As we shall see, such a homological condition translates into a finiteness one and, in particular, the necessary and sufficient conditions for the local coherence of the involved HRS~hearts furtherly simplify. In fact, our assumption is consistent and independent from the previous subsection, thanks to Example~\ref{e:torsion-pair-NO-coh-HRS}, which exhibits a non-trivial TTF~triple $({\cal F},{\cal T},{\cal F})$ over a non-coherent commutative ring, in which $({\cal T},{\cal F})$ is stable for $({\cal F},{\cal T})$ being hereditary. 

We need some auxiliary preliminary results, which in fact specialise the conditions of Corollary~\ref{c:HRS_heart_LC_characterisation} within the stability assumption.

\begin{lem}
If\/ $({\cal T},{\cal F})$ is a stable hereditary torsion pair of\/ $R\lMod$, then for every $X,Y\in{\cal T}$ we have $\Ext^k_R(X,Y)\cong\Ext^k_{\cal T}(X,Y)$, for all\/ $k\in\N\cup\{0\}$.%
\label{l:stable-auxiliary1}
\end{lem}
\begin{proof}
By the adjuction $\smash[b]{\jmath:{\cal T}\coreflective R\lMod:x}$ and by \cite[Proposition~2.28]{NS14}, we have the adjoint pair
$$
	{\bf L}\jmath:\De({\cal T})\longrightleftarrows\De(R):{\bf R}x
$$
of derived functors. In particular, for all $X,Y\in{\cal T}$, regarding the stalk of $X$ as an object of $\De({\cal T})$ and the stalk of $Y$ as a complex of $\De(R)$, being $x$ an exact functor by hereditariness, we have the natural isomorphism
\begin{align*}
	\Hom_{\De(R)}({\bf L}\jmath(X[0]),Y[n]) &\cong
		\Hom_{\De({\cal T})}(X[0],{\bf R}x(Y[n])) \cr
	&= \Hom_{\De({\cal T})}(X[0],x({\bf i}Y[n])) \cr
	&\cong \Hom_{\De({\cal T})}(X[0],{\bf i}Y[n]) \cr
	&\cong \Hom_{\De({\cal T})}(X[0],Y[n]),
\end{align*}
where ${\bf i}$ is the homotopically injective coresolution functor, computed equivalently either on $\De(R)$ or in $\De({\cal T})$, for $\cal T$ being a stable torsion class and an exact subcategory of $R\lMod$. By \cite{Ver}, the latter group of the display is isomorphic to $\Ext^n_{\cal T}(X,Y)$, so we claim that the first displayed group is isomorphic to $\Ext^n_R(X,Y)$. Indeed, we have
\begin{align*}
	\Hom_{\De(R)}({\bf L}\jmath(X[0]),Y[n]) &= \Hom_{\De(R)}(\jmath({\bf p}X[0]),Y[n]) \cr
	&\cong \Hom_{\De(R)}({\bf p}X[0],Y[n]) \cr
	&\cong \Hom_{\De(R)}(X[0],Y[n]),
\end{align*}
where ${\bf p}\colon\De(R)\to\Ho(R)$ is the homotopically projective resolution functor. \qedhere
\end{proof}

\begin{lem}
Let $({\cal T},{\cal F})$ be a stable torsion pair of $R\lMod$. Assume that conditions~$\rm(i)$ and~$\rm(iv)$ of Corollary~\ref{c:HRS_heart_LC_characterisation} hold true. Then, for every finitely generated ideal\/ $J$ in the Gabriel filter associated with $\cal T$, it is $R/J\in\mathop{\rm FP}_3(R)$, i.e.~the functors $\Ext^k_R(R/J,-)$ commute with direct limits for $k=0,1,2$.%
\label{l:stable-auxiliary2}%
\end{lem}
\begin{proof}
Let $(M_i)_{i\in I}$ be a direct system in $R\lMod$ and consider the direct system $(0\to X_i\to M_i\to Y_i\to0)_{i\in I}$ formed by the approximations of its members within $({\cal T},{\cal F})$. Since $R/J$ is a finitely presented torsion module (so that $R/J[0]$ is a finitely presented object of $\cal H$), by applying the functors
$$
	L^k\leqdef \varinjlim_{i\in I}\Ext^k_R(R/J,-)\quad\hbox{and}\quad
	\varGamma^k\leqdef \Ext^k_R(R/J,\varinjlim_{i\in I}(-))\qquad (k\in\N\cup\{0\})
$$
on the latter direct system, we find at once that $L^0(M_i)\cong\varGamma^0(M_i)$; moreover, in the following commutative diagram with exact rows,
\[
\xymatrix@C=1.5em{%
	L^1(X_i) \ar[d]_{f_1}\ar@{>->}[r] & L^1(M_i) \ar[d]^{f_2}\ar[r] & L^1(Y_i) \ar[d]^{f_3}\ar[r] &
		L^2(X_i) \ar[d]^{f_4}\ar[r] & L^2(M_i) \ar[d]^{f_5}\ar[r] &
			L^2(Y_i) \ar[d]^{f_6}\ar[r] & L^3(X_i) \ar[d]^{f_7} \cr
	\varGamma^1(X_i) \ar@{>->}[r] & \varGamma^1(M_i) \ar[r] & \varGamma^1(Y_i) \ar[r] &
		\varGamma^2(X_i) \ar[r] & \varGamma^2(M_i) \ar[r] & \varGamma^2(Y_i) \ar[r] & \varGamma^3(X_i)
}
\]
the canonical maps $f_1,f_4,f_7$ are isomorphisms by hypothesis~(i) and Lemma~\ref{l:stable-auxiliary1}, $f_6$ is an isomorphism by hypothesis~(iv), while $f_3$ is iso as well by \cite{Ver,BBD82} and since $R/J[0]\in\fp({\cal H})$. Therefore, by the Five~Lemma we deduce that $f_2$ and $f_5$ are isomorphisms, as desired. \qedhere
\end{proof}

\begin{rem}
By the proof of \cite[Lemma~2.14]{GT12} every indexing set $I$ is the union of a well-ordered chain of directed subposets $(I_\alpha\mid\alpha<\lambda)$, where each $I_\alpha$ has cardinality less than $I$. Moreover, for every direct system $(M_i)_{i\in I}$ of $R$-modules,
$(\smash[b]{\varinjlim\limits_{\hidewidth i\in I_\alpha\hidewidth}M_i}\mid\alpha<\lambda)$
is a well-ordered direct system satisfying%
\label{r:GT12-Lemma2.14}
$$
	\varinjlim_{i\in I}M_i=\varinjlim_{\alpha<\lambda}\varinjlim_{i\in I_\alpha}M_i \;.
$$
\end{rem}

\begin{lem}
Let\/ $({\cal T},{\cal F})$ be a stable torsion pair in $R\lMod$. Assume that condition~$\rm(ii)$ of Corollary~\ref{c:HRS_heart_LC_characterisation} holds true. Then, for every $B\in R\lmod$ and every direct system $(M_i)_{i\in I}$ in $\cal T$, the canonical homomorphism
$$
	\varinjlim_{i\in I}\Ext^1_R(y(B),M_i)\longrightarrow
	\Ext^1_R(y(B),\varinjlim_{i\in I}M_i)
$$
is injective.%
\label{l:stable-auxiliary3}%
\end{lem}
\begin{proof}
We formerly prove the statement in case $I$ is a well~ordered directed poset.

If $I$ is a finite set, there exist indices $\bar\imath,\bar\jmath\in I$ such that $\varinjlim_{i\in I}M_i=M_{\bar\imath}$ and $\varinjlim_{i\in I}\Ext^1_R(y(B),M_i)=\Ext^1_R(y(B),M_{\bar\jmath})$; moreover, there exists $k\ge\bar\imath,\bar\jmath$ making the displayed canonical map an isomorphism indeed.

If $I$ is infinite, by \cite[Lemma~3.5]{Hrb20} there exists a direct system $(0\to M_i\to E_i\to C_i\to0)_{i\in I}$ in which $E_i$ is the injective envelope of $M_i$, so that the direct system is in $\cal T$ by the stability hypothesis. Therefore, the canonical homomorphism displayed in the statement factors through the kernel of the map
$$
	\Ext^1_R(y(B),\varinjlim_{i\in I}M_i)\longrightarrow
	\Ext^1_R(y(B),\varinjlim_{i\in I}E_i)
$$
by means of an isomorphism, thanks to the Snake~Lemma and the assumption on $y(B)$ (similarly to the proof \cite[Proposition~1.6]{PSV19}). In other words, our statement is true for well~ordered directed posets.

This said, the general case follows as soon as we write $I=\bigcup_{\alpha<\lambda}I_\alpha$ as in Remark~\ref{r:GT12-Lemma2.14}; indeed, by the argument of the previous part (applied twice) and by $\rm AB\mathchar`-5$ condition of abelian groups, we obtain the following composition of monomorphisms
$$
	\varinjlim_{\alpha<\lambda}\varinjlim_{i\in I_\alpha}\Ext^1_R(y(B),M_i)\lhook\joinrel\longrightarrow
	\varinjlim_{\alpha<\lambda}\Ext^1_R(y(B),\!\varinjlim_{i\in I_\alpha}M_i)\lhook\joinrel\longrightarrow
	\Ext^1_R(y(B),\!\varinjlim_{\alpha<\lambda}\varinjlim_{i\in I_\alpha}M_i),
$$
which coincides with the natural map of the statement. \qedhere
\end{proof}

\begin{cor}
Let\/ $({\cal T},{\cal F})$ be a stable torsion pair in $R\lMod$. Then its HRS~heart $\cal H$ is a locally coherent Grothendieck category if and only if the torsion pair is of finite type and the following three conditions hold:

\begin{enumerate}
\item[(i)] $\fp({\cal T})\subseteq\mathop{\rm FP}_3(R)$;

\item[(ii)] $\fp({\cal F})\subseteq R\lmod$;

\item[(iii)] For all\/ $B\in R\lmod$, the functor $\Ext^1_R(y(B),-)$ commutes with direct limits of direct systems of\/ $\cal F$.
\end{enumerate}
\end{cor}
\begin{proof}
We shall prove that the stated conditions are equivalent to the ones of Corollary~\ref{c:HRS_heart_LC_characterisation}. Let us start by proving that our three hypotheses imply the conditions of the Corollary.

\noindent(i)\enspace By Proposition~\ref{p:torsion-lfp}, Proposition~\ref{p:Ht_LC=>t_her_ft}, and \cite{SS20}, $\cal T$ is a locally finitely presented Grothendieck category. It remains to show that $\fp({\cal T})$ is an exact abelian subcategory of $\cal T$, and this follows by condition~(i), namely for the kernel of any epimorphism in $\fp({\cal T})$ is finitely presented as well.

\noindent(ii)\enspace It follows immediately by our hypothesis~(ii).

\noindent(iv)\enspace It follows immediately by our hypothesis~(i).

Let us prove that the four conditions of Corollary~\ref{c:HRS_heart_LC_characterisation} imply our hypotheses~(i) and~(ii).

\noindent(i)\enspace Let $B\in\fp({\cal T})$; by Corollary~\ref{c:R-mod&T_0-are_in_fp(H)} there exist finitely generated ideals $J',J$ in the Gabriel filter associated with $\cal T$ and an exact row $(R/J')^n\buildrel\alpha\over\to (R/J)^m\to B\to0$ for some $n,m\in\N$. By Lemma~\ref{l:stable-auxiliary2}, $R/J$ and $R/J'$ are objects of $\mathop{\rm FP}_3(R)$, thus, being $\ker\alpha$ a finitely presented torsion module by~Corollary~\ref{c:HRS_heart_LC_characterisation}(i), in view e.g.~of \cite{BPz16} we infer that $\im\alpha\in\mathop{\rm FP}_2(R)$ and consequently that $B\in\mathop{\rm FP}_3(R)$. 

\noindent(ii)\enspace Since $\fp({\cal F})=\mathop{\rm add}y(R\lmod)$, we shall prove our assertion~(ii) on torsionfree modules of the form $y(B)$, where $B\in R\lmod$. Let $(M_i)_{i\in I}$ be a direct system in $R\lMod$ and consider its approximation $(0\to X_i\to M_i\to Y_i\to0)_{i\in I}$ within $({\cal T},{\cal F})$. By applying the functors
$$
	L^k\leqdef \varinjlim_{i\in I}\Ext^1_R(y(B),-)\quad\hbox{and}\quad
	\varGamma^k\leqdef \Ext^k_R(y(B),\varinjlim_{i\in I}(-))\qquad(k\in\N\cup\{0\})
$$
on the latter direct system, we obtain
\[
\xymatrix{%
	0 \ar[r] & L^0(X_i) \ar[d]_{g_1}\ar[r] & L^0(M_i) \ar[d]^{g_2}\ar[r] & L^0(Y_i) \ar[d]^{g_3}\ar[r] &
		L^1(X_i) \ar[d]^{g_4} \cr
	0 \ar[r] & \varGamma^0(X_i) \ar[r] & \varGamma^0(M_i) \ar[r] & \varGamma^0(Y_i) \ar[r] & \varGamma^1(X_i)
}
\]
in which $g_1$ is an isomorphism by Corollary~\ref{c:HRS_heart_LC_characterisation}(ii), $g_3$ is isomorphism since $y(B)\in\fp({\cal F})$, and $g_4$ is monic by Lemma~\ref{l:stable-auxiliary3}. By the Five~Lemma, we conclude that $y(B)$ is a finitely presented module.
\end{proof}

\subsection{On restrictable t-structures over commutative noetherian rings}%
Let $R$ be a commutative noetherian ring. Recall that $\De^b_{R\lmod}(R)$ is the triangulated subcategory of $\De(R)$ formed by the bounded complexes with finitely presented cohomologies; moreover, being $R$ noetherian, we have $\De^b_{R\lmod}(R)\cong\De^b_{\vphantom R}(R\lmod)$ canonically. A compactly generated t-structure $({\cal U},{\cal V})$ of $\De(R)$ is said to {\em restrict to $\De^b_{R\lmod}(R)$\/} if, for every complex $M\in\De^b_{R\lmod}(R)$, in the canonical triangle
$$
	\tau^\le_{\cal U}(M)\longrightarrow M\longrightarrow
		\tau^>_{\cal U}(M)\buildrel+\over\longrightarrow
$$
also the two outer vertices belong to $\De^b_{R\lmod}(R)$. In this case, the image of ${\cal U}\cap\De^b_{R\lmod}(R)$ under the aforementioned canonical equivalence is an aisle in $\De^b(R\lmod)$. Saor\'\i n proved in \cite[Theorem~6.3]{Sao17} that a compactly generated t-structure $({\cal U},{\cal V})$ of $\De(R)$ which restricts to $\De^b_{R\lmod}(R)$ has a locally coherent heart $\cal H$, with ${\cal H}\cap\De^b_{R\lmod}(R)$ as the abelian subcategory of its finitely presented objects.

We want to recover Saor\'\i n's result for Thomason filtration of finite length; that is, we want to check that in case the commutative ring $R$ is noetherian, given a Thomason filtration $\Phi$ of length $l+1$ whose (compactly generated) t-structure $({\cal U},{\cal V})$ restricts to $\De^b_{R\lmod}(R)$, then its heart is locally coherent and $\fp({\cal H})={\cal H}\cap \De^b_{R\lmod}(R)$.

\begin{lem}
Let\/ $R$ be a commutative noetherian ring. For every Thomason filtration $\Phi$ of\/ $\Spec R$, if\/ $X\in\fp({\cal T}_k)$, then $H_{\cal H}(X[-k])\in\fp({\cal H})$.
\end{lem}
\begin{proof}
Notice that it suffices to prove the statement for the heart ${\cal H}_k$, instead of $\cal H$, thanks to Theorem~\ref{t:H_*-TTF_finite_type}. This said, we conclude by Lemma~\ref{l:fp(H)&TFT} using the hypotheses on $R$ and $X$.
\end{proof}

\begin{lem}
Let\/ $R$ be a commutative noetherian ring, and let\/ $\Phi$ be a Thomason filtration whose t-structure $({\cal U},{\cal V})$ restricts to $\De^b_{R\lmod}(R)$. For all integers $k\in\Z$, if\/ $X\in{\cal T}_k$ is such that $H_{\cal H}(X[-k])\in\fp({\cal H})$, then\/ $H_{\cal H}(X[-k])\in\De^b_{R\lmod}(R)$.
\end{lem}
\begin{proof}
Let $(X_i)_{i\in I}$ be the canonical direct system of finitely generated submodules of $X$ in ${\cal T}_k$, such that $X=\varinjlim_{i\in I}X_i$. Since each torsion pair associated with $\Phi$ is of finite type, we have
$$
	y_{k+1}(X)\cong \varinjlim_{i\in I}y_{k+1}(X_i);
$$
and the direct system $(y_{k+1}(X_i))_{i\in I}$ lives in $\TF_{\!k}$, and $y_{k+1}(X)$ as well. Therefore, in view of Remark~\ref{r:functorial_monomorphism_sigma}, we get a commutative diagram with exact rows,
\[
\xymatrix{%
	0 \ar[r] & y_{k+1}(X_i) \ar[r]^-{\sigma_i}\ar[d] &
		\Sigma^{k}(y_{k+1}(X_i)) \ar[r]\ar[d] & C_i \ar[r]\ar[d] & 0 \cr
	0 \ar[r] & y_{k+1}(X) \ar[r]^-{\sigma_X} &
		\Sigma^{k}(y_{k+1}(X)) \ar[r] & C_X \ar[r] & 0
}
\]
where $C_X,C_i\in{\cal T}_{k+2}$, for all $i\in I$, and the vertical $R$-linear maps are the canonical ones. Moreover, we have 
$$
	\Sigma^{k}(y_{k+1}(X))\cong H^k(H_{\cal H}(X[-k])) \;.
$$
We claim that the vertical $R$-linear map of the previous diagram is an isomorphism. First, observe that the functor $H_{\cal H}\colon\De(R)\to{\cal H}$ commutes with coproduct (see \cite[Lemma~3.3]{AHMV17}). This said, consider the following diagram of triangles of $\De(R)$,
\[
\xymatrix{%
	&& \hidewidth(\ker f)[-k+1]\hidewidth \ar[d] \cr
	\coprod\limits_{i\le j}X_{ji}[-k] \ar[r]^-{f[-k]} & \coprod\limits_{i\in I}X_i[-k] \ar[r] &
		Z \ar[d]\ar[r]^-{+} & \cr
	&& \hidewidth X[-k]\hidewidth \ar[d]^-{+} \cr
	&& {}
}
\]
where $f\colon\bigoplus_{i\le j}X_{ji}\to\bigoplus_{i\in I}X_i$ is the colimit-defining homomorphism of $X$ in $R\lMod$. Applying $H_{\cal H}$ on the diagram, we obtain the following one with exact rows
\[
\xymatrix{%
	\bigoplus\limits_{i\le j}H_{\cal H}(X_{ji}[-k]) \ar[d]_-{\cong}\ar[r] &
		\bigoplus\limits_{i\in I}H_{\cal H}(X_i[-k]) \ar[d]^-{\cong}\ar[r] &
		\varinjlim\limits_{i\in I}\adjust H_{\cal H}(X_i[-k]) \ar@{.>}[d]^-{\cong}\ar[r] & 0 \cr
	H_{\cal H}(\coprod\limits_{i\le j}X_{ji}[-k]) \ar[r] &
		H_{\cal H}(\coprod\limits_{i\in I}X_i[-k]) \ar[r] &
		H_{\cal H}(Z) \ar[d]^-{\cong}\ar[r] & 0 \cr
	&& \hidewidth H_{\cal H}(X[-k])\hidewidth
}
\]
where the dotted morphism is given by the universal property of the cokernel, and the last vertical morphism is iso since $(\ker f)[-k+1]\in{\cal U}[1]$. The diagram show that $H_{\cal H}(X[-k])\cong \varinjlim_{i\in I}^{({\cal H})}H_{\cal H}(X_i[-k])$, and being the former a finitely presented object of $\cal H$, its identity morphism factors through $H_{\cal H}(X_{\bar\imath}[-k])$, for some index $\bar\imath\in I$. Consequently, the resulting $R$-linear map
$$
	H^k(H_{\cal H}(X_{\bar\imath}[-k]))\longrightarrow H^k(H_{\cal H}(X[-k]))
$$
is a split epimorphism, by additivity of the standard cohomology functors. Such homomorphism coincides with the relevant vertical $R$-linear map in the middle of the first commutative diagram of the proof. Eventually, from the following commutative diagram with canonical morphisms
\[
\xymatrix{%
	0 \ar[r] & x_{k+1}(X_{\bar\imath}) \ar@{>->}[d]\ar[r] & X_{\bar\imath} \ar@{>->}[d]\ar[r] &
		y_{k+1}(X_{\bar\imath}) \ar[d]\ar[r] & 0 \cr
	0 \ar[r] & x_{k+1}(X) \ar[r] & X \ar[r] & y_{k+1}(X) \ar[r] & 0
}
\]
by the Snake~Lemma we deduce that $y_{k+1}(X_{\bar\imath})\to y_{k+1}(X)$ is monic as well, so that, in the first commutative diagram, the left vertical map is injective too. Altogether, in such diagram, by the Snake~Lemma again, we infer that the middle vertical map is an isomorphism, since its kernel is a submodule of the kernel of $C_{\bar\imath}\to C_X$, which belongs to ${\cal T}_{k+2}\subseteq{\cal T}_{k+1}$. On the other hand, $y_{k+1}(X_{\bar\imath})[-k-1]$ belongs to $\De^b_{R\lmod}(R)$, and since the t-structure restricts by hypothesis, the associated triangle, call it
$$
	U\longrightarrow y_{k+1}(X_{\bar\imath})[-k-1]\longrightarrow V\buildrel+\over\longrightarrow,
$$
lives in $\De^b_{R\lmod}(R)$. Now, in view of the proof of Lemma~\ref{l:SigmaT}, we compute:
\begin{align*}
	H_{\cal H}(X[-k]) &\cong H_{\cal H}(y_{k+1}(X)[-k])\cong
		H_{\cal H}(\Sigma^{k}(y_{k+1}(X))[-k]) \cr
	&\cong H_{\cal H}(\Sigma^{k}(y_{k+1}(X_{\bar\imath}))[-k])\cong
		H_{\cal H}(y_{k+1}(X_{\bar\imath})[-k])=V[1],
\end{align*}
where the equality follows from the previous triangle. Since $V[1]$ belongs to $\De^b_{R\lmod}(R)$, we conclude. \qedhere
\end{proof}

\begin{prop}
For every complex $M\in \De^b_{R\lmod}(R)$, every direct system $(X_i)_{i\in I}$ in $R\lMod$, and integer $k\in\Z$, the canonical morphism
$$
	\varinjlim_{i\in I}\Hom_{\De(R)}(M,X_i[k])\longrightarrow
	\Hom_{\De(R)}(M,(\varinjlim_{i\in I}X_i)[k])
$$
is an isomorphism.
\end{prop}
\begin{proof}
Let us start by showing that the displayed homomorphism is epic. Consider the standard triangle $\tau^{\le-k-1}(M)\to M\to\tau^{>-k-1}(M)\buildrel+\over\to$ and apply the cohomological contravariant hom~functors of the stalks $X_i[k]$ and $(\varinjlim_{i\in I}X_i)[k]$, to get the commutative diagram with exact rows
\[
\xymatrix{%
	0 \ar[r] & \varinjlim\limits_{i\in I}\Hom_{\De(R)}(\tau^{>-k-1}(M),X_i[k]) \ar[d]\ar[r] &
		\varinjlim\limits_{i\in I}\Hom_{\De(R)}(M,X_i[k]) \ar[r]\ar[d] & 0 \cr
	0 \ar[r] & \Hom_{\De(R)}(\tau^{>-k-1}(M),(\varinjlim\limits_{i\in I}X_i)[k])) \ar[r] &
		\Hom_{\De(R)}(M,(\varinjlim\limits_{i\in I}X_i)[k]) \ar[r] & 0
}
\]
In order to conclude our claim, it suffices to check that the left vertical homomorphism is epic. Consider the triangle $H^{-k}(M)[k]\to \tau^{>-k-1}(M)\to \tau^{>-k-2}(M)\buildrel+\over\to$, which we rename $H\to V_1\to V_2\buildrel+\over\to$ for brevity. Applying the same functors of before, call them respectively $\varGamma^k_{\!i}$ and $\varGamma^k_{\vphantom i}$, we obtain the following commutative diagram with exact rows:
\[
\xymatrix@C1em{%
	\varinjlim\limits_{i\in I}\varGamma^k_{\!i}(V_2[-1]) \ar[r]\ar[d]_-{\mu_1} &
		\varinjlim\limits_{i\in I}\varGamma^k_{\!i}(H) \ar[r]\ar[d]_-{\mu_2} &
		\varinjlim\limits_{i\in I}\varGamma^k_{\!i}(V_1) \ar[r]\ar[d]^-{\mu_3} &
		\varinjlim\limits_{i\in I}\varGamma^k_{\!i}(V_2) \ar[r]\ar[d]^-{\mu_4} &
		\varinjlim\limits_{i\in I}\varGamma^k_{\!i}(H[1]) \ar[d]^-{\mu_5} \cr
	\varGamma^k_{\vphantom i}(V_2[-1]) \ar[r] &
		\varGamma^k_{\vphantom i}(H) \ar[r] &
		\varGamma^k_{\vphantom i}(V_1) \ar[r] &
		\varGamma^k_{\vphantom i}(V_2) \ar[r] &
		\varGamma^k_{\vphantom i}(H[1])
}
\]
where $\mu_2$ and $\mu_5$ are isomorphisms since $H^{-k}(M)$ is finitely presented. By the Snake~Lemma, we reduced to prove that $\mu_4$ is epic. Moreover, arguing inductively, bearing in mind that $M$ belongs to $\De^b_{R\lmod}(R)$, it suffices to check our claim for a truncation $\tau^{>j}(M)$, for some $j\le-k$, having just two nonzero cohomologies. In such case, we have the following standard triangle:
$$
	H^{j-1}(M)[-j+1]\longrightarrow \tau^{>j}(M)\longrightarrow H^{j-2}(M)[-j+2]\buildrel+\over\longrightarrow,
$$
say it $\tilde H\to \tilde V_1\to\tilde V_2\buildrel+\over\to$. Applying the functors $\varGamma^k_{\!i}$ and $\varGamma^k_{\vphantom i}$ as above, we obtain
\[
\xymatrix@C1em{%
	\varinjlim\limits_{i\in I}\varGamma^k_{\!i}(\tilde V_2[-1]) \ar[r]\ar[d]_-{\tilde\mu_1} &
		\varinjlim\limits_{i\in I}\varGamma^k_{\!i}(\tilde H) \ar[r]\ar[d]_-{\tilde\mu_2} &
		\varinjlim\limits_{i\in I}\varGamma^k_{\!i}(\tilde V_1) \ar[r]\ar[d]^-{\tilde\mu_3} &
		\varinjlim\limits_{i\in I}\varGamma^k_{\!i}(\tilde V_2) \ar[r]\ar[d]^-{\tilde\mu_4} &
		\varinjlim\limits_{i\in I}\varGamma^k_{\!i}(\tilde H[1]) \ar[d]^-{\tilde\mu_5} \cr
	\varGamma^k_{\vphantom i}(\tilde V_2[-1]) \ar[r] &
		\varGamma^k_{\vphantom i}(\tilde H) \ar[r] &
		\varGamma^k_{\vphantom i}(\tilde V_1) \ar[r] &
		\varGamma^k_{\vphantom i}(\tilde V_2) \ar[r] &
		\varGamma^k_{\vphantom i}(\tilde H[1])
}
\]
whose vertical homomorphisms but $\tilde\mu_3$ are bijective, since $H^{j-1}(M)$ and $H^{j-2}(M)$ are finitely presented. By the Five~Lemma, $\tilde\mu_3$ is bijective as well.

The injectivity of the stated homomorphism follows by the recursive argument of the previous part. Indeed, since $V_2[-1]$ belongs to $\De^b_{R\lmod}(R)$, the homomorphism~$\mu_1$ is surjective, so that $\mu_3$ is an isomorphism if and only if $\mu_4$ is so; this said, since in the last step of the previous argument we proved that $\tilde\mu_3$ is bijective, it follows that $\mu_4$ will be bijective as well.
\end{proof}

\begin{prop}
Let\/ $\Phi$ be a Thomason filtration of finite length. Then for every complex $M\in{\cal H}\cap \De^b_{R\lmod}(R)$ we have that $M\in\mathop{\rm FP}_2({\cal H})$; that is, the functors $\Hom_{\cal H}^{\vphantom1}(M,-)$ and\/ $\Ext^1_{\cal H}(M,-)$ preserve direct limits.%
\label{p:R_noetherian-finite-length-FP2}%
\end{prop}
\begin{proof}
Let $l+1$ be the length of $\Phi$. Let $(M_i)_{i\in I}$ be a direct system of $\cal H$. Recall that the standard cohomology functors preserve direct limits of $\cal H$, so that
$$
	H^k(\varinjlim_{i\in I}\adjust M_i)\cong\varinjlim_{i\in I}H^k(M_i)
$$
for all $k\in\Z$. Let us prove that, for all integers $k\in\Z$, that the canonical morphism
$$
	\varinjlim_{i\in I}\Hom_{\De(R)}(M,M_i[k])\longrightarrow
	\Hom_{\De(R)}(M,(\varinjlim_{i\in I}\adjust M_i)[k])
$$
is an isomorphism. Consider the following family of triangles:
\begin{gather*}
	\bigl(H^{-l-1}(M_i)[l+1+k]\longrightarrow M_i[k]\longrightarrow
		\tau^{>-l-1}(M_i)[k]\buildrel+\over\longrightarrow\bigr)_{i\in I} \cr
	\noalign{\hbox{and}}
	H^{-l-1}(\varinjlim_{i\in I}M_i)[l+1+k]\longrightarrow (\varinjlim_{i\in I}M_i)[k]\longrightarrow
		\tau^{>-l-1}(\varinjlim_{i\in I}M_i)[k]\buildrel+\over\longrightarrow
\end{gather*}
let us rename them by
\begin{gather*}
	\bigl(H_i[k]\longrightarrow M_i[k]\longrightarrow
		V_i[k]\buildrel+\over\longrightarrow\bigr)_{i\in I} \cr
	\noalign{\hbox{and}}
	H[k]\longrightarrow L[k]\longrightarrow
		V[k]\buildrel+\over\longrightarrow
\end{gather*}
for short. Apply the functor $\Delta\leqdef\Hom_{\De(R)}(M,-)$ and the direct limit functor on these triangles to get
\[
\xymatrix@C1em{%
	\cdots \ar[r] & \varinjlim\limits_{i\in I}\Delta(H_i[k]) \ar[d]\ar[r] &
		\varinjlim\limits_{i\in I}\Delta(M_i[k]) \ar[d]\ar[r] & \varinjlim\limits_{i\in I}\Delta(V_i[k]) \ar[d]\ar[r] & \varinjlim\limits_{i\in I}\Delta(H_i[k+1]) \ar[r]\ar[d] & \cdots \cr
	\cdots \ar[r] & \Delta(H[k]) \ar[r] & \Delta(L[k]) \ar[r] & \Delta(V[k]) \ar[r] & \Delta(H[k+1]) \ar[r] & \cdots
}
\]
By the previous Proposition, all the vertical homomorphisms in which the $H$'s and their shifting do appear are bijective. Our claim then reduces to prove that all the vertical morphisms involving the $V$'s are bijective. Arguing recursively, we have to show that the canonical homorphism
$$
	\varinjlim_{i\in I}\Hom_{\De(R)}(M,\tau^{>-1}(M_i)[k])\longrightarrow
	\Hom_{\De(R)}(M,\tau^{>-1}(\varinjlim_{i\in I}\adjust M_i)[k])
$$
is bijective, but this occurs, by the previous Proposition, since we have
\begin{align*}
	\tau^{>-1}(M_i)=H^0(M_i)[0] \qquad\hbox{and}\qquad
	\tau^{>-1}(\varinjlim_{i\in I}\adjust M_i) &\cong H^0(\varinjlim_{i\in I}\adjust M_i)[0] \cr
	&\cong (\varinjlim_{i\in I}H^0(M_i))[0] \;.
\end{align*}
Thus, we conclude by \cite{BBD82} letting $k=0,1$ in the second display of the present proof, bearing in mind that $M\in{\cal H}$. \qedhere
\end{proof}

We are now ready to recover the aforementioned result of Saor\'\i n for Thomason filtrations of finite length:

\begin{cor}[{\cite[Theorem~6.3]{Sao17}}]
Let\/ $R$ be a commutative noetherian ring, and\/ $\Phi$ be a Thomason filtration of finite length. If the associated t-structure restricts to $\De^b_{R\lmod}(R)$, then the heart $\cal H$ is locally coherent and\/ $\fp({\cal H})={\cal H}\cap\De^b_{R\lmod}(R)$.
\end{cor}
\begin{proof}
By \cite[Theorem~8.31]{SS20}, we know that $\cal H$ is a locally finitely presented Grothendieck category, with $\fp({\cal H})=\mathop{\rm add}H_{\cal H}({\cal U}\cap\De^c(R))$. By the hypothesis on $R$, it is clear that $\De^c_{\vphantom R}(R)\subseteq\De^b_{R\lmod}(R)$, and since the t-structure restricts to $\De^b_{R\lmod}(R)$, we infer that $\fp({\cal H})={\cal H}\cap\De^b_{R\lmod}(R)$. Now, by \cite[Proposition~1.13(2)]{PSV19} and in view of Proposition~\ref{p:R_noetherian-finite-length-FP2}, we obtain that $\cal H$ is locally coherent.
\end{proof}

Let us conclude this subsection with the following result, which shows that over a commutative noetherian ring any Thomason filtration of length~$1$ has locally coherent heart; we recall that Saor\'\i n exhibited in \cite[Remark~4.6]{Sao17} an example of a Thomason filtration of length~$1$ with locally coherent heart whose t-structure does not restrict to $\De^b_{R\lmod}(R)$.

\begin{cor}
Let\/ $R$ be a commutative noetherian ring. If\/ $\Phi$ is a Thomason filtration of length~$0$ or~$1$, then its heart $\cal H$ is locally coherent.
\end{cor}
\begin{proof}
The length~$0$ case has been proved in Corollary~\ref{c:coherent_ring=>torsion_LC_Grothendieck}. For the length~$1$ case we will show that the four conditions of Corollary~\ref{c:recursive-length1} hold true. In our hypotheses, only conditions~(2) and~(3) of such Corollary are not trivially satisfied. However, these conditions follow by the following inclusion
$$
	\fp(\TFT_{\!-1})\subseteq\fp({\cal T}_{-1}),
$$	
which we now prove. Indeed, let $B\in\fp(\TFT_{\!-1})$ and $(B_i)_{i\in I}$ be the canonical direct system in $R\lmod$ such that $B=\varinjlim_{i\in I}B_i$. Notice that the direct system lives in $\TFT_{\!-1}={\cal T}_{-1}\cap{\cal F}_0$, hence by hypothesis the identity of $B$ factors through $B_{\bar\imath}$, for some index $\bar\imath\in I$, whence $B$ is finitely presented in ${\cal T}_{-1}$. \qedhere
\end{proof}

\end{document}